\documentclass{amsart}
\usepackage{fullpage}
\usepackage{amsfonts,amssymb,amsmath,amsthm,amscd, lipsum}
\usepackage{tikz}
\usepackage{url}
\usepackage{enumerate}
\usepackage{fancyhdr, hyperref}
\UseRawInputEncoding
\usepackage[utf8]{inputenc}

\definecolor{lavender}{rgb}{0.75,0.58,0.89}

\hypersetup{
    colorlinks=true,
    linkcolor=lavender,
    citecolor=lavender,
    urlcolor=lavender
}

\makeatletter
\AtBeginDocument{
  \hypersetup{linktoc=none} 
}
\makeatother

\newtheorem{theorem}{Theorem}[section]
\newtheorem{lemma}[theorem]{Lemma}
\newtheorem{proposition}[theorem]{Proposition}
\newtheorem{corollary}[theorem]{Corollary}

\newtheorem{definition}[theorem]{Definition}
\newtheorem{hypothesis}[theorem]{Hypothesis}
\theoremstyle{definition}

\DeclareFontEncoding{OT2}{}{}
\newcommand{\textcyr}[1]{{\fontencoding{OT2}\fontfamily{wncyr}\fontseries{m}\fontshape{n}\selectfont #1}}
\newcommand{\Sha}{{\mbox{\textcyr{Sh}}}}

\newenvironment{remark}[1][Remark]{\begin{trivlist}
\item[\hskip \labelsep {\bfseries #1}]}{\end{trivlist}}

\begin{document}

\title[$L$-functions of elliptic curves over real ring class fields]{$L$-functions of elliptic curves in ring class extensions of real quadratic fields via regularized theta liftings}

\author{Jeanine Van Order }
\email{vanorder@puc-rio.br}
\address{Departamento de Matem\'atica, Pontif\'icia Universidade Cat\'olica do Rio de Janeiro (PUC-Rio)}

\begin{abstract} We derive new integral presentations for central derivative values of $L$-functions of elliptic curves $E/{\bf{Q}}$ twisted by ring class characters of a real quadratic field $K$ 
in terms of automorphic Green's functions for certain Hirzebruch-Zagier-like arithmetic divisors on the product of modular curves $X_0(N) \times X_0(N)$ along real geodesic cycles. 
We also relate these sums to Birch-Swinnerton-Dyer constants and periods. \end{abstract}

\maketitle

\tableofcontents

\section{Introduction}

Let $E$ be an elliptic curve of conductor $N$ defined over the rational number field ${\bf{Q}}$, 
with corresponding Hasse-Weil $L$-function denoted by $L(E, s)$.
The modularity theorem of Wiles, Taylor-Wiles, and Breuil-Conrad-Diamond-Taylor implies that 
$L(E, s)$ has an analytic continuation $\Lambda(E, s)$ via the Mellin transform 
\begin{align}\label{modularity-id} \Lambda (E, s) &= \Lambda(s-1/2, f) := \int_{0}^{\infty} f \left( \frac{iy}{\sqrt{N}} \right) y^s \frac{dy}{y} 
= N^{\frac{s}{2}} (2 \pi)^{-s} \Gamma(s) L(s-1/2, f) \end{align}
of some weight-two newform 
\begin{align*} f(\tau) = f_E(\tau) = \sum\limits_{n \geq 1} c_f(n) e(n \tau) = 
\sum\limits_{n \geq 1} a_f(n) n^{\frac{1}{2}} e(n \tau) \in S_2^{\operatorname{new}}(\Gamma_0(N)) \end{align*}
with $L$-function corresponding to the Mellin transform (first for $\Re(s) >1$)
\begin{align*} L(s, f) := \sum\limits_{n \geq 1} a_f(n) n^{-s} = \sum\limits_{n \geq 1} c_f(n) n^{-(s + 1/2)}. \end{align*}
Here, we use the unitary normalization for the automorphic $L$-functions $\Lambda(s, f)$, so that $s=1/2$ is the central value. 
We use the arithmetic normalization for the Hasse-Weil $L$-function $\Lambda(E,s)$, so that $s=1$ is the central value.
That is, writing $\pi = \otimes_v \pi_v$ to denote the cuspidal automorphic representation of $\operatorname{GL}_2({\bf{A}})$ associated to $f$,
with $\Lambda(s, \pi) = \prod_{v \leq \infty} L(s, \pi_v)$ its standard $L$-function, we have equivalences of $L$-functions 
\begin{align*} \Lambda(E, s) &= \Lambda(s -1/2, f) = \Lambda(s - 1/2, \pi). \end{align*}

Let $k$ be any number field. The Mordell-Weil theorem implies that the group of $k$-rational points $E(k)$ has the structure
of a finitely generated abelian group $E(k) \cong {\bf{Z}}^{r_E(k)} \oplus E(k)_{\operatorname{tors}}$.
It is a fundamental open problem to characterize the rank $r_E(k) = \operatorname{rk}_{\bf{Z}} E(k)$. 
Writing $L(E/k, s)$ to denote the Hasse-Weil $L$-function of $E/k$, Birch and Swinnerton-Dyer conjectured
that this generating series $L(E/k, s)$, defined a priori only for $\Re(s) > 3/2$, has an analytic continuation $\Lambda(E/k,s)$ to all $s \in {\bf{C}}$, 
with $\Lambda(E/k, s)$ satisfying a functional equation relating values at $s$ to $2-s$ (so that $s=1$ is the central point). 
Taking for granted this preliminary hypothesis\footnote{which remains open in general}, 
the conjecture of Birch and Swinnerton-Dyer predicts that the rank $r_E(k)$ is given by the 
order of vanishing $\operatorname{ord}_{s=1}\Lambda(E/k, s)$ at this central point. Although 
this conjecture has been verified over the past several decades for $r_E(k) \leq 1$ with $k = {\bf{Q}}$ or $k$ an imaginary quadratic field, 
it remains open at large, without a single known example for $r_E(k) \geq 2$. 
The most stunning progress to date has come through the Iwasawa theory of elliptic curves, 
using as a starting point special value formulae for the values $\Lambda^{(r_E(k))}(E/k, 1)$. 
In particular, the celebrated theorem of Gross-Zagier \cite{GZ} 
(with generalizations such as \cite{YZZ} and \cite{BY}) for the central derivative value $\Lambda'(E/k, \chi, 1)$, 
with $\chi$ a class group character of an imaginary quadratic field $k$, has played a major role underlying most of this progress for rank one. 
This work makes essential use of the theory of complex multiplication and explicit class field theory for imaginary quadratic fields,
and especially a construction of points $e_H \in E(k[1])$ dating back to Heegner to relate the central derivative values $\Lambda'(E/k, \chi, 1)$ 
for $\chi$ a character of the class group $\operatorname{Pic}(\mathcal{O}_k) \cong \operatorname{Gal}(k[1]/k)$ (with $k[1]/k$ the Hilbert class field) 
to the regulator term $R_E(k) = [e_H, e_H]$ (with $[\cdot, \cdot]$ the N\'eron-Tate height pairing). 

Here, we return to the largely unexplored setting of $k  = K = {\bf{Q}}(\sqrt{d})$ a real quadratic field of discriminant 
\begin{align*} d_K  &= \begin{cases} d &\text{ if $d \equiv 1 \bmod 4$} \\ 4 d &\text{ if $d \equiv 2, 3 \bmod 4$} \end{cases} \end{align*} 
prime to $N$, and corresponding even Dirichlet character $\eta = \eta_{K/{\bf{Q}}}$. 
Let $\chi$ be any ring class character of $K$ of conductor $c \in {\bf{Z}}_{\geq 1}$ prime to $d_K N$. 
Hence, we view $\chi$ a character of the corresponding ring class group $\operatorname{Pic}(\mathcal{O}_c) \cong \operatorname{Gal}(K[c]/K)$ 
of the ${\bf{Z}}$-order $\mathcal{O}_c := {\bf{Z}} + c \mathcal{O}_K$ of conductor $c$ in $K$, 
\begin{align*} \chi: \operatorname{Pic}(\mathcal{O}_c) 
:= {\bf{A}}_K^{\times} / K_{\infty}^{\times} K^{\times} \widehat{\mathcal{O}}_c^{\times} &\longrightarrow {\bf{S}}^1,
\quad \widehat{\mathcal{O}}_c^{\times} := \prod_{v < \infty} \mathcal{O}_{c, v}^{\times}. \end{align*}
Via $(\ref{modularity-id})$, the theories of Rankin-Selberg convolution and quadratic basechange imply that the Hasse-Weil $L$-function
$L(E/K, \chi, s)$ has an analytic continuation $\Lambda(E/K, \chi, s)$ to all $s \in {\bf{C}}$ via a functional equation relating values at $s$ to $2-s$.
Writing $\pi(\chi)$ to denote the automorphic representation of $\operatorname{GL}_2({\bf{A}})$ of level $d_K c^2$ and character $\eta$ 
induced from the ring class character $\chi$, this completed $L$-function $\Lambda(E/K, \chi, s)$ is 
equivalent to the corresponding shifted $\operatorname{GL}_2({\bf{A}}) \times \operatorname{GL}_2({\bf{A}})$ Rankin-Selberg $L$-function 
$\Lambda(s-1/2, \pi \times \pi(\chi))$. Writing $\Pi = \operatorname{BC}_{K/{\bf{Q}}}(\pi)$ to denote the quadratic basechange lifting of $\pi$ to a 
cuspidal automorphic representation of $\operatorname{GL}_2({\bf{A}}_K)$, the $L$-function $\Lambda(E/K, \chi, s)$ is also equivalent
to the shifted $\operatorname{GL}_2({\bf{A}}_K) \times \operatorname{GL}_1({\bf{A}}_K)$ automorphic $L$-function $\Lambda(s-1/2, \Pi \otimes \chi)$.  
Hence, we see the analytic continuation through the equivalent presentations 
\begin{align*} \Lambda(E/K, \chi, s) &= \Lambda(s-1/2, \pi \times \pi(\chi)) = \Lambda(s-1/2, \Pi \otimes \chi). \end{align*}
As explained in $(\ref{symmFE})$ below, each $\Lambda(E/K, \chi, s)$ satisfies a symmetric functional equation.
This gives the following immediate consequence, whose proof we explain in the discussion leading to Hypothesis \ref{EHH} below:

\begin{lemma}\label{key} 

Let $E$ be an elliptic curve of conductor $N$ defined over ${\bf{Q}}$, and $\pi = \pi(f)$ the cuspidal automorphic representation of
$\operatorname{GL}_2({\bf{A}})$ associated to the eigenform $f \in S_2^{\operatorname{new}}(\Gamma_0(N))$ parametrizing $E$.
Let $K$ be a real quadratic field of discriminant $d_K$ prime to $N$, with $\eta (\cdot) = \eta_K (\cdot) = \left( \frac{d_K}{\cdot} \right)$
the corresponding Dirichlet character. Hence, we can write $N = N^+ N^-$ for $N^+$ the product of prime divisors 
$q \mid N$ which split in $K$, and $N^-$ the product of prime divisors $q \mid N$ which remain inert in $K$, and $\eta(-N) = \eta(N) = \eta(N^-)$.
If $N^-$ is the squarefree product of an odd number of primes, then we have the vanishing of the central value
\begin{align*} \Lambda(E/K, \chi, 1) = \Lambda(1/2, \pi \times \pi(\chi)) = \Lambda(1/2, \Pi \otimes \chi) = 0 \end{align*}
for any ring class character $\chi$ of $K$ of conductor $c$ prime to $d_K N$. \end{lemma}

\noindent \textbf{Informal description of the main result.} In this forced vanishing setting, we show that the central derivative
$\Lambda'(E/K,\chi,1) = \Lambda'(1/2, \pi \times \pi(\chi)) = \Lambda'(1/2, \Pi \otimes \chi)$ can be expressed in terms of a $\chi$-twisted sum of
certain regularized theta lifts evaluated along real geodesic cycles on $Y_0(N)\times Y_0(N)$ attached to the class group of $K$.
These regularized theta lifts are automorphic Green's functions for Hirzebruch-Zagier divisors on the product of modular curves $Y_0(N) \times Y_0(N)$,
and these geodesic cycles play the role of real quadratic analogues of the CM cycles appearing in the Gross-Zagier formula. 
A comparison with Euler characteristic calculations then gives us new insights into the refined conjecture of Birch and Swinnerton-Dyer for $E(K)$.

\subsection{Description of the main result}

In the setup of forced vanishing described for Lemma \ref{key}, we derive novel integral presentations for  
$\Lambda^{\prime}(E/K, \chi, 1) = \Lambda'(1/2, \pi \times \pi(\chi)) = \Lambda'(1/2, \Pi \otimes \chi)$, given in terms of sums
of automorphic Green's functions of arithmetic Hirzebruch-Zagier divisors on $Y_0(N) \times Y_0(N)$ evaluated 
along geodesic cycles associated with ring class groups of $K$.
To do this, we adapt and develop the calculation of Bruinier-Yang \cite[Theorem 4.7]{BY}, 
related to their distinct proof of the Gross-Zagier formula \cite[$\S 7$]{BY}, cf.~\cite{GZ} and \cite{YZZ}. 
This allows us to show an analogue of the Gross-Zagier formula for real quadratic fields.
While there is no known global analogue of the Heegner point construction in this setting, 
we present some depiction of the provenance of potential points $E(K[c])$ via 
Euler characteristic calculations used to verify special cases of the refined conjecture of Birch and Swinnerton-Dyer.

Fix a primitive ring class character $\chi$ of $K$ of conductor $c$ prime to $d_KN$ (which we shall assume exists).
For each class $A \in \operatorname{Pic}(\mathcal{O}_c)$, we fix an integral representative
$\mathfrak{a} \subset \mathcal{O}_K$ so that $A = [\mathfrak{a}] \in \operatorname{Pic}(\mathcal{O}_c)$,
and write $Q_{\mathfrak{a}}(z) := {\bf{N}}_{K/{\bf{Q}}}(z)/{\bf{N}} \mathfrak{a}$ to denote the corresponding norm form of signature $(1,1)$.
Here, ${\bf{N}}_{K/{\bf{Q}}}(z) = z z^{\tau}$ denotes the norm homomorphism, with $\tau \in \operatorname{Gal}(K/{\bf{Q}})$ the nontrivial automorphism. 
We consider the quadratic space $(V_A, q_A)$ of signature $(2,2)$ defined by
\begin{align*}V_A &= \mathfrak{a}_{\bf{Q}} \oplus \mathfrak{a}_{\bf{Q}}, \quad
Q_A(z) = Q_A((z_1, z_2)) := Q_{\mathfrak{a}}(z_1) - Q_{\mathfrak{a}}(z_2). \end{align*} 
We consider the corresponding general spin group $\operatorname{GSpin}(V_A)$. We have an exceptional isomorphism 
\begin{align*} \operatorname{GSpin}(V_A) \cong \operatorname{GL}_2 \times_{{\bf{G}}_m} \operatorname{GL}_2
= \left\lbrace (g_1, g_2) \in \operatorname{GL}_2 \times \operatorname{GL}_2: \det(g_1) = \det(g_2) \right\rbrace \end{align*}
of algebraic groups over ${\bf{Q}}$ (see Proposition \ref{spinID}). Consider the Grassmannian 
\begin{align*} D(V_A) &= \left\lbrace z \subset V_A({\bf{R}}): \operatorname{dim}(z)=2, Q_A\vert_z <0 \right\rbrace \end{align*}
of oriented negative $2$-planes in $V_A(\bf{R})$. 
Note that $D(V_A)$ has two connected components $D^{\pm}(V_A)$ corresponding to the choice of orientation.
We shall fix one of these $D^{\pm}(V_A) \cong \mathfrak{H}^{\pm} = \mathfrak{H}^+ \coprod \mathfrak{H}^{-}$ throughout. 
For any compact open subgroup $U_A \subset \operatorname{GSpin}(V_A)({\bf{A}}_f)$, we then
consider the corresponding spin Shimura variety $X_A = \operatorname{Sh}(D(V_A), \operatorname{GSpin}(V_A))$ with complex points 
\begin{align*} X_A({\bf{C}}) &= \operatorname{GSpin}(V_A)({\bf{Q}}) \backslash \left( D(V_A) \times \operatorname{GSpin}(V_A) ( {\bf{A}}_f ) / U_A \right). \end{align*}
This $X_A$ determines a non-compact surface defined over {\bf{Q}}. 
Via the exceptional isomorphism $\operatorname{GSpin}(V_A) \cong  \operatorname{GL}_2 \times_{{\bf{G}}_m} \operatorname{GL}_2$, 
we can take $U_A$ to be the compact open subgroup of $\operatorname{GSpin}(V_A)({\bf{A}}_f)$ 
corresponding to the two-fold product of congruence subgroup $\Gamma_0(N)$ (see $(\ref{level})$), and we have 
\begin{align*} X_A({\bf{C}}) &\cong Y_0(N) \times Y_0(N). \end{align*}
The $X_A \cong Y_0(N)^2$ come equipped with arithmetic divisors. To describe them, define for each $m \in {\bf{Q}}_{>0}$  
\begin{align*} \Omega_{m, A}({\bf{Q}}) = \left\lbrace x \in V_A: Q_A(x) =m \right\rbrace. \end{align*}
Consider the natural projection 
$\operatorname{pr}: D(V_A) \times \operatorname{GSpin}(V_A)({\bf{A}}_f) \rightarrow X_A$.
Given $x \in V_A({\bf{R}})$, we have the orthogonal projection 
$D(V_A)_x = \left\lbrace z \in D(V_A): z \perp x \right\rbrace$.
Let $L_A \subset V_A$ denote the integral lattice stabilized by the compact open subgroup 
$U_A \subset \operatorname{GSpin}(V_A)$, with $L_A^{\vee}$ its dual lattice,
and $L_A^{\vee}/L_A$ the corresponding discriminant group.
We define for each coset $\mu \in L_A^{\vee}/L_A$ the divisor 
\begin{align*} Z_A(\mu, m) &= \sum\limits_{x \in \left( \operatorname{GSpin}(V_A)({\bf{Q}}) \cap U_A \right) 
\backslash \Omega_{A, m}({\bf{Q}})} {\bf{1}}_{\mu}(x) \operatorname{pr}(D(V_A)_x). \end{align*} 
Sums over cosets $\mu \in L_A^{\vee}/L_A$ of these special divisors can be related to classical Hirzebruch-Zagier divisors. 
As we explain below, these divisors are arithmetic in the sense of Arakelov theory -- they come equipped with 
explicit Green's functions $\Phi = G_{Z_A(\mu, m)}$. We evaluate these functions along the following real geodesic cycles. 
Let $V_{A, 2} \subset V_A$ denote the anisotropic subspace of signature 
$(1, 1)$ given by the integer ideal $\mathfrak{a}$ with its norm form:
\begin{align*} (V_{A, 2}, Q_{A, 2}), \quad V_{A, 2} := \mathfrak{a}_{\bf{Q}} = \mathfrak{a} \otimes {\bf{Q}}, 
\quad Q_{A,2}(\lambda) = Q_{\mathfrak{a}}(\lambda) = \frac{ {\bf{N}}(\lambda) }{ {\bf{N}} \mathfrak{a} } = \frac{\lambda \lambda^{\tau}}{ {\bf{N}} \mathfrak{a}  }. \end{align*}
Each such subspace $(V_{A, 2}, Q_{A, 2})$ gives rise to a set of oriented geodesic lines 
\begin{align*} D(V_{A, 2}) = \left\lbrace z \in V_{A, 2}({\bf{R}}) : \dim(z) = 1,~~~ Q_{A,2} \vert_z <0 \right\rbrace \end{align*}
Here, we have two connected components $D^{\pm}(V_{A, 2})$ corresponding to the orientation of a line $z$ in 
$V_{A, 2}({\bf{R}}) = \mathfrak{a}_{\bf{Q}} \otimes {\bf{R}} $. 
Each component $D^{\pm}(V_{A, 2})$ determines an open subset of 
real projective space \footnote{In the projective model $D(V_A) \cong \lbrace [x] \in {\bf{P}}(V_{A,2}({\bf{R}})) : Q_{A,2}(x) <0 \rbrace $, 
a pair of boundary isotropic lines $z^{\pm} \subset V_{A,2}({\bf{R}})$ determines a projective line ${\bf{P}}(\lbrace z^+, z^- \rbrace)$,
where $\lbrace z^+, z^- \rbrace$ denotes the span of $z^+$ and $z^-$. The geodesic defined by the intersection of 
${\bf{P}}(\lbrace z^+, z^- \rbrace)$ with $D(V_{A,2}) \cong \lbrace [x] \in {\bf{P}}(V_{A,2}({\bf{R}})) : Q_{A,2}(x) <0 \rbrace$
is a one-dimensional real manifold, so a real curve of dimension one. In the setting we consider, this corresponds to a real 
geodesic in the upper-half plane $\mathfrak{H}$, embedded into $X_A$ via a modular curve. 
Via the identifications $\operatorname{GSpin}(V_A) \cong  \operatorname{GL}_2 \times_{{\bf{G}}_m} \operatorname{GL}_2$ and 
$X_A \cong Y_0(N) \times Y_0(N)$, each pair of lines $z^{\pm} $ determines 
a real geodesic on $\mathfrak{H}$, and hence on any modular curve $C$. In this way,
through embeddings of modular curves into the surface $Y_0(N) \times Y_0(N)$ (e.g.~via Hirzebruch-Zagier divisors), 
we can realize these geodesics on $X_A \cong Y_0(N) \times Y_0(N)$.} of dimension one with a fixed orientation,
\begin{align*} D^{\pm}(V_{A, 2}) &= \left\lbrace z^{\pm} = [x:y] \in {\bf{P}}^1({\bf{R}}), ~\text{orientation $\pm$}: Q_{A, 2}(x, y) < 0 \right\rbrace. \end{align*}
We consider for each class $A \in \operatorname{Pic}(\mathcal{O}_c)$ the finite ``geodesic set''
\begin{align*} \mathfrak{G}(V_{A, 2}) &= \operatorname{GSpin}(V_{A, 2})({\bf{Q}}) \backslash
\operatorname{GSpin}(V_{A, 2})({\bf{A}}_f) / U_{A,2}, \quad U_{A,2} :=U_A \cap \operatorname{GSpin}(V_{A, 2})({\bf{A}}_f). \end{align*}
Fixing a representative $h \in \operatorname{GSpin}(V_{A,2})({\bf{A}}_f)$ for each class $[h] \in \mathfrak{G}(V_{A,2})$, 
and writing $\operatorname{GSpin}(V_{A, 2})({\bf{R}})^0$ for the connected component of the identity in $\operatorname{GSpin}(V_{A, 2})({\bf{R}})$, we consider the symmetric space defined by
\begin{align*} C_{A, h} &= \Gamma_{A, h} \backslash D^{\pm}(V_{A,2}), 
\quad \Gamma_{A, h} := \operatorname{GSpin}(V_{A,2})({\bf{Q}}) \cap \operatorname{GSpin}(V_{A, 2})({\bf{R}})^0 h U_{A,2} h^{-1}. \end{align*}
We consider integrals 
\begin{align*} \int\limits_{C_{A,h}} \Phi(z, h) d \nu (z) \end{align*}
over this symmetric space $C_{A,h}$, where $d \nu$ denotes the $\operatorname{O}(1,1)$-invariant length measure, and the sums
\begin{align*} \Phi(\mathcal{G}_A) &:= \sum\limits_{ h \in \mathfrak{G}(V_{A,2}) \atop h \in \operatorname{GSpin}(V_{A,2}({\bf{A}}_f))} \frac{1}{\# \operatorname{Aut}(h)} 
\int\limits_{C_{A, h}} G(z, h) d \nu (z), \end{align*}
over the corresponding real geodesic cycles
\begin{align}\label{G_A} \mathcal{G}_A &:= \coprod_{[h] \in \mathfrak{G}(V_{A,2}) \atop h \in \operatorname{GSpin}(V_{A,2})({\bf{A}}_f)} C_{A,h}
= \coprod_{[h] \in \mathfrak{G}(V_{A,2}) \atop h \in \operatorname{GSpin}(V_{A,2})({\bf{A}}_f)} \Gamma_{A, h} \backslash D^{\pm}(V_{A,2}). \end{align}

Let $\omega_{L_A}$ denote the Weil representation of $\operatorname{SL}_2({\bf{A}})$ 
on the space of Schwartz functions $\mathcal{S}(V \otimes {\bf{A}})$ determined by the quadratic module $(L_A, Q_A)$. 
Let $H_l(\omega_{L_A})$ denote the space of harmonic weak Maass forms of weight $l \in \frac{1}{2}{\bf{Z}}$ and representation $\omega_{L_A}$ in the sense of 
Bruinier-Funke \cite{BF}. Each $f_l \in H_l(\omega_{L_A})$ decomposes as a sum $f_l(\tau) = f_l^+(\tau) + f_l^-(\tau)$ of a holomorphic part 
\begin{align*} f_{l}^+(\tau) &= \sum\limits_{ \mu \in L_A^{\vee}/L_A} f^+_{l, \mu} (\tau) {\bf{1}}_{\mu} =  \sum\limits_{ \mu \in L_A^{\vee}/L_A} 
\left( \sum\limits_{m \in {\bf{Q}} \atop m \gg - \infty} c^{+}_{f_{l}}(\mu, m) e(m \tau)\right)  {\bf{1}}_{\mu}, \quad {\bf{1}}_{\mu} = \operatorname{char}(\mu + L_A). \end{align*} 
and a nonholomorphic part $f_l^-(\tau)$, and fits into the exact sequence
\begin{align*} 0 \to M_l^!(\omega_{L_A}) \to H_l(\omega_{L_A}) \overset{\xi_l}{\longrightarrow} S_{2-l}(\overline{\omega}_{L_A}) \to 0, \end{align*}
where $\xi_l$ is the Bruinier-Funke differential operator
\begin{align*} \xi_l: H_l(\omega_L) \longrightarrow S_{2-l}(\overline{\omega}_L), \quad f(\tau) \longmapsto v^{l-2} \overline{L_l f(\tau)}, 
\quad L_l := - 2 i v^2 \cdot \frac{\partial }{ \partial \overline{\tau}}. \end{align*}
Here, $M^!_{l}(\omega_{L_A}) \subset H_{l}(\omega_{L_A})$ denotes the subspace of weakly holomorphic forms, 
with $M_{l}(\omega_{L_A}) \subset M^!_{l}(\omega_{L_A}) $ the subspace of holomorphic forms, 
and $S_{l}(\omega_{L_A}) \subset M_{l}(\omega_{L_A})$ the subspace of holomorphic cuspidal forms; $L_l$ denotes the standard Maass weight-lowering operator. 
We have a natural inner product 
\begin{align*} \langle \langle f, g \rangle \rangle &= \sum\limits_{  \mu \in L_A^{\vee}/L_A } f_{\mu}(\tau) g_{\mu}(\tau) \end{align*} 
defined on vector-valued forms
\begin{align*} f(\tau) &= \sum\limits_{  \mu \in L_A^{\vee}/L_A } f_{\mu}(\tau) {\bf{1}}_{\mu} \in A_{l}(\omega_{L_A})
\quad \text{~~~and~~~} \quad g(\tau) \sum\limits_{  \mu \in L_A^{\vee}/L_A } g_{\mu}(\tau) {\bf{1}}_{\mu} \in A_{-l}(\overline{\omega}_{L_A}).\end{align*} 
Here, we write $\overline{\omega}_{L_A}$ to denote the Weil representation of the quadratic module $(L_A, -Q_A)$. 
Writing \begin{align*} \mathcal{F} = \lbrace \tau = u + iv \in \mathfrak{H}: \vert u \vert \leq 1/2, u^2 + v^2 \geq 1 \rbrace \end{align*} to denote the standard 
fundamental domain for $\operatorname{SL}_2({\bf{Z}})$ acting on $\mathfrak{H}$ by fractional linear transformation, 
we define the corresponding Petersson inner product (when it converges) by 
\begin{align*} \langle f, g \rangle &= \int\limits_{\mathcal{F}} \langle \langle f(\tau), \overline{g(\tau)} \rangle \rangle v^l d \mu(\tau), 
\quad d \mu(\tau) = \frac{du dv}{v^2}. \end{align*}

Let $\theta_{L_A}(\tau, z, h)$ denote the Siegel theta series defined on $\tau \in \mathfrak{H}$, $z \in D(V_A)$, and $h \in \operatorname{GSpin}(V_A)({\bf{A}}_f)$.
As a function in $\tau$, this is a nonholomorphic form of weight $0$ and contragredient Weil representation $\overline{\omega}_{L_A}$.
Given $f_0 \in H_{0}(\omega_{L_A})$, we consider the corresponding regularized theta lift 
\begin{align*} \Phi(f_0, z, h)  &= \int_{\operatorname{SL}_2({\bf{Z}}) \backslash \mathfrak{H}}^{\star} \langle \langle f_0(\tau), \theta_{L_A} (\tau, z, h) \rangle \rangle \frac{du dv}{v^2}\\
&:= \operatorname{CT}_{s=0}  \left(  \lim_{T \rightarrow \infty} \int_{\mathcal{F}_T} \langle \langle f_0(\tau), \theta_{L_A} (\tau, z, h) \rangle \rangle v^{-s} \frac{du dv}{v^2} \right) \end{align*} 
given by the constant term in the Laurent series expansion around $s=0$ of 
\begin{align*} \lim_{T \rightarrow \infty} \int_{\mathcal{F}_T} \langle \langle f_0(\tau), \theta_{L_A} (\tau, z, h) \rangle \rangle v^{-s} \frac{du dv}{v^2},
\quad \mathcal{F}_T = \left\lbrace \tau = u + iv \in \mathcal{F}: v \leq T \right\rbrace. \end{align*}
A theorem of Bruinier \cite{BrB} extending Borcherds \cite{Bo} allows us to view these regularized theta lifts
$\Phi(f_0, \cdot)$ as automorphic Green's functions in the sense of Arakelov theory. 
To be more precise, if the Fourier coefficients $c_{f_0}^+(\mu, m)$ of the holomorphic part $f_0^+$ of $f_0$ are integers, we define the divisor
\begin{align}\label{Z_A} Z_A(f_{0}) &= \sum\limits_{ \mu \in L_A^{\vee} / L_A } \sum\limits_{m \in {\bf{Q}} \atop m >0} c_{f_{0}}^+(\mu, -m) Z_A(\mu, m). \end{align}
The regularized theta lift $\Phi(f_0, \cdot)$ is the automorphic Green's function $G_{Z_A(f_0)}$ for this divisor 
$Z_A(f_0) \subset X_A$ (Theorem \ref{Green}), giving an arithmetic divisor $\widehat{Z}_A(f_0) = (Z_A(f_0), G_{Z_A(f_0)})$.

For each class $A \in \operatorname{Pic}(\mathcal{O}_c)$, we take $f_{0, A} \in H_{0}(\omega_{L_A})$ to be the harmonic weak Maass 
whose image $g = g_{f, A} = \xi_0(f_{0, A}) \in S_{2}(\overline{\omega}_{L_A})$ 
under the operator $\xi_0 : H_{0}(\omega_{L_A}) \rightarrow S_{2}(\overline{\omega}_{L_A})$ 
has a canonical lift as described in Theorem \ref{YZhang} to the scalar-valued eigenform $f \in S_2^{\operatorname{new}}(\Gamma_0(N))$. 
Each of the vector-valued cusp forms  $g_{f, A}$ has a Fourier series expansion given explicitly in terms of the 
Fourier coefficients of the eigenform $f \in S_2^{\operatorname{new}}(\Gamma_0(N))$. 
That is, we have for each class $A = [\mathfrak{a}] \in \operatorname{Pic}(\mathcal{O}_c)$ the relation 
\begin{align*} g_{f, A}(\tau) &= \sum\limits_{\mu \in L_A^{\vee}/L_A} g_{f, A, \mu}(\tau) {\bf{1}}_{\mu}
= \sum\limits_{\mu \in L_A^{\vee}/ L_A} \left( \sum\limits_{m \in {\bf{Q}}_{>0} \atop m \equiv N Q_A(\mu) \bmod N} 
c_f(m) s(m) e \left( \frac{m \tau}{N} \right) \right) {\bf{1}}_{\mu}. \end{align*}
Here, we write $s$ to denote the function defined on classes $m \bmod N$ by $s(m) = 2^{\Omega(m, N)}$, 
where $\Omega(m, N)$ is the number of prime divisors of the greatest common divisor $(m, N)$ of $m$ and $N$. 
Our main results, Theorem \ref{MAIN} and Corollary \ref{DCNF}, 
allow us to express the central derivative value $\Lambda'(1/2, \Pi \otimes \chi)$ as a $\chi$-twisted linear combination 
the Green's functions $G_{Z(f_{0, A})}$ evaluated along the real geodesic cycles $\mathcal{G}_A$ defined in $(\ref{G_A})$ above.

To describe this, we first explain how to decompose the theta series $\theta_{L_{A}}(\tau, z, h)$ for our main calculation.
Consider the anisotropic subspaces $V_{A, 1} := \mathfrak{a}_{\bf{Q}}$ with $Q_{A, 1}(z) = - Q_{\mathfrak{a}}(z)$ 
and $V_{A, 2} = \mathfrak{a}_{\bf{Q}}$ with $Q_{A, 2}(\lambda) = Q_{\mathfrak{a}}(z)$ of signature $(1,1)$.
We consider for each $j = 1,2$ the sublattice $L_{A,j} := L_A \cap V_{A, j}$, and the corresponding Siegel theta series 
$\theta_{L_{A,j}}(\tau, z, h)$ of weight $0$ and representation $\omega_{L_{A,j}}$. 
Since we evaluate at elements $z_A \in D^{\pm}(V_{A,2})$ and $h \in \operatorname{GSpin}(V_{A, 2})({\bf{A}}_f)$, 
we can replace the Siegel theta series $\theta_{L_A}(\tau, z_A, h)$ with the product of specializations 
$\theta_{L_{A,1}}(\tau) \otimes \theta_{L_{A,2}}(\tau, z_A, h) = \theta_{L_{A,1}}(\tau, 1, 1) \otimes \theta_{L_{A,2}}(\tau, z_A, h)$.
The Siegel-Weil theorem (Theorem \ref{SW-abstract} and Theorem \ref{SW-vector}) allows us to interpret  the sum
\begin{align*} \int\limits_{ \operatorname{SO} (V_{A, 2})({\bf{Q}}) \backslash \operatorname{SO}(V_{A, 2})({\bf{A}})  } \theta_{L_{A,2}}(\tau, z_A, h) dh \end{align*} 
as the value at $s=0$ of a vector-valued Eisenstein series $E_{L_{A,2}}(\tau, s; 0)$ of weight $0$. 
Following the approach of Kudla \cite{KuBL}, we interpret this Eisenstein series as the image 
under the Maass weight-lowering operator $L_2$ of the derivative $E'_{L_{A,2}}(\tau, 0; 2)$ of an incoherent Eisenstein series $E_{L_{A,2}}(\tau, 0; 2)$ of weight $2$. 
We describe this in more detail below; see Propositions \ref{MSEis}, \ref{incoherent}, and \ref{vanish}.
Let $\mathcal{E}_{L_{A,2}}(\tau)$ denote the holomorphic part of $E'_{L_{A,2}}(\tau, 0; 2)$. 
We consider its pairing with the holomorphic part $f_{0,A}^+(\tau)$, 
\begin{equation}\begin{aligned}\label{CC} \operatorname{CT}  \langle\langle f_{0, A}^+(\tau), {\bf{1}}_{L_{A,1} \oplus 0} \otimes \mathcal{E}_{L_{A,2}}(\tau) \rangle \rangle.
\end{aligned}\end{equation} Note that $(\ref{CC})$ is an algebraic number. 
We also consider the holomorphic ``shadow" theta series of weight $2$ defined by 
$\xi_0(\theta_{L_{A,1}})(\tau)$ together with the corresponding regularized Rankin-Selberg integral 
\begin{equation}\begin{aligned}\label{regint}  I(s, f_{0, A} \times \xi_0(\theta_{L_{A,1}})) &:=
\int_{\mathcal{F}}^{\star} \langle \langle f_{0, A}(\tau), \overline{\xi_0(\theta_{L_{A,1}})(\tau)} \otimes E_{L_{A,2}}(\tau, s; 2) \rangle \rangle v^2 d \mu (\tau) \\
&= \lim_{T \rightarrow \infty} \int\limits_{\mathcal{F}_T}  
\langle \langle f_{0, A}(\tau), \overline{\xi_0(\theta_{L_{A,1}})(\tau)} \otimes E_{L_{A,2}}(\tau, s; 2) \rangle \rangle v^2 d \mu(\tau). \end{aligned}\end{equation}
Although the growth\footnote{roughly $c_{f_{0,A}}^+(\mu, m) \approx \exp(C \sqrt{m})$ times a power of $m$ for some constant $C>0$}  
of the Fourier coefficients of the holomorphic part $f_{0,A}^+(\tau)$ of $f_{0,A}$ force the underlying Dirichlet series to diverge, 
we can think of $I(s, f_{0,A} \times \xi_0(\theta_{L_{A,1}}))$ as a Rankin-Selberg product due to the appearance of the Eisenstein series $E_{L_{A,2}}(\tau, s; 2)$, 
from which  $I(s, f_{0,A} \times \xi_0(\theta_{L_{A,1}}))$ inherits an analytic continuation as a function of $s \in {\bf{C}}$ (see $(\ref{IFE})$). 
For each class $A \in \operatorname{Pic}(\mathcal{O}_c)$, we compute the sum
\begin{align}\label{sumG_A} \Phi(f_{0, A}, \mathcal{G}_A ) &= \sum\limits_{  [h] \in \mathfrak{G}(V_{A,2}) \atop h \in \operatorname{GSpin}(V_{A,2})({\bf{A}}_f) }
\frac{1}{\# \operatorname{Aut}(h)} \int\limits_{  C_{A,h} \cong \Gamma_{A,h} \backslash D^{\pm}(V_{A,2}) } \Phi(f_{0, A}, z, h) d \nu (z) \end{align}
along the real geodesic cycle $\mathcal{G}_A$ defined in $(\ref{G_A})$ above. Note that this amounts to summing the Green's function
$G_{Z_A(f_{0,A})}(\cdot) = \Phi(f_{0, A}, \cdot)$ for the divisor $Z_A(f_{0,A}) \subset X_A \cong Y_0(N) \times Y_0(N)$ along the real geodesic cycle $\mathcal{G}_A$.
Let $h_K$ denote the class number of $K$, and $\varepsilon_K$ the fundamental unit, 
so that $\varepsilon_K = \frac{1}{2}(t + u \sqrt{d_K})$ is the least integral solution (with $u$ minimal) to Pell's equation $t^2 - d_K u^2 = 4$.

\begin{theorem}[Theorem \ref{4.7}, Theorem \ref{MAIN}]\label{sumform}

In the setup described above, we have the integral presentation
\begin{align*} &\Lambda'(1/2, \Pi \otimes \chi) = \Lambda'(E/K, \chi, 1) \\ &= - 2 h_K \log \varepsilon_K 
\sum\limits_{ A \in \operatorname{Pic}(\mathcal{O}_c) \atop A = [\mathfrak{a}] } \chi(A)
 \left( \operatorname{CT} \langle \langle f^+_{0, A}, {\bf{1}}_{L_{A,1} \oplus 0} \otimes \mathcal{E}_{L_{A,2}} \rangle \rangle
+ I'(0, f_{0, A} \times \xi_0(\theta_{L_{A,1}}))+\frac{\operatorname{vol}(U_{A, 2})}{2} \cdot \Phi(f_{0, A}, \mathcal{G}_A) \right). \end{align*} \end{theorem}

If we assume that the inert level $N^-$ is given by the squarefree product of an odd number of primes (Hypothesis \ref{EHH}) , 
then $L(1/2, \Pi \otimes \chi) = 0$ by $(\ref{symmFE})$, and $\Lambda'(1/2, \Pi \otimes \chi)$ is not forced to vanish. 
Here, Theorem \ref{MAIN} should be viewed as a real quadratic analogue of the formula of Gross-Zagier \cite[Theorem $\S$I (6.3)]{GZ}.
In contrast to the imaginary quadratic setting, the real quadratic case admits no CM points, and the relevant geometric objects are non-compact geodesic cycles on 
the product of modular curves $Y_0(N)\times Y_0(N)$. As such, our argument does not follow from the arithmetic intersection framework of Kudla's programme, nor
from the calculations of Bruinier-Yang \cite{BY} and Andreatta-Goren-Howard-Madapusi Pera \cite{AGHMP} of Green's functions along CM cycles of spin Shimura varieties. 
Instead, the integral presentation we derive arises from the evaluation of Green's functions along geodesic cycles parametrized by 
Lorentzian quadratic spaces, with the corresponding Siegel-Weil identity interpreted via Maass weight-raising operators. 
The Fourier coefficients of the weight-zero Maass forms $f_{0,A}$ govern the resulting arithmetic,
as reflected in the appearance of the regularized inner product $I(s, f_{0, A} \times \xi_0 \theta_{L_{A,1}})$. 
To our knowledge, no such formula relating geodesic cycles to central derivative values has previously appeared in the literature.

Our Theorem \ref{MAIN} may be viewed as a derivative analogue of Popa's theorem \cite[$\S$ 1, Theorem 6.3.1]{Po}
for central values $\Lambda(1/2, \Pi \otimes \chi)$ in the setting with root number $\eta(-N) = \eta(N) = + 1$.
This develops Waldspurger's theorem \cite{Wa} to give an exact toric period formula for these central values, 
and generalizes the formula of Gross \cite{Gr} for the analogous setup with an imaginary quadratic field. 
Waldspurger's theorem \cite{Wa} equates the nonvanishing of the central value 
$\Lambda(1/2, \pi \times \pi(\chi))$ with that of the period integral \begin{align*}\int_{ {\bf{A}}_K^{\times}/K^{\times} } \varphi(t) \chi(t) dt, \end{align*}
for $\varphi \in \pi^{\operatorname{JL}}$ a vector in the Jacquet-Langlands lift $\pi^{\operatorname{JL}}$ of $\pi$ to an indefinite
quaternion algebra $B$ over ${\bf{Q}}$ with ramification given by the inert level: $\operatorname{Ram}(B) = \lbrace q \mid N^- \rbrace$. 
Popa \cite{Po} gives an exact and classical formula for $L(1/2, \pi \times \pi(\chi))$ as such as toric integral, 
which can be viewed as a twisted sum over geodesic cycles on the modular curve $X_0(N)$ parametrizing $E$. 
The main novelty of this work is the realization of central derivative values $\Lambda'(1/2, \Pi \otimes \chi) = \Lambda'(E/K, \chi, 1)$ 
in the real quadratic setting as geometric invariants associated to real geodesic cycles on the product of modular curves $Y_0(N)\times Y_0(N)$, 
suggesting the following boundary interpretation inside higher-dimensional spin Shimura varieties.

\subsubsection{A geometric interpretation} 

We offer the following interpretation of the real geodesic cycles $\mathcal{G}_A$ in our discussion above.
We can identify the Grassmannian $D(V_{A, 2}) \cong \lbrace z = [x:y] \in {\bf{P}}^1({\bf{R}}): Q_{A, 2}(x, y) <0 \rbrace$ of oriented lines
with the symmetric space $D(\operatorname{GSpin}(1,1))$ of $\operatorname{GSpin}(1,1) \cong {\bf{G}}_m \times \operatorname{SO}(1,1)$. 
On the other hand, we can consider the symplectic group $\operatorname{GSp}_4(W)$ acting on a four-dimensional symplectic space $W$. 
The Siegel parabolic $P = \lbrace g \in \operatorname{GSp}_4(W): gL =L \rbrace$ of $\operatorname{GSp}_4(W)$ stabilizing a
(maximal isotropic) two-dimensional Lagrangian subspace $L \subset W$ has Levi subgroup 
$M_P \cong {\bf{G}}_m \times \operatorname{GL}_2$. Viewing $\operatorname{GL}_2$ as an extension of $\operatorname{SO}(1,1)$ via the inclusion 
\begin{align*} \operatorname{SO}(1,1) \subset \operatorname{GSpin}(1,1) \cong {\bf{G}}_m \times {\bf{G}}_m &\longrightarrow \operatorname{GL}_2, 
\quad (t_1, t_2) \longmapsto \left( \begin{array}{cc} t_1 & ~\\ ~& t_2 \end{array} \right), \end{align*}
we obtain an embedding of $D(V_{A,2})$ into the corresponding symmetric space $D(M_P)$ for $M_P$. 
In this way, we can realize each real geodesic cycle $\mathcal{G}_A$ inside a component of the boundary of the Borel-Serre compactification of a $\operatorname{GSp}_4(W)$ Shimura variety. 
More formally, let $(\mathfrak{V}_{A,2}, \mathfrak{Q}_{A,2})$ be any rational quadratic space of signature $(3, 2)$ into which $(V_{A, 2}, q_{A,2})$ embeds.
Consider the corresponding spin group $\operatorname{GSpin}(\mathfrak{V}_{A, 2})$ and Grassmannian $D(\mathfrak{V}_{A, 2})$. 
Let $\mathfrak{L}_{A, 2} \subset \mathfrak{V}_{A, 2}$ be any lattice for which $\mathfrak{L}_{A, 2} \cap V_{2, A} = L_{A,2} = \mathfrak{a}$, 
and let $\mathfrak{U}_{A, 2}$ denote the corresponding compact open subgroup of $\operatorname{GSpin}(\mathfrak{V}_{A, 2})({\bf{A}}_f)$ fixing this lattice.
The spin Shimura variety $\mathfrak{X}_{A, 2}$ with complex points 
\begin{align*} \mathfrak{X}_{A, 2}({\bf{C}}) &=\operatorname{GSpin}(\mathfrak{V}_{A, 2})({\bf{Q}}) \backslash D(\mathfrak{V}_{A, 2}) 
\times \operatorname{GSpin}(\mathfrak{V}_{A, 2})({\bf{A}}_f)/ \mathfrak{U}_{A, 2} \end{align*}
defines a quasiprojective variety of dimension $3$ over ${\bf{Q}}$. 
It can be identified as a twisted or quaternionic Siegel threefold via the accidental isomorphism $\operatorname{GSpin}(3,2) \cong \operatorname{GSp}_4(W)$,
defined over any number field over which the quadratic space splits (see \cite[Appendix A]{KuRa}). 
Hence, $D(V_{A, 2})$ can be realized as a component in the boundary $\partial \mathfrak{X}_{A, 2}^{\operatorname{BS}}$ 
of the Borel-Serre compactification $\mathfrak{X}_{A, 2}^{\operatorname{BS}}$ of $\mathfrak{X}_{A,2}$. 
Via Theorem \ref{sumform}, this suggests that a study of the boundaries of Borel-Serre compactifications of twisted Siegel threefolds  
might shed light on the provenance of ``Stark-Heegner" points in $X_0(N)(K[c]) \rightarrow E(K[c])$. 
This observation also allows us to interpret our main formula in terms of $\partial \mathfrak{X}_{A, 2}^{\operatorname{BS}}$ for any such threefold $\mathfrak{X}_{A,2}$. 
We note that the strategy of realizing locally symmetric spaces for $\operatorname{GL}_n$ in the boundaries of Borel-Serre compactifications of symplectic or unitary Shimura varieties, 
which seems to go back to Clozel (cf.~\cite{Cl20}), is used crucially in the constructions by Scholze \cite{Sch15}, Harris-Lan-Taylor-Thorne \cite{HLTT}, 
and Allen-Calegari-Caraiani-Gee-Helm-Le Hung-Newton-Scholze-Taylor-Thorne
\cite{ACC} of Galois representations associated to cuspidal $\operatorname{GL}_n$-automorphic representations. 

\subsubsection{Other remarks}\label{remark} 

\noindent{(i).} The regularized theta lifts $\Phi(f_{0, A}, \cdot) = G_{Z(f_{0, A})(\cdot)}$ can be related to the theta lifts 
constructed by Kudla-Millson in \cite{KM} by the arguments of Bruinier-Funke \cite[Theorems 1.4 and 1.5]{BF}. \\

\noindent{(ii).} The role played by the holomorphic projection in \cite{GZ} is replaced here by the holomorphic part 
$\mathcal{E}_{L_{A,2}}(s, \tau)$ of the derivative Eisenstein series $E'_{L_{A,2}}(s, \tau; 2)$ appearing in the integral presentation of $L'(0, \xi_0(f_{0, A}) \times \theta_{L_{A,1}})$. \\

\noindent{(iii).} Recall that a complex number is a period if its real and imaginary parts can be expressed as integrals of rational functions with rational coefficients, 
over domains in ${\bf{R}}^n$ given by polynomial inequalities with rational coefficients. 
We expect the values $\Lambda'(E/K, \chi, 1)$ to be periods (cf.~\cite[Question 4]{KZ}). 
We can deduce this in the special cases, as described in Corollary \ref{URBSD} 
via the argument\footnote{We remark that the deduction, not given explicitly in \cite[$\S 3.5$]{KZ}, is to use the formulae of Gross-Zagier \cite{GZ} and 
Gross-Kohnen-Zagier \cite{GKZ} to verify that $L'(E, 1) = \alpha \cdot R \cdot \Omega$, where
$\alpha$ denotes some nonzero rational number, $R = R_E({\bf{Q}}) = \langle e, e \rangle$ the regulator
(given by the arithmetic height of a Heegner divisor on the modular curve $X_0(N)$), and $\Omega = \Omega_E({\bf{Q}})$ the real period.
Assuming the finiteness of the Tate-Shafarevich group $\Sha_E({\bf{Q}})$, the argument of \cite[$\S$ 3.5]{KZ} shows that the Birch-Swinnerton-Dyer constant 
$\kappa_E({\bf{Q}}) := (R_E({\bf{Q}}) \cdot T_E({\bf{Q}}) \cdot \Sha_E({\bf{Q}}) \cdot  \Omega_E({\bf{Q}}))/ \# E({\bf{Q}})^2$ is a period.
Hence, the deduction consists of first relating $L'(E, 1)$ to $\kappa_E({\bf{Q}})$ via the Gross-Zagier formula, then using 
that $\kappa_E({\bf{Q}})$ is a period to deduce that $L'(E,1)$ must be a period. There does not seem to be any direct proof 
in the literature that the central derivative value $L'(E,1)$ is a period.} given in Kontsevich-Zagier \cite[$\S 4$]{KZ}. 

\subsection{Applications towards Birch-Swinnerton-Dyer}

Theorem \ref{MAIN} also suggests a possible origin of points in the $K[c]$-rational Mordell-Weil groups $E(K[c])$ in 
via embeddings of Hirzebruch-Zagier divisors into spin Shimura varieties. 
In this spirit, we also describe how the refined Birch and Swinnerton-Dyer conjecture 
suggests new characterizations of the Tate-Shafarevich group $\Sha_E(K[c])$ 
and regulator term $R_E(K[c])$. We refer to $(\ref{REG})$, $(\ref{SHA})$, 
and below for more details of what can be deduced here. One consequence is the following. 

\begin{corollary}[Theorem \ref{URBSD}] 

Assume $(N, d_K)= 1$, and that the inert level $N^-$ is given by the squarefree
product of an odd number of primes (Hypothesis \ref{EHH}), so that $L(1/2, \Pi \otimes \chi) = 0$ by the symmetric functional equation $(\ref{symmFE})$.
Writing $E$ for the underlying elliptic curve over ${\bf{Q}}$, let $E^{(d_K)}$ denote its quadratic twist.
Assume $E$ has semistable reduction so that its conductor $N$ is squarefree, and for each prime $p \geq 5$ that:
\begin{itemize}
\item The residual Galois representations $E[p]$ and $E^{(d_K)}[p]$ attached to $E$ and $E^{(d_K)}$ are irreducible. 
\item There exists a prime divisor $l \mid \mid N$ distinct from $p$ where the residual representation $E[p]$ is ramified,
and a prime divisor $q \mid \mid N d_K$ distinct from $p$ where the residual representation $E^{(d_K)}[p]$ is ramified.  \end{itemize}

For either elliptic curve $A = E, E^{(d_K)}$, let us write $\Sha_A({\bf{Q}})$ to denote the Tate-Shafarevich group,
with $T_A({\bf{Q}})$ the product over local Tamagawa factors, and $\omega_A$ a fixed invariant differential for $A/{\bf{Q}}$. 
Suppose that $\operatorname{ord}_{s=1} \Lambda(E/K, 1) = 1$, so that either $\Lambda(E, 1) = \Lambda(1/2, \pi)$ 
or the quadratic twist $\Lambda(E^{(d_K)}, 1) = \Lambda(1/2, \pi \otimes \eta)$ vanishes. Writing $[e, e]$ to denote the regulator of 
either $E$ or $E^{(d_K)}$ according to which factor vanishes, we have the following unconditional identity, up to powers of $2$ and $3$: \\
\begin{align*} &\frac{ \# \Sha_E({\bf{Q}}) \cdot \# \Sha_{ E^{(d_K)}}( {\bf{Q}}) \cdot [e, e] \cdot T_E({\bf{Q}}) \cdot T_{ E^{(d_K)} }({\bf{Q}})  }
{ \# E({\bf{Q}})_{\operatorname{tors}}^2 \cdot  \# E^{(d_K)}( {\bf{Q}} )_{\operatorname{tors}}^2  } \cdot \int_{ E({\bf{R}}) } \vert \omega_E \vert
\cdot \int_{ E^{(d_K)}({\bf{R}}) } \vert \omega_{E^{(d_K)}} \vert \\
&= - 2 h_K \log \varepsilon_K \sum\limits_{A \in \operatorname{Pic}(\mathcal{O}_K)} 
\left( \operatorname{CT} \langle \langle f^+_{0, A}, {\bf{1}}_{L_{A,1} \oplus 0} \otimes \mathcal{E}_{L_{A,2}} \rangle \rangle  +
I'(0, f_{0, A} \times \xi_0(\theta_{L_{A,1}})) +
\frac{ \operatorname{vol}(U_{A, 2}) }{2} \cdot \Phi(f_{0,A}, \mathcal{G}_A) \right).\end{align*} 
Note that the value on the left-hand side is known to be a period via the argument of \cite[$\S$4]{KZ}. \end{corollary}

It would be interesting to develop these relations in connection to the real quadratic Borcherds products studied by \cite{DaVo}, 
perhaps leading to a global analogue of Darmon's conjecture \cite[Conjecture 5.6]{Da} via the Borel-Serre compactifications of Siegel threefolds arising as 
spin Shimura varieties associated to subspaces $(\mathfrak{V}_{A,2}, \mathfrak{Q}_{A, 2}) \supset (V_{A, 2}, Q_{A,2})$ of signature (3,2). 
It would also be interesting to investigate the arithmetic properties of the regularized inner product $I(s, f_{0, A} \times \xi_0(\theta_{L_{A,1}}))$
of the harmonic weak Maass forms $f_{0,A}$ and the holomorphic theta series $\xi_0(\theta_{L_{A,1}})$ defined in $(\ref{regint})$, with relations to $\Lambda(s, \Pi \otimes \chi) = \Lambda(E/K, \chi, s+1/2)$.

\subsubsection*{Acknowledgements} I am grateful to Jan Bruinier for many comments and suggestions.
I also thank Thomas Zink for encouraging discussions in 2021-2022, as well as Spencer Bloch, Ashay Burungale,
Henri Darmon, Ben Howard, Steve Kudla, Alex Popa, Don Zagier, and an anonymous referee for helpful exchanges.

\subsubsection*{Outline} We review $L$-functions and their functional equations in $\S 2$, then spin Shimura varieties in $\S 3$, and regularized theta lifts in $\S4$.
We derive the main results in Theorem \ref{4.7} (using Proposition \ref{vanish}), Theorem \ref{MAIN}, and Corollary \ref{DCNF}.
Finally, we describe relations to the Birch and Swinnerton-Dyer Conjecture in $\S 5$.

\section{Background on $L$-functions}

\subsection{Equivalences of $L$-functions and symmetric functional equations}

Let $E$ be an elliptic curve of conductor $N$ defined over ${\bf{Q}}$, parametrized via modularity by a cuspidal 
newform $f \in S_2(\Gamma_0(N))$. Let $\pi = \otimes_v \pi_v$ denote the cuspidal automorphic representation
of $\operatorname{GL}_2({\bf{A}})$ generated by $f$. Hence, we have identifications of completed $L$-functions 
\begin{align}\label{id1} \Lambda(E, s) &= \Lambda (s-1/2 , f) = \Lambda(s-1/2, \pi) = \prod_{v \leq \infty} L(s-1/2, \pi_v). \end{align}

Again, we fix $K$ a real quadratic field of discriminant $d_K$ prime to the conductor $N$, and write $\eta = \eta_{K/ {\bf{Q}}}$
to denote the corresponding Dirichlet character. As well, we fix a ring class character $\chi$ of $K$ of some conductor $c \in {\bf{Z}}_{\geq 1}$
coprime to $d_K N$. Let $K[c]$ denote the ring class extension of $K$ of conductor $c$. 
Inspired by the conjecture of Darmon \cite[Conjecture 5.6]{Da} and the theorem of Gross-Zagier \cite{GZ},
we seek to detect Heegner-like points in the Mordell-Weil group $E(K[c])$
of $K[c]$-rational points through the study of integral presentations of the central derivative value $\Lambda'(E/K, \chi, 1)$ of the completed Hasse-Weil
$L$-function $\Lambda(E/K, \chi, s)$ of $E$ basechanged to $K$ and twisted by $\chi$. 
By the theory of Rankin-Selberg convolution (cf.~e.g.~\cite{GZ}),
we deduce from $(\ref{id1})$ that the Hasse-Weil $L$-function $L(E/K, \chi, s)$ has an analytic continuation 
$\Lambda(E/K, \chi, s)$ to all $s \in {\bf{C}}$ via its identification with the Rankin-Selberg $L$-function 
$\Lambda(s, \pi \times \pi(\chi))$ of $\pi$ times the representation
$\pi(\chi) = \otimes_v \pi(\chi)_v$ of $\operatorname{GL}_2({\bf{A}})$ induced by $\pi$:
\begin{align}\label{id2} \Lambda(E/K, \chi, s) &= \Lambda(s-1/2, \pi \times \pi(\chi)) = \prod_{v \leq \infty} L(s-1/2, \pi_v \times \pi(\chi)_v). \end{align} 
On the other hand, by the theory of cyclic basechange (\cite{La}, \cite{AC}), we can attach
to $\pi$ a cuspidal automorphic representation $\Pi = \operatorname{BC}_{K/{\bf{Q}}}(\pi)$ of $\operatorname{GL}_2({\bf{A}}_K)$.
It is then well-known that the Rankin-Selberg $L$-function $\Lambda(s, \pi \times \pi(\chi))$ 
for $\operatorname{GL}_2({\bf{A}}) \times \operatorname{GL}_2({\bf{A}})$
is equivalent to the twisted standard or Godement-Jacquet $L$-function 
$\Lambda(s, \Pi \otimes \chi)$ on $\operatorname{GL}_2({\bf{A}}_K) \times \operatorname{GL}_1({\bf{A}}_K)$.
This gives us another equivalence of $L$ functions 
\begin{align}\label{id3} \Lambda(E/K, \chi, s) &= \Lambda(s-1/2, \Pi \otimes \chi) = \prod_{w \leq \infty} L(s-1/2, \Pi_w \otimes \chi_w), \end{align}
where we view $\chi$ as an idele class character $\chi = \otimes_w \chi_w$ of $K$ having trivial archimedean component $\chi_{\infty} \equiv 1$.

In each of these presentations $(\ref{id2})$ and $(\ref{id3})$, the $L$-function $L(s, \pi \times \pi(\chi)) = L(s, \Pi \otimes \chi)$ has a well-known
analytic continuation to all $s \in {\bf{C}}$, and satisfies a functional equation relating values at $s$ to $1-s$. 
Moreover, since $\pi \cong \widetilde{\pi}$ is self-dual, and ring class characters are equivariant under complex conjugation, 
the Rankin-Selberg $L$-function $\Lambda(s, \pi \times \pi(\chi))$ satisfies a symmetric functional equation 
\begin{align}\label{symmFE} \Lambda(s, \pi \times \pi(\chi)) &= \epsilon(s, \pi \times \pi(\chi)) \Lambda(1-s, \pi \times \pi(\chi)) \end{align}
with epsilon factor 
\begin{align*} \epsilon(s, \pi \times \pi(\chi)) &= c(\pi \times \pi(\chi))^{\frac{1}{2} - s} \cdot \epsilon(1/2, \pi \times \pi(\chi)) 
= (d_K^2 N^2 c^4)^{\frac{1}{2} - s} \cdot \epsilon(1/2, \pi \times \pi(\chi))\end{align*}
and root number $\epsilon(1/2, \pi \times \pi(\chi)) \in \lbrace \pm 1 \rbrace \subset {\bf{S}}^1$.
If $(N d_K, c)=1$, then this root number $\epsilon(1/2, \pi \times \pi(\chi))$ is well-known (see e.g.~\cite[Proposition 4.1]{BaRa} or \cite[Theorem 2.2, Example 2]{Li}) to be given by the simple formula 
\begin{align}\label{RNF} \epsilon(1/2, \pi \times \pi(\chi)) &=  \eta(-N) = \eta(N). \end{align}
Here, we write $c(\pi \times \pi(\chi)) = d_K^2 N^2 c^4$ to denote the conductor of the $L$-function $\Lambda(s, \pi \times \pi(\chi))$,
and use that the quadratic Dirichlet character $\eta = \eta_{K/{\bf{Q}}}$ is even (as $K$ is a real quadratic field). 
Note that this formula $(\ref{RNF})$ holds for any choice of ring class character $\chi$ of $K$ of conductor $c$ coprime to the product $d_K N$,
and that this functional equation does not depend on the choice of ring class character $\chi$. 
Since the functional equation $(\ref{symmFE})$ is symmetric, we deduce that there must
be forced vanishing of the central value $\Lambda(1/2, \pi \times \pi(\chi)) = \Lambda(1/2, \Pi \otimes \chi) = 0$ when $\eta(N) = -1$. We can
therefore impose the following condition on the level $N$ of $\pi$, equivalently the conductor $N$ of $f$ and $E$, to ensure this forced vanishing.
Here, since we assume that $N$ is coprime to the discriminant $d_K$, we can assume that the conductor $N$ factorizes as $N = N^+ N^{-}$, 
where for each prime $q \mid N$,
\begin{align*} q \mid N^+ \quad &\iff \quad \eta(q) = 1 \quad \iff \text{ $q$ splits in $K$} \\
q \mid N^- \quad &\iff \quad \eta(q) = -1 \quad \iff \text{ $q$ is inert in $K$}. \end{align*} 

\begin{hypothesis}\label{EHH} 

Let us assume that the inert level $N^-$ is the squarefree product of an odd number of primes,
and hence that the root number of $\Lambda(s, \pi \times \pi(\chi))$ for $\chi$ any ring class character of $K$ of conductor 
$c$ prime to $d_KN$ is given by $\epsilon(1/2, \pi \times \pi(\chi)) = \eta(-N) = \eta(N) = \eta(N^-)=-1$. \end{hypothesis}

If the condition of Hypothesis \ref{EHH} is met, then the corresponding central value 
$\Lambda(1/2, \pi \times \pi(\chi))$ is forced by the functional equation $(\ref{symmFE})$ to vanish: 
$\Lambda(1/2, \pi \times \pi(\chi)) = \Lambda(1/2, \Pi \otimes \chi) =0$.
It then makes sense to derive integral presentations for the central derivative values in this case, 
\begin{align*} \Lambda'(1/2, \pi \times \pi(\chi)) &= \Lambda'(1/2, \Pi \otimes \chi) = \textit{?} \end{align*}
The conjectures of Birch-Swinnerton-Dyer and Darmon \cite[Conjecture 5.6]{Da} suggest that this central derivative value should 
be related to the height of a CM-type point on the modular curve $X_0(N)$, or some higher-dimensional Shimura variety into which it embeds. 

\subsection{The basechange representation} 

We consider the quadratic basechange lifting $\Pi = \operatorname{BC}_{K/{\bf{Q}}}(\pi)$ of $\pi$ to $\operatorname{GL}_2({\bf{A}}_K)$,
which exists by the theory of Langlands \cite{La} and more generally Arthur-Clozel \cite{AC}. Note that $\Pi$ has trivial central character. 
We refer to the article of G\'erardin-Labesse \cite{GL} for more background. 

\begin{proposition}\label{cuspidal} 

Let $\pi = \pi(f)$ be the cuspidal automorphic representation of 
$\operatorname{GL}_2(\mathbf A_\mathbf Q)$ attached to a weight $2$ newform parametrizing an elliptic curve $E/\mathbf Q$.
Let $K$ be a real quadratic field and let $\Pi = \operatorname{BC}_{K/\mathbf Q}(\pi)$ denote its quadratic basechange. Then $\Pi$ is cuspidal.

\end{proposition}

\begin{proof}  

By quadratic basechange for $\operatorname{GL}_2({\bf{A}}_K)$ (\cite{La}, \cite[Ch.~3]{AC}), the basechange of a cuspidal representation is cuspidal
unless the original representation is automorphically induced from a Hecke character of the quadratic extension $K$.
If $\pi$ were automorphically induced from $K$, then it would be dihedral with respect to $K$,
and the associated elliptic curve would have complex multiplication by $K$. Since $K$ is real quadratic, this is impossible. Hence $\Pi$ is cuspidal. \end{proof}

Let $E/\mathbf Q$ be an elliptic curve parametrized by a newform $f$ of weight $2$ with associated cuspidal representation $\pi$.
Let $K/\mathbf Q$ be a real quadratic field with quadratic character $\eta$. 
As a consequence of quadratic basechange and the compatibility of $L$-functions (\cite{La}, \cite{AC}), the completed $L$-function of $E$ over $K$ satisfies
\begin{align*} \Lambda(E/K,s) = \Lambda(s-1/2,\Pi) = \Lambda(s-1/2,\pi)\,\Lambda(s-1/2,\pi\otimes\eta). \end{align*}
Equivalently,
\begin{align*} \Lambda(E/K,s) = \Lambda(E,s)\,\Lambda(E\otimes\eta,s). \end{align*} 

\section{Spin groups and orthogonal groups}

We now describe spin groups associated to rational quadratic spaces of signature $(2, 2)$.
Here, we follow \cite[$\S$ 2.3-2.7]{Br} and \cite[$\S$ 2-4]{BY}, but adapt for the special setting we consider in Proposition \ref{spinID} below.

\subsection{Rational quadratic spaces of signature $(2, 2)$}

Let $(V, Q)$ be any rational quadratic space $(V, Q)$ of signature $(2,2)$ and bilinear form $(v_1, v_2) = \frac{1}{2} \left\lbrace Q(v_1 + v_2) - Q(v_1) - Q(v_2) \right\rbrace$.
We shall later focus on the special example described above. That is, we consider the real quadratic field $K = {\bf{Q}}(\sqrt{d})$ with $d > 0$. 
Recall that for an integer $c \geq 1$, we consider the ring class group $\operatorname{Pic}(\mathcal{O}_c)$ 
of the ${\bf{Z}}$-order $\mathcal{O}_c := {\bf{Z}} + c \mathcal{O}_K$ of conductor $c$ in $K$ through which 
our fixed ring class character $\chi$ factors. We assume throughout that $\operatorname{Pic}(\mathcal{O}_c)$ is defined.
Note that this will always be so for $c=1$, in which case $\operatorname{Pic}(\mathcal{O}_c) = \operatorname{Pic}(\mathcal{O}_K)$
can be identified with the ideal class group of $\mathcal{O}_K$. 
We fix for each class $A \in \operatorname{Pic}(\mathcal{O}_c)$ an integral ideal representative $\mathfrak{a} \subset \mathcal{O}_K$. 

\begin{definition}\label{V_A}

Writing $Q_{\mathfrak{a}}(z) = {\bf{N}}_{K/{\bf{Q}}}(z)/{\bf{N}} \mathfrak{a}$ to denote the corresponding 
norm form of signature $(1,1)$, we consider the quadratic space defined by 
$V_A =  \mathfrak{a}_{\bf{Q}} \oplus \mathfrak{a}_{\bf{Q}}$ 
for $\mathfrak{a}_{\bf{Q}} = \mathfrak{a} \otimes {\bf{Q}}$ with one of the two following (essentially equivalent) quadratic forms $q_A$ and $Q_A$:

\begin{itemize}

\item[(i)] $V_A = \mathfrak{a}_{\bf{Q}} \oplus \mathfrak{a}_{\bf{Q}}$ 
with $q_A(x, y, \lambda) := Q_{\mathfrak{a}}(\lambda) - xy = {\bf{N}} \mathfrak{a}^{-1} \cdot {\bf{N}}_{K/{\bf{Q}}}(\lambda) - xy$,

\item[(ii)] $V_A = \mathfrak{a}_{\bf{Q}} \oplus \mathfrak{a}_{\bf{Q}}$ 
with $Q_A(z) = Q_A(z_1, z_2) := Q_{\mathfrak{a}}(z_1) - Q_{\mathfrak{a}}(z_2)$.

\end{itemize}

\end{definition}

We see by inspection that $(V_A, q_A)$ is a rational quadratic space of signature $(2, 2)$ as $d >0$ .
We also see by inspection that $(V_A, Q_A)$ has signature $(2,2)$ if $d > 0$.
For either choice of quadratic form, we write  
$\left(  \cdot, \cdot \right)_A : V_A \times V_A \rightarrow {\bf{Q}}$ for the corresponding hermitian bilinear form.

\subsection{Spin groups and exceptional isomorphisms} 

Let $(V, Q)$ be any rational quadratic space of signature $(2,2)$.
Let $C_V$ denote the corresponding Clifford algebra. Hence, $C_V = T_V/I_V$ for $T_V$ the tensor algebra
\begin{align*} T_V &= \bigoplus_{m \geq 0} V^{\otimes m} = {\bf{Q}} \oplus V \oplus (V \otimes_{\bf{Q}} V) \oplus \cdots, \end{align*}
with $I_V \subset T_V$ the two-sided ideal generated by $v \otimes v - Q(v)$ for $v \in V$. 
So, $C_V$ is a ${\bf{Q}}$-module of rank $2^4 = 16$. There are canonical embeddings of ${\bf{Q}}$ and $V$ into $C_V$. 
We denote an element of the form $v_1 \otimes \cdots \otimes v_m$ in $C_V$ for $v_i \in V$ by $v_1 \cdots v_m$ for simplicity. 
Let $C_V^0 \subset C_V$ denote the ${\bf{Q}}$-submodule generated by products of even numbers of basis vectors of $V$.
Writing $C_V^1 \subset C_V$ to denote the ${\bf{Q}}$-submodule generated by products of odd numbers of basis 
vectors of $V$, we have the decomposition $C_V \cong C_V^0 \oplus C_V^1$. 

\begin{theorem}\label{Clifford} 

Let $(V, Q)$ be any rational quadratic space of signature $(2,2)$, with Clifford algebra $C_V$. \\

\begin{itemize}

\item[(i)] Fix any orthogonal basis $v_1, v_2, v_3, v_4$ of $V$, and put $\delta = v_1 v_2 v_3 v_4$. 
We can identify the centre $Z(C_V)$ of the Clifford algebra $C_V$ with ${\bf{Q}}$, 
and the centre $Z(C_V^0)$ of its even part $C_V^0$ with ${\bf{Q} + {\bf{Q}}} \delta$. \\

\item[(ii)] Fix any basis $v_1, v_2, v_3, v_4 \in V$ and let $S = ((v_i, v_j))_{i, j}$ denote the corresponing Gram matrix.
The determinant $d(V) = \det(S)$ does not depend on the chosen basis and defines the discriminant of $V$.
Moreover, we have the relation $\delta^2 = 2^{-4} d(V) \in {\bf{Q}}^{\times}/({\bf{Q}}^{\times})^2$ for the volume form $\delta$ defined in (i). \\

\end{itemize}

\end{theorem}

\begin{proof} See \cite[$\S$ 2.2, Theorem 2.6 and Remark 2.5]{Br}, these results are standard. \end{proof}

Multiplication by $-1$ defines an isometry of $V$, and induces an algebra homomorphism $J: C_V \rightarrow C_V$. 
Note that we can also define $C_V^0 = \lbrace v \in C_V: J(v)=v \rbrace$. 
We have the canonical anti-involution ${}^tC_V \rightarrow C_V$ defined by $(x_1 \otimes \cdots \otimes x_m)^t := x_m \otimes \cdots \otimes x_1$,
from which we define the Clifford norm
\begin{align*} N_{C_V}: C_V \longrightarrow C_V, \quad N_{C_V}(x) := x^t x. \end{align*}
Note that for $x \in V$, we have $N_{C_V}(x) = Q(x)$. We conside the Clifford group $\operatorname{CG}_V$ defined by
\begin{align*} \operatorname{CG}_V &= \left\lbrace x \in C_V : x \text{ ~invertible,~} x V J(x)^{-1} = V  \right\rbrace. \end{align*}
This definition allows us to associate to each $x \in C_V$ an automorphism $\alpha_x$ of $V$ defined by $\alpha_x(v) = x v J(x)^{-1}$
(for any $v \in V$). We obtain from this a linear representation 
$\alpha: \operatorname{CG}_V \longrightarrow \operatorname{Aut}_{\bf{Q}}(V), ~ x \mapsto \alpha_x$ known as the vector representation. 
Note that the involution $x \mapsto x^t$ sends $\operatorname{CG}_V$ to itself, and so $N_{C_V}(x) \in \operatorname{CG}_V$ for any $x \in C_V$. 
We also know (see \cite[Lemma 2.11]{Br}) that the kernel of the vector representation $\alpha: \operatorname{CG}_V \rightarrow \operatorname{Aut}_{\bf{Q}}(V)$
equals ${\bf{Q}}^{\times}$, that the Clifford norm $N_{C_V}$ induces a homomorphism $\operatorname{CG}_V \rightarrow {\bf{Q}}^{\times}$, and that $N_{C_V}$ in this setting is multiplicative. 
We now consider the general spin group $\operatorname{GSpin}_V = \operatorname{CG}_V \cap C_V^0$ and underlying spin group
\begin{align*} \operatorname{Spin}(V) 
&= \left\lbrace x \in \operatorname{GSpin}_V = \operatorname{CG}_V \cap C_V^0: N_{C_V}(x) = 1 \right\rbrace.\end{align*}
As the vector representation $\alpha$ here is surjective with kernel ${\bf{Q}}^{\times}$,
we see that the Clifford group $\operatorname{GC}_V$ is a central extension of the orthogonal group $\operatorname{O}(V)$, 
and that the general spin group $\operatorname{GSpin}_V$ is a central extension of the special orthogonal group 
$\operatorname{SO}(V)$. That is, we have short exact sequences
\begin{equation}\begin{aligned}\label{GSpinSES} 1 &\longrightarrow {\bf{G}}_m \longrightarrow \operatorname{CG}_V \longrightarrow \operatorname{O}(V) \longrightarrow 1, \\
1 &\longrightarrow {\bf{G}}_m \longrightarrow \operatorname{GSpin}(V) \longrightarrow \operatorname{SO}(V) \longrightarrow 1.\end{aligned}\end{equation}

\begin{proposition}\label{spinID} 

We have the following identifications of spin groups for the rational quadratic spaces $(V_A, q_A)$ and $(V_A, Q_A)$ described in Definition \ref{V_A}.
Fix any class $A \in \operatorname{Pic}(\mathcal{O}_c)$ with integer ideal representative $\mathfrak{a} \subset \mathcal{O}_K$. 
We again write $Q_{\mathfrak{a}} (z) = {\bf{N}}_{K/{\bf{Q}}}(z)/{\bf{N}} \mathfrak{a}$ to denote the norm form, as well
as ${\bf{N}}_{K/{\bf{Q}}}(z) = z z^{\tau}$ for the nontrivial automorphism $\tau \in \operatorname{Gal}(K/{\bf{Q}})$ to denote the norm homomorphism. \\

\begin{itemize}

\item[(i)] Consider the quadratic space $(V_A, q_A)$ of signature (2,2) given by $V_A = \mathfrak{a}_{\bf{Q}} \oplus \mathfrak{a}_{\bf{Q}}$
and quadratic form $q_A(x, y, \lambda) := Q_{\mathfrak{a}}(\lambda) - xy$. Then, the centre $Z(C_{V_A}^0)$ of the even Clifford
algebra $C_{V_A}^0$ is given by $K$, and we have an exceptional isomorphism 
$\operatorname{Spin}(V_A) \cong \operatorname{Res}_{K/{\bf{Q}}} \operatorname{SL}_2(K)$ of algebraic groups over ${\bf{Q}}$. \\

\item[(ii)] Consider the quadratic space $(V_A, Q_A)$ of signature (2,2) given by  $V_A = \mathfrak{a}_{\bf{Q}} \oplus \mathfrak{a}_{\bf{Q}}$
with the quadratic form $Q_A(z) = Q_A((z_1, z_2)) := Q_{\mathfrak{a}}(z_1) - Q_{\mathfrak{a}}(z_2)$. The centre $Z(C_{V_A}^0)$ of the even Clifford
algebra $C_{V_A}^0$ is given by ${\bf{Q}}$, and we have exceptional isomorphisms $\operatorname{Spin}(V_A) \cong \operatorname{SL}_2^2$ and 
\begin{align*} \operatorname{GSpin}(V_A) \cong  \operatorname{GL}_2 \times_{{\bf{G}}_m} \operatorname{GL}_2
= \left\lbrace (g_1, g_2) \in \operatorname{GL}_2 \times \operatorname{GL}_2 : \det(g_1) = \det(g_2) \right\rbrace\end{align*}
of algebraic groups over ${\bf{Q}}$. \\
\end{itemize}

\end{proposition}

\begin{proof} Write $d(V_A)$ in each case to denote the discriminant of the space, computed as the determinant 
$d(V_A) = \det((e_i, e_j)_A)$ of the Gram matrix $((e_i, e_j)_A)_{i, j}$ for any choice of basis $\lbrace e_i \rbrace_{i=1}^4$ of $V_A$. 
For (i), a direct calculation shows that $d(V_A) \equiv d \pmod{(\mathbf{Q}^\times)^2}$ (cf.~\cite[$\S$2.7]{Br}).
Hence $Z(C_{V_A}^0) = \mathbf{Q}(\sqrt d)=K$. By \cite[$\S$ 2.7]{Br}, the even Clifford algebra is isomorphic to $M_2(K)$,
and we obtain the exceptional isomorphism $\operatorname{Spin}(V_A) \cong \operatorname{Res}_{K/\mathbf Q} \operatorname{SL}_2.$
For (ii), a direct calculation shows that $d(V_A) \equiv 1 \pmod{(\mathbf{Q}^\times)^2}$.
Hence, $C_{V_A}^0$ is split and isomorphic to $M_2(\mathbf Q) \oplus M_2(\mathbf Q)$. 
So, $\operatorname{GSpin}(V_A)$ can be identified with a subgroup of $\operatorname{GL}_2^2$.
Now, observe that the second short exact sequence $(\ref{GSpinSES})$ forces 
$\dim \operatorname{GSpin}(V_A) = \dim \operatorname{SO}(V_A) + 1 = \dim \operatorname{SO}(2,2)+1 = 6+1=7$.
We deduce from this constraint on the dimension that $\operatorname{GSpin}(V_A) \cong  \operatorname{GL}_2 \times_{{\bf{G}}_m} \operatorname{GL}_2$. \end{proof}

\begin{corollary}\label{lattices} 

Fix an integer $N \geq 1$. Let $K_0(N)$ denote the compact open subgroup of 
$\operatorname{GL}_2(\widehat{\bf{Z}}) \subset \operatorname{GL}_2({\bf{A}}_f)$ corresponding to the 
congruence subgroup $\Gamma_0(N) = K_0(N) \cap \operatorname{GL}_2({\bf{Q}})$, given by 
\begin{align*} K_0(N) &= \left\lbrace \left( \begin{array}{cc} a & b \\ c & d\end{array} \right) 
\in \operatorname{GL}_2(\widehat{\bf{Z}}) : c \equiv 0 \bmod N \right\rbrace. \end{align*}
Fix $(V_A, Q_A)$ any of the quadratic spaces described in Proposition \ref{spinID} (ii).
Let $U_A =  U_A(N)$ denote the compact open subgroup corresponding to $K_0(N) \times K_0(N)$
under the isomorphism $\operatorname{GSpin}(V_A)({\bf{A}}_f) \cong \operatorname{GL}_2({\bf{A}}_f) \times_{{\bf{G}}_m} \operatorname{GL}_2({\bf{A}}_f)$.
Under the action of $\operatorname{GSpin}(V_A)({\bf{A}}_f)$ on $V_A$ by conjugation, there is a unique lattice $L_A = L_A(N)$ of $V_A$ whose adelization 
$L_A \otimes \widehat{{\bf{Z}}}$ is stabilized by $U_A = U_A(N) \cong K_0(N)^2$.
More explicitly, this lattice is given by $L_A = L_A(N) = N^{-1} \mathfrak{a} \oplus N^{-1} \mathfrak{a}$, with dual lattice
given by $L_A^{\vee} = L_A(N)^{\vee} = \mathfrak{d}_K^{-1} N^{-1} \mathfrak{a} \oplus \mathfrak{d}_K^{-1} N^{-1} \mathfrak{a}$. 
This lattice has level $N = \lbrace \min a \in {\bf{Z}} : a Q_A(\lambda) \in {\bf{Z}} \quad \forall \lambda \in L_A^{\vee} \rbrace$. \end{corollary}

\begin{proof} 

Recall we have the canonical embedding $V_A \rightarrow C_{V_A}$, that we can identify 
$\operatorname{GSpin}(V_A)$ with the elements in $C_{V_A}^0$ with Clifford norm in ${\bf{Q}}^{\times}$,
and by Proposition \ref{spinID} that $C_{V_A}^0 \cong M_2({\bf{Q}}) \oplus M_2({\bf{Q}})$. 
Writing $R(N)$ to denote the Eichler order of level $N$ in $M_2({\bf{Z}})$, 
we seek to find the lattice $L_A(N) \subset V_A$ fixed by conjugation by the invertible matrices in $R(N) \oplus R(N)$. 
The conjugation action $g \cdot v = g v g^{-1}$ of $g = (g_1, g_2) \in \operatorname{GSpin}(V_A)$
on $v = (v_1, v_2) \in V_A = \mathfrak{a}_{\bf{Q}} \oplus \mathfrak{a}_{\bf{Q}}$ is given by 
$(g_1, g_2) \cdot (v_1, v_2) = (g_1 v_1 g_1^{-1}, g_2 v_2 g_2^{-1}).$ 
A direct verification shows that the lattice $L_A = L_A(N) = N^{-1} \mathfrak{a} \oplus N^{-1} \mathfrak{a}$ is stabilized under this action,
that the dual lattice is given by $L_A^{\vee} = L_A(N)^{\vee} = \mathfrak{d}_k^{-1} N^{-1} \mathfrak{a} \oplus \mathfrak{d}_k^{-1} N^{-1} \mathfrak{a}$, and the level is $N$. \end{proof}

\begin{remark}[Relation to quadratic basechange liftings.] 

Consider the split quadratic space $V_0 = {\bf{Q}} \oplus {\bf{Q}} \oplus K$ 
with quadratic form $q_0(x, y, \lambda) = {\bf{N}}_{K/{\bf{Q}}}(\lambda) - xy$.
Although we do not use this quadratic space $(V_0, q_0)$ for our main calculations, 
we note that the accidental isomorphism $\operatorname{Spin}(V_0) \cong \operatorname{Res}_{K/{\bf{Q}}}(\operatorname{SL}_2(K))$ 
realizes the quadratic basechange lifting $\Pi = \operatorname{BC}_{K/{\bf{Q}}}(\pi)$ of the cuspidal 
automorphic representation $\pi =\pi(f)$ to $\operatorname{GL}_2({\bf{A}}_K)$ as a theta lift from 
$\operatorname{SL}_2({\bf{A}})$ to $\operatorname{Spin}(V_0)({\bf{A}})$, which after extending to similitudes can 
be viewed as a theta lift from $\operatorname{GL}_2({\bf{A}})$ to $\operatorname{GSpin}(V_0)({\bf{A}})$. 
We refer to \cite[$\S$2-3]{Br} for a classical description of this setup. \end{remark}

\section{Regularized theta lifts and automorphic Green's functions}

We now introduce regularized theta lifts associated with the quadratic spaces $(V_A, Q_A)$ described in Proposition \ref{spinID} (ii) 
above following \cite{Bo, KuBL, BrB, BF, BY}. We shall later compute these along the real geodesic cycle attached to the anisotropic subspace 
$(V_{A, 2}, Q_{A, 2}) = (V_{A,2}, Q_{A}\vert_{V_{A,2}})$ of signature $(1,1)$ defined by $V_{A, 2} := \mathfrak{a}_{\bf{Q}} = \mathfrak{a} \otimes {\bf{Q}}$
and $Q_{A, 2}(\lambda) = Q_{\mathfrak{a}}(\lambda) = {\bf{N}}_{K/{\bf{Q}}}(\lambda)/{\bf{N}} \mathfrak{a}$.
We note that this presents a substantial departure from these works, as these Lorentzian 
spaces of signature $(1,1)$ give rise to nonholomorphic Siegel theta series $\theta_{L_{A,1}}(\tau)$ and an additional term
$I'(0, f_{0,A} \times \xi_0 \theta_{L_{A,1}})$ coming from a regularized Rankin-Selberg-type integral $I(s, f_{0,A} \times \xi_0 \theta_{L_{A,1}})$.
These sums over geodesic cycles allow us to derive new integral presentations for the central derivative values 
$\Lambda'(E/K, \chi, 1) = \Lambda'(1/2, \Pi \otimes \chi) = \Lambda'(1/2, f \times \theta(\chi))$. 

\subsection{Setup} 

Fix a primitive ring class character $\chi$ of $K$ of some conductor $c \in {\bf{Z}}_{\geq 1}$ coprime to $N d_K$, 
which we assume exists. (This is always the case for conductor $c=1$, whence $\chi$ is a class group character). 
Thus, $\chi$ factors through the ring class group $\operatorname{Pic}(\mathcal{O}_c)$. 
Let us for each class $A \in \operatorname{Pic}(\mathcal{O}_c)$ 
fix an integral ideal representative $\mathfrak{a} \subset \mathcal{O}_K$
of $A = [\mathfrak{a}] \in \operatorname{Pic}(\mathcal{O}_c)$. 
We consider the rational quadratic space $(V_A, Q_A)$ of signature $(2, 2)$ defined in Definition \ref{V_A} (ii),
hence with vector space $V_A = \mathfrak{a}_{\bf{Q}} \oplus \mathfrak{a}_{\bf{Q}}$ 
and quadratic form $Q_A(z) = Q_A((z_1, z_2)) = Q_{\mathfrak{a}}(z_1) - Q_{\mathfrak{a}}(z_2)$.

\subsubsection{Exceptional isomorphisms}

Recall that by Proposition \ref{spinID} (ii), we have an exceptional isomorphism 
\begin{align}\label{accidentV_A} \zeta: \operatorname{GSpin}(V_A) \cong  \operatorname{GL}_2 \times_{{\bf{G}}_m} \operatorname{GL}_2 \end{align} 
of algebraic groups over ${\bf{Q}}$. As described in Corollary \ref{lattices}, 
we take $U_A \subset \operatorname{GSpin}(V_A)({\bf{A}}_f)$ to be the compact open subgroup  
$U_A = \prod_{p < \infty} U_{A, p}$, with each local component given by $\zeta(U_{A, p}) \cong K_{0, p}(N) \times K_{0, p}(N)$, where
\begin{align}\label{level} K_{0, p}(N) &= \left\lbrace \left( \begin{array}{cc} a & b \\ c & d \end{array} \right) \in \operatorname{GL}_2({\bf{Z}}_p) 
: c \in N {\bf{Z}}_p \right\rbrace \subset \operatorname{GL}_2({\bf{Z}}_p). \end{align}
Given any integral lattice $L_A \subset V_A$, we write $L_A^{\vee}$ to denote the corresponding dual lattice, 
and $L_A^{\vee}/L_A$ to denote the corresponding finite abelian discriminant group. We shall later 
take $L_A = L_A(N)$ to be the lattice whose adelization is fixed by $U_A$, as described in Corollary \ref{lattices}.

\subsubsection{Weil representations} 

Let $\psi = \otimes_v \psi_{v}$ denote the standard additive character of ${\bf{A}}/{\bf{Q}}$, with archimedean component $\psi_{\infty}(x) = e(x)= \exp(2 \pi i x)$. 
Recall that for each $A \in \operatorname{Pic}(\mathcal{O}_c)$, we have a short exact sequence
\begin{align*} 1 \longrightarrow {\bf{G}}_m \longrightarrow \operatorname{GSpin}(V_A) \longrightarrow \operatorname{SO}(V_A) \longrightarrow 1\end{align*}
of algebraic groups defined over ${\bf{Q}}$. Let $\omega_{L_A}$ denote the corresponding Weil representation
\begin{align*} \omega_{L_A} = \omega_{L_A, \psi}: \operatorname{SL}_2({\bf{A}})  \times \operatorname{GSpin}(V_A)({\bf{A}}) 
\longrightarrow \operatorname{Aut} \left(  \mathcal{S}(V_A({\bf{A}})) \right) \end{align*}
of $\operatorname{SL}_2({\bf{A}}) \times \operatorname{GSpin}(V_A)({\bf{A}}) $ acting on the space $\mathcal{S}(V_A({\bf{A}}))$
of Schwartz-Bruhat functions on $V_A({\bf{A}})$. 
The action of $\operatorname{SL}_2({\bf{A}})$ on $\mathcal{S}(V_A({\bf{A}}))$ commutes with that of $\operatorname{GSpin}(V_A)({\bf{A}})$.
We write $\omega_{L_A}(h) \varphi (x) = \varphi(h^{-1} x)$ for $h \in \operatorname{GSpin}(V_A)({\bf{A}})$ and $\varphi \in \mathcal{S}(V_A({\bf{A}}))$ to denote the latter action. 

\subsubsection{Subspaces of Schwartz functions}

Let $\mathcal{S}_{L_A} \subset \mathcal{S}(V_A({\bf{A}}_f))$ 
denote the subspace of Schwartz functions with support on $\widehat{L}_A^{\vee} = L_A^{\vee} \otimes \widehat{ {\bf{Z}} }$ 
which are constant on cosets of $\widehat{L}_A = L_A \otimes \widehat{ {\bf{Z}} }$.
Note that $\mathcal{S}_{L_A}$ admits a basis of characteristic functions,
\begin{align}\label{latticebasis} \mathcal{S}_{L_A} 
&= \bigoplus_{\mu \in L_A^{\vee}/L_A} {\bf{C}} \cdot {\bf{1}}_{\mu} \subset \mathcal{S}(V_A({\bf{A}}_f)),
\quad {\bf{1}}_{\mu} := \operatorname{char} \left( \mu + \widehat{L}_A \right).\end{align}
This space $\mathcal{S}_{L_A}$ is stable under the action of $\operatorname{SL}_2({\bf{Z}})$ through the Weil representation $\omega_{L_A}$. 
Moreover, the space of Schwartz functions $\mathcal{S}(V_A({\bf{A}}_f))$ can be expressed as the direct limit $\varinjlim_{L_A} \mathcal{S}_{L_A}$ of these subspaces.

\subsubsection{Anisotropic subspaces}

For each of the quadratic spaces $(V_A, Q_A)$ described in Definition \ref{V_A} (ii) above, 
we consider the anisotropic subspace $(V_{A, 2}, Q_{A, 2}) = (V_{A,2}, Q_{A}\vert_{V_{A,2}})$ 
of signature $(1,1)$ defined by the fractional ideal $V_{A, 2} := \mathfrak{a}_{\bf{Q}} = \mathfrak{a} \otimes {\bf{Q}}$
and norm form $Q_{A, 2}(\lambda) = Q_{\mathfrak{a}} = {\bf{N}}_{K/{\bf{Q}}}(\lambda)/{\bf{N}} \mathfrak{a}$. 
We also consider the anisotropic subspace $(V_{A, 1}, Q_{A, 1}) = (V_{A,1}, Q_{A}\vert_{V_{A,1}})$ 
of signature $(1,1)$ defined by $V_{A, 1} := \mathfrak{a}_{\bf{Q}}$ and negative norm form 
$Q_{A, 1}(x, y) = - Q_{\mathfrak{a}}$. We write $(V_{A, j}, Q_{A, j})$ for $j=1,2$ to denote either of these spaces.

Writing $K^1 \subset K^{\times}$ to denote the elements of norm one, it is easy to see that 
$\operatorname{Spin}(V_{A, j}) \cong \operatorname{SO}(V_{A, j}) \cong K^1$ for each of $j=1,2$.
Writing $K_{ {\bf{A}} }^1$ to denote the adelic points, we have the Hilbert exact sequence 
\begin{align*}\begin{CD} 1 @>>> {\bf{A}}^{\times} @>>> {\bf{A}}_K^{\times} @>>> K_{\bf{A}}^1 @>>> 1. \end{CD} \end{align*}
In particular, we obtain natural identifications for the corresponding adelic quotient spaces
\begin{align*} \operatorname{Spin}(V_{A, j})({\bf{Q}}) \backslash \operatorname{Spin}(V_{A, j})({\bf{A}}) \cong 
\operatorname{SO}(V_{A, j})({\bf{Q}}) \backslash \operatorname{SO}(V_{A, j})({\bf{A}}) \cong {\bf{A}}_K^{\times}/ {\bf{A}}^{\times} K^{\times}. \end{align*} 
Hence, we can view the ring class character $\chi: {\bf{A}}_K^{\times}/ {\bf{A}}^{\times} K^{\times} \rightarrow {\bf{C}}^{\times}$ as an automorphic
representation of $\operatorname{SO}(V_{A, j})({\bf{A}})$. In a similar way, we have natural identifications 
\begin{align*} \operatorname{GSpin}(V_{A, j})({\bf{Q}}) \backslash \operatorname{GSpin}(V_{A, j})({\bf{A}}) \cong 
\operatorname{GO}(V_{A, j})({\bf{Q}})\backslash \operatorname{GO}(V_{A, j})({\bf{A}}) \cong {\bf{A}}_K^{\times}/K^{\times}. \end{align*} 
Here, strictly speaking, we fix one of the two connected components $\operatorname{GO}^{\pm}(V_{A, j})$ of $\operatorname{GO}(V_{A,j})$ so that
\begin{align*}\operatorname{GSpin}(V_{A, j})({\bf{Q}}) \backslash \operatorname{GSpin}(V_{A, j})({\bf{A}}) \cong 
\operatorname{GO}^{\pm}(V_{A, j})({\bf{Q}}) \backslash \operatorname{GO}^{\pm}(V_{A, j})({\bf{A}}) \cong {\bf{A}}_K^{\times}/K^{\times}. \end{align*} 
We refer to the discussion in \cite[Theorem 2.3.3]{Po}.
 
\subsection{Hermitian symmetric domains}

The symmetric spaces associated to each quadratic space $(V_A, Q_A)$ are hermitian symmetric domains, so have a complex structure. We have the following equivalent realizations. 
    
\subsubsection{The Grassmannian model}    
     
Recall we let $D(V_A)$ denote the Grassmannian of oriented $2$-planes of $V_A({\bf{R}})$ on which $Q_A$ is negative definite.

\subsubsection{The projective model}

Note that $D(V_A)$ can be identified with the complex surface  
\begin{align*} Q(V_A) &= \left\lbrace w \in V_A({\bf{C}}) : (w, w)_A = 0, (w, \overline{w})_A < 0 \right\rbrace / {\bf{C}}^{\times} \subset {\bf{P}} (V_A({\bf{C}})) \end{align*}
via the map 
\begin{align}\label{quadric} D^{\pm}(V_A) \longrightarrow Q(V_A), \quad z \longmapsto v_1 - i v_2 = w, \end{align}
for $v_1, v_2$ a properly-oriented standard basis of $D^{\pm}(V_A)$ with $(v_1, v_1)_A = (v_2, v_2)_A =-1$ 
and $(v_1, v_2)_A = 0$. We refer to this identification $D^{\pm}(V_A) \cong Q(V_A)$ as the {\it{projective model}}. 

\subsubsection{The tube domain model}

Fix a Witt decomposition $V_A({\bf{R}}) = V_{A, 0} + {\bf{R}} \cdot e + {\bf{R}} \cdot f$,
with $e$ and $f$ chosen so that $(e, e)_A = (f, f)_A= 0$ and $(e, f)_A =1$, 
and $C(V_A) = \left\lbrace y \in V_{A, 0}: (y, y)_A < 0 \right\rbrace$ its negative cone. 
We can then identify $D^{\pm}(V_A) \cong Q(V_A)$ with the corresponding tube domain
\begin{align*}\mathcal{H}(V_A) 
&:= \left\lbrace z \in V_{A, 0}({\bf{C}}) : \Im(z_0) \in C(V_A)  \right\rbrace \cong \mathfrak{H}^{\pm} = \mathfrak{H}^{+} \coprod \mathfrak{H}^{-} \end{align*}
via the map $\mathcal{H}(V_A) \longrightarrow V_A({\bf{C}})$ sending $ z \longmapsto w(z) := z + e - q_A(z) f$ 
composed with the projection to $Q(V_A)$. We call $\mathcal{H}(V_A) \subset V_{A, 0}({\bf{C}}) \cong {\bf{C}}^2$ the {\it{tube domain model}}. 

\subsection{Spin Shimura varieties}
 
We now describe the Shimura varieties associated with each group $\operatorname{GSpin}(V_A)$.
Here, we take $U_A \subset \operatorname{GSpin}(V_A)({\bf{A}}_f)$ to be any compact open subgroup. 
 
\subsubsection{Orbifolds} 

Consider the Shimura varieties $X_{U_A} = \operatorname{Sh}_{U_A}(\operatorname{GSpin}(V_A), D(V_A))$ with complex points  
\begin{align*} X_{U_A}({\bf{C}}) 
&= \operatorname{GSpin}(V_A)({\bf{Q}}) \backslash \left( D(V_A) \times \operatorname{GSpin}(V_A) ( {\bf{A}}_f ) / U_A \right) \\
&\cong \operatorname{GSpin}(V_A)({\bf{Q}}) \backslash \left( \mathfrak{H}^{\pm} \times \mathfrak{H}^{\pm} \times \operatorname{GSpin}(V_A) ( {\bf{A}}_f ) / U_A \right). \end{align*}
Note that this is a surface defined over ${\bf{Q}}$. 
Via the exceptional isomorphism $(\ref{accidentV_A})$ with choice of level $(\ref{level})$, we obtain the identification
\begin{align*} X_{U_A}({\bf{C}}) \cong  \operatorname{GL}_2({\bf{Q}}) \times_{{\bf{G}}_m} \operatorname{GL}_2({\bf{Q}}) \backslash \left( \mathfrak{H}^{\pm} \times \mathfrak{H}^{\pm} \times 
 \operatorname{GL}_2({\bf{A}}_f) \times_{{\bf{G}}_m} \operatorname{GL}_2({\bf{A}}_f) / \zeta(U_A) \right) \end{align*}
with the two-fold product $Y_0(N) \times Y_0(N)$ of the noncompactified modular curve $Y_0(N) = \Gamma_0(N) \backslash \mathfrak{H}$. 

\subsubsection{Decompositions}

Fix a set of representatives $h_j \in \operatorname{GSpin}(V_A)({\bf{Q}}) \backslash \operatorname{GSpin}(V_A)({\bf{A}}_f) /U_A$ so that
\begin{align}\label{GS-decomp} \operatorname{GSpin}(V_A)({\bf{A}}) 
&= \coprod_j \operatorname{GSpin}(V_A)({\bf{Q}}) \operatorname{GSpin}(V_A)({\bf{R}})^0 h_j U_A, \end{align}
where $\operatorname{GSpin}(V_A)({\bf{R}})^0$ denotes the identity component of 
$\operatorname{GSpin}(V_A)({\bf{R}}) \cong \operatorname{GSpin}(2,2)$. 
This gives us the corresponding decomposition of the Shimura variety as 
\begin{equation}\begin{aligned}\label{Sh-decomp} 
\operatorname{Sh}_{U_A}(\operatorname{GSpin}(V_A), D(V_A)) 
&= \coprod_j X_{A, j}, \quad \text{where} ~~ X_{A,j} = \Gamma_j \backslash D^{\pm}(V_A) \end{aligned}\end{equation}
for the arithmetic subgroup  
$ \Gamma_{A, j} = \operatorname{GSpin}(V_A)({\bf{Q}}) \cap \left( \operatorname{GSpin}(V_A)({\bf{R}})^0 h_j U h_j^{-1} \right)$.
Choosing $U_A$ according to $(\ref{level})$ via $(\ref{accidentV_A})$, this simply recovers the 
decomposition $X_{U_A}=\operatorname{Sh}_{U_A}(\operatorname{GSpin}(V_A), D(V_A)) \cong Y_0(N) \times Y_0(N)$.

\subsubsection{Special divisors} 

We now consider the following special divisors on $X_{U_A} = \operatorname{Sh}_{U_A}(D(V_A), \operatorname{GSpin}(V_A))$.
Given a vector $x \in V_A({\bf{Q}})$ with $Q_A(x) > 0$, let $V_{A, x}:= x^{\perp} \subset V_A$ denote the orthogonal complement, with 
\begin{align*} D(V_A)_x  = \lbrace z \in D^{\pm}(V_A): x \perp z \rbrace. \end{align*}
Let $\operatorname{GSpin}(V_{A, x})({\bf{A}}_f)$ denote the stabilizer in $\operatorname{GSpin}(V_A)({\bf{A}}_f)$ of $x$.
We have a natural map defined on $h \in \operatorname{GSpin}(V_A)({\bf{A}}_f)$ by 
\begin{equation}\begin{aligned}\label{natural-x}
\operatorname{GSpin}(V_{A, x})({\bf{Q}}) \backslash D(V_A)_x \times \operatorname{GSpin}(V_{A, x})({\bf{A}}_f) / 
\left( \operatorname{GSpin}(V_{A, x})({\bf{A}}_f) \cap h U_A h^{-1}  \right)
&\longrightarrow X_{U_A} \\ [z, h_1] &\longmapsto [z, h_1 h]. \end{aligned}\end{equation}

\begin{definition} Given $x \in V_A({\bf{Q}})$ with $Q_A(x) > 0$ and $h \in \operatorname{GSpin}(V_A)({\bf{A}}_f)$, let
$Z_A(x,h) = Z_A(x, h, U_A)$ denote the image of the map $(\ref{natural-x})$. Here, we drop the compact open subgroup 
$U_A \subset \operatorname{GSpin}(V_A)({\bf{A}}_f)$ from the notation when the context is clear. \end{definition} 

This image $Z_A(x, h) = Z_A(x, h, U_A)$ determines a special codimension-$1$ cycle on $X_{U_A}$ defined over ${\bf{Q}}$. 
Let us for a given $m \in {\bf{Q}}_{>0}$ write $\Omega_{A, m}({\bf{Q}})$ to denote the hyperboloid  
\begin{align*} \Omega_{A, m}({\bf{Q}}) &= \left\lbrace x \in V_A : Q_A(x) = m \right\rbrace. \end{align*}
If $\Omega_{A, m}({\bf{Q}})$ is not empty, we fix a point $x_0 \in \Omega_{A, m}({\bf{Q}})$.
The corresponding finite adelic points $\Omega_{A, m}({\bf{A}}_f)$ determine a closed subgroup of $V_A({\bf{A}}_f)$. 
Given a Schwartz function $\varphi_f = \otimes_{v < \infty} \varphi_v \in \mathcal{S}(V_A ({\bf{A}}_f))^{U_A}$, we write
\begin{align}\label{support} \operatorname{supp}(\varphi_f) \cap \Omega_{A, m}({\bf{A}}_f) 
&= \coprod_r U_A \cdot \zeta_r^{-1} \cdot x_0 \end{align} 
for some finite set of representatives $\zeta_r \in \operatorname{GSpin}(V_A)({\bf{A}}_f)$. 
Via $(\ref{support})$, we define the analytic divisor
\begin{align}\label{AD} Z_A(\varphi_f, m, U_A) &= \sum_r \varphi_f(\zeta_r^{-1} \cdot x_0) Z_A(x_0, \zeta_r, U_A). \end{align}

\begin{definition}\label{specialdivisor} 

Given a positive rational number $m >0$ for which $\Omega_{A, m}({\bf{Q}}) \neq \emptyset$ and a 
coset $\mu \in L^{\vee}_A/L_A$ with corresponding characteristic function ${\bf{1}}_{\mu}$, 
we write $Z_A(\mu, m) = Z_A( {\bf{1}}_{\mu}, m) = Z_A( {\bf{1}}_{\mu}, m, U_A)$ for the corresponding analytic divisor on
$X_{U_A}= \operatorname{Sh}_{U_A}(\operatorname{GSpin}(V_A), D(V_A)) \cong Y_0(N) \times Y_0(N)$. \end{definition}

\subsubsection{Relation to Hirzebruch-Zagier divisors}\label{HiZa}

Suppose we fix the level $U_A \subset \operatorname{GSpin}(V_A)({\bf{A}}_f)$ as in Corollary \ref{lattices}.
The special divisors $Z_A(\mu, m)$ of Definition \ref{specialdivisor} are then sums of Hirzebruch-Zagier divisors on the product of modular curves
$X_A \cong Y_0(N) \times Y_0(N)$. More explicitly, we have
\begin{equation*}\begin{aligned} &Z_A(\mu, m)({\bf{C}}) \cong \Gamma_0(N)^2 \Big\backslash \coprod\limits_{ x \in \mu + L_A \atop Q_A(x) = m } D(V_A)_x
= \Gamma_0(N)^2 \Big\backslash \coprod\limits_{ x \in \mu + L_A \atop Q_A(x) = m } \left\lbrace z \in D^{\pm}(V_A): (z, x)_A =0 \right\rbrace \\
&\cong  \Gamma_0(N)^2 \Big\backslash \coprod\limits_{ x \in \mu + L_A \atop Q_A(x) = m } \left\lbrace z=(z_1, z_2) \in \mathfrak{H}^{\pm}: (z, x)_A=0 \right\rbrace
 \subset Y_0(N)({\bf{C}}) \times Y_0(N)({\bf{C}}). \end{aligned}\end{equation*}
Note that these special divisors $Z_A(\mu, m)$ can be viewed as embeddings of modular curves into the product $Y_0(N) \times Y_0(N)$.
Indeed, each point in $\Omega_{A, \mu, m}({\bf{Q}}) = \lbrace x \in \mu + L_A: Q_A(x) = m \rbrace$
gives rise to a rational quadratic subspace $W_A = x^{\perp} \subset V_A$ of signature $(1,2)$, with general spin group
$\operatorname{GSpin}(W_A) \subset \operatorname{GSpin}(V_A)$, level $\overline{U}_A = U_A \cap \operatorname{GSpin}(W_A)({\bf{A}}_f)$,
and Grassmannian $D(W_A) \subset D(V_A)$. This determines a modular curve
\begin{align*} C_{\overline{U}_A}:= \operatorname{Sh}_{\overline{U}_A}(D(W_A), \operatorname{GSpin}(W_A), )
\longrightarrow X_{U_A}  \cong Y_0(N) \times Y_0(N). \end{align*}

\begin{remark} Recall the Hirzebruch-Zagier divisor $T_m = T_m(L_A)$ of discriminant  $m$ for $L_A \subset V_A$ is defined by 
\begin{align}\label{HZ} T_m = T_m(L_A) 
&= \sum\limits_{ { \lambda \in L_A^{\vee}/ \lbrace \pm 1 \rbrace \atop Q_A(\lambda) = \frac{m}{\Delta}} }
\left\lbrace z = (z_1, z_2) \in \mathfrak{H}^2: Q_A(z+ \lambda) - Q_A(z) - Q_A(\lambda) = 0 \right\rbrace, \end{align}
where $\Delta = c^2 d_K$ denotes the discriminant of the order $\mathcal{O}_c = {\bf{Z}} + c \mathcal{O}_K$. Hence, we find the relation 
\begin{align*} T_m = T_m(L_A) &= \sum\limits_{ \mu \in L_A^{\vee} / L_A } Z_A(\mu, m/\Delta). \end{align*}
We refer to \cite[Definition 2.27]{Br}, \cite[$\S$3]{HY}, and \cite[$\S$8]{BY} for more background on these Hirzebruch-Zagier divisors. \end{remark}
  
\subsubsection{Arithmetic automorphic forms}

Let $\mathcal{L}_{D(V_A)} = \mathcal{L}_{D^{\pm}(V_A)}$ denote the restriction to $D (V_A) \cong Q(V_A)$ of the tautological bundle on ${\bf{P}}(V_A({\bf{C}}))$. 
The natural action of the orthogonal group $\operatorname{O}(V_A)({\bf{R}})$ on $V_A({\bf{C}})$ induces one of the connected 
component of the identity $\operatorname{GSpin}(V_A)({\bf{R}})^{0}$ of $\operatorname{GSpin}(V_A)({\bf{R}})$
on $\mathcal{L}_{D(V_A)}$. Hence, there is a holomorphic line bundle 
\begin{align*} \mathcal{L}_A= \operatorname{GSpin}(V_A)({\bf{Q}}) 
\backslash \left( \mathcal{L}_{D(V_A)} \times \operatorname{GSpin}(V_A)({\bf{A}}_f) / U_A \right)\longrightarrow X_{U_A}. \end{align*}
Note that $\mathcal{L}_A$ has a canonical model over ${\bf{Q}}$ by \cite{Ha}.  
We define a hermitian metric $h_{ \mathcal{L}_{D(V_A)} }$ on $\mathcal{L}_{D(V_A)}$ by 
\begin{align*} h_{ \mathcal{L}_{ D(V_A) }}(w_1, w_2)_A &:= \frac{1}{2} \cdot (w_1, \overline{w}_2)_A. \end{align*}
This metric is invariant under the action by $\operatorname{O}(V_A)({\bf{R}})$, and hence descends to $\mathcal{L}_A$. 
The map $z \mapsto w(z)$ used to identify $D(V_A) \cong \mathcal{H}_{\pm}(V_A) \cong \mathfrak{H}^2$
can be viewed as a nowhere vanishing section of $\mathcal{L}_{D(V_A)}$ of norm 
\begin{align*} \vert \vert w( z) \vert \vert_A = - \frac{1}{2} \cdot (w(z), \overline{w}(z))_A = - (y, y)_A =: \vert y \vert^2_A. \end{align*}
For $h \in \operatorname{GSpin}(V_A)({\bf{R}})$, we have that $h \cdot w(z) = w(hz) \cdot j(h, z)$ for a holomorphic automorphy factor 
\begin{align*} j: \operatorname{GSpin}(V_A)({\bf{R}}) \times D(V_A) \longrightarrow {\bf{C}}^{\times}. \end{align*} 
In this way, holomorphic sections of $\mathcal{L}_A^{\otimes l}$ for $l \in \frac{1}{2}{\bf{Z}}$ can be viewed as holomorphic functions 
\begin{align*} \Psi: D(V_A) \times \operatorname{GSpin}(V_A)({\bf{A}}_f) \longrightarrow {\bf{C}} \end{align*} 
of $z \in D(V_A)$ and $h \in \operatorname{GSpin}(V_A)({\bf{A}}_f)$ satisfying the transformation properties

\begin{itemize}

\item $\Psi(z, h u) = \Psi(z, h)$ for all $u \in U_A$, 

\item $\Psi(\gamma z, \gamma h) = j(\gamma, z)^l \cdot \Psi(z, h)$ for all $\gamma \in \operatorname{GSpin}(V_A)({\bf{Q}})$.

\end{itemize} 

We define the Petersson norm of a section $(z, h) \rightarrow \Psi(z, h) \cdot w(z)^{\otimes l}$ to be  
$\vert \vert \Psi(z, h) \vert \vert_A^2 = \vert \Psi(z, h) \vert_A^2 \cdot \vert y \vert_A^{2l}$.

\subsection{Regularized theta lifts} 

We now describe the regularized theta lifts for the spaces $(V_A, Q_A)$.

\subsubsection{Gaussian functions}

Given $z \in D(V_A) = D^{\pm}(V_A)$, let $\operatorname{pr}_z : V_A({\bf{R}}) \rightarrow z$ 
denote the projection, whose kernel defines the orthogonal complement $z^{\perp} := \ker(\operatorname{pr}_z)$. 
Given $x \in V_A({\bf{R}})$, we then define the resultant
\begin{align*} R(x, z)_A &:= - \left( \operatorname{pr}_z(x), \operatorname{pr}_z(x) \right) = \left\vert  (x, w(z))_A \right\vert_A^2 \cdot \vert y \vert_A^2.\end{align*}
Using this resultant, we can associate to a hyperplane $z \in D(V_A)$ and vector $x \in V_A({\bf{R}})$ a majorant
\begin{align*} \left( x, x \right)_{A, z} &:= (x, x)_A + 2 \cdot R(x, z)_A. \end{align*}
Writing $\mathcal{C}^{\infty}(D(V_A))$ to denote the space of smooth functions on $D(V_A)$,
we use this majorant to define a Gaussian function 
$\varphi_{\infty}(x, z) \in \mathcal{S}(V_A( {\bf{R}} ) ) \otimes \mathcal{C}^{\infty}(D(V_A))$ by the rule
\begin{align*} \varphi_{\infty}(x, z) &:= \exp \left( - \pi \cdot \left(x, x \right)_{A, z} \right). \end{align*}
Note that $\varphi_{\infty}(h x, h z) = \varphi_{\infty}(x, z)$ for all $h \in \operatorname{GSpin}(V_A)({\bf{R}})$,
and has weight $0$ under the action $\operatorname{SO}_2({\bf{R}})$. 

\subsubsection{Theta kernels}

Given $z \in D(V_A)$, $g \in \operatorname{SL}_2({\bf{A}})$, and $h_f \in \operatorname{GSpin}(V_A)({\bf{A}}_f)$,
let $\theta_{L_A}^{\star}$ denote the linear functional on $\varphi_f \in \mathcal{S}(V_A({\bf{A}}_f))$ defined by
\begin{equation}\begin{aligned}\label{functional} \varphi_f \longmapsto \theta_{L_A}^{\star}(g, z, h_f; \varphi_f ) 
&:= \sum\limits_{x \in V_A({\bf{Q}})} \omega_{L_A}(g) \left( \varphi_{\infty}(\cdot, z) \otimes \omega_{L_A}(h_f) \varphi_f \right) (x) \\
&=  \sum\limits_{x \in V_A({\bf{Q}})} \omega_{L_A}(g,1) \left( \varphi_{\infty}(\cdot, z) \otimes \omega_{L_A}(1,h_f) \varphi_f \right) (x). \end{aligned}\end{equation}
It is easy to see that for all $\gamma \in \operatorname{GSpin}(V_A)({\bf{Q}})$, we have
\begin{align*} \theta_{L_A}^{\star}(g, \gamma z, \gamma h_f; \varphi_f) &= \theta_{L_A}^{\star}(g, z, h_f; \varphi_f). \end{align*}
By Poisson summation (see \cite{We, KuBL}), we have for all $\gamma \in \operatorname{SL}_2({\bf{Q}})$ that 
\begin{align*} \theta_{L_A}^{\star}(\gamma g, z,  h_f; \varphi_f) &= \theta_{L_A}^{\star}(g, z, h_f; \varphi_f). \end{align*}
Using properties of $\omega_{L_A}$, we can also see that for any $g' \in \operatorname{SL}_2({\bf{A}})$ and $h_f' \in \operatorname{GSpin}(V_A)({\bf{A}}_f)$
\begin{align}\label{1.23} \theta_{L_A}^{\star}(gg', z,  h_f h_f'; \varphi_f) 
&= \theta_{L_A}^{\star}(g, z, h_f; \omega_{L_A}(g', h_f') \varphi_f). \end{align}
Hence for any compact open subgroup $U_A \subset \operatorname{GSpin}(V_A)({\bf{A}}_f)$ and $\varphi_f \in \mathcal{S}(V_A({\bf{A}}_f))^U$, the functional 
\begin{align*} (z, h_f) &\longmapsto \theta_{L_A}^{\star}(g, z, h_f; \varphi_f) \end{align*} 
on $(z, h_f) \in D(V_A) \times \operatorname{GSpin}(V_A)({\bf{A}}_f)$ descends to a function on $X_{U_A}$:
\begin{align*} \theta_{L_A}^{\star}: X_{U_A} \times \operatorname{SL}_2({\bf{Q}}) \backslash \operatorname{SL}_2({\bf{A}})
&\longrightarrow \left( \mathcal{S}(V_A({\bf{A}}_f))^{U_A}\right)^{\vee}. \end{align*}


\subsubsection{Regularized theta lifts}

Let \begin{align*} \phi: \operatorname{SL}_2({\bf{Q}}) \backslash \operatorname{SL}_2({\bf{A}}) 
&\longrightarrow \mathcal{S}(V_A({\bf{A}}_f))^{U_A} \end{align*}
be any function such that for each $g \in \operatorname{SL}_2({\bf{A}})$, $k_{\infty} \in \operatorname{SO}_2({\bf{R}})$, and $k \in \mathcal{K}$,
\begin{align*} \phi(g k k_{\infty}) &= \omega_{L_A}(k)^{-1} \cdot \phi(g). \end{align*}
It is then easy to check that the ${\bf{C}}$-linear pairing $\left\{ \cdot, \cdot \right\} $ defined as a function on $g \in \operatorname{SL}_2({\bf{A}})$ by the rule
\begin{align*} \left\{ \phi(g), \theta_{L_A}^{\star}(z, h_f, g) \right\} &:= \theta_{ L_A }^{\star}(z, h_f, g; \phi(g)) \end{align*}
is both left $\operatorname{SL}_2({\bf{Q}})$-invariant and right $\mathcal{K} \operatorname{SO}_2({\bf{R}})$-invariant. 
We can then consider the regularized theta lift
\begin{align*}\Phi(\phi, z, h_f)  &= \int_{ \mathcal{F} }^{\star} \left\{ \phi(g), \theta_{L_A}^{\star}(g, z, h_f) \right \} dg
= \int_{ \mathcal{F}}^{\star} \theta_{L_A}^{\star}(g, z, h_f;  \phi(g)) dg, \end{align*}
as a function on $(z, h) \in X_{U_A}$. To describe these more explicitly, we descend via Iwasawa decomposition 
Hence, fixing the standard fundamental domain 
$\mathcal{F} = \lbrace \tau = u + iv \in \mathfrak{H}: \vert \Re(\tau) \vert \leq 1/2, \tau \overline{\tau} \geq 1 \rbrace$ 
for the action of $\operatorname{SL}_2({\bf{Z}})$ on $\mathfrak{H}$, each adelic matrix 
$g \in \operatorname{SL}_2({\bf{A}})$ can be expressed uniquely as a product 
\begin{align}\label{UDSL2} g &= \gamma \cdot \left( \begin{array}{cc} 1 & u \\~& 1 \end{array} \right) 
\cdot \left( \begin{array}{cc} v^{\frac{1}{2}} & ~~\\~& v^{- \frac{1}{2}}\end{array} \right) \cdot k \end{align}
for some $\gamma \in \operatorname{SL}_2({\bf{Q}})$, $\tau = u+ i v \in \mathcal{F}$, and $k \in \operatorname{SO}_2({\bf{R}})$. 
Taking the decomposition $(\ref{UDSL2})$ for granted, 
let us define for a given $g \in \operatorname{SL}_2({\bf{A}})$ the corresponding mirabolic matrix 
\begin{align*} g_{\tau} 
&:= \left( \begin{array}{cc} 1 & u \\ ~& 1\end{array} \right) \left( \begin{array}{cc} v^{\frac{1}{2}} & ~ \\~& v^{- \frac{1}{2}}\end{array} \right). \end{align*}
We define the Siegel theta series $\theta_{L_A}(\tau, z, h)$ on $\tau = u + iv \in \mathfrak{H}$, 
$z \in D(V_A)$, and $h \in \operatorname{GSpin}(V_A)({\bf{A}}_f)$ by
\begin{align}\label{classicalsiegeltheta} \theta_{L_A}(\tau, z, h) 
&= \sum\limits_{\mu \in L_A^{\vee}/L_A} \theta_{L_A, \mu}(\tau, z, h) {\bf{1}}_{\mu}, 
\quad \theta_{L_A, \mu}(\tau, z, h) = \theta_{L_A}^{\star}(g_{\tau}, z, h; {\bf{1}}_{\mu}).\end{align}
Given a weight-zero $L^2$-automorphic form $\phi$ on $\operatorname{SL}_2({\bf{Q}}) \backslash \operatorname{SL}_2({\bf{A}})$,
let $f(\tau) := \phi(g_{\tau})$ to denote the corresponding classical weight-zero Maass form on $\tau = u + i v \in \mathfrak{H}$. 
Writing $\mathcal{F}$ again to denote the standard fundamental domain for the action of $\operatorname{SL}_2({\bf{Z}})$ on $\mathfrak{H}$, 
we define the regularized integral as above
\begin{align*} \Phi(f, z, h) = \int_{\mathcal{F}}^{\star}\left( f(\tau), \theta_{L_A}^{\star}(g_{\tau}, z, h_f) \right) d \mu(\tau) 
&= \operatorname{CT}_{s=0} \left( \varinjlim_T \int_{\mathcal{F}_T} \left\{ f(\tau), \theta_{L_A}^{\star}(g_{\tau}z, h_f) \right \}  v^{-s} d \mu(\tau) \right) \\
&= \operatorname{CT}_{s=0} \left( \varinjlim_T \int_{\mathcal{F}_T} \theta_{L_A}^{\star}(g_{\tau}, z, h_f; f(\tau)) v^{-s} d \mu(\tau) \right) \\
&= \operatorname{CT}_{s=0} \left( \varinjlim_{T\rightarrow \infty} \int\limits_{\mathcal{F}_T} \langle \langle f(\tau), \theta_{L_A}(\tau, z, h) \rangle \rangle v^{-s} d \mu(\tau) \right). \end{align*}
Again, we write $d \mu(\tau) = du dv/v^2$ for the Poincar\'e measure, $\mathcal{F}_T = \lbrace \tau = u+iv \in \mathcal{F}: v \leq T \rbrace$ for the truncated fundamental domain,
and $\operatorname{CT}_{s=0} F(s)$ for the constant term in the Laurent series around $s=0$ of $F(s)$.

\subsubsection{Harmonic weak Maass forms}\label{HWMF}

Suppose $l \in \frac{1}{2}{\bf{Z}}$ is any half-integer weight. (We shall later take $l=0$).
Let $\vert_{l, \omega_{L_A}}$ denote the Petersson weight $l$ operator with respect to $\omega_{L_A}$,  
defined on a function $f: \mathfrak{H}\rightarrow {\bf{C}}$ by   
\begin{align*} f\vert_{l, \omega_{L_A}}(\gamma(\tau)) 
&= (c \tau + d)^l \cdot \omega_{L_A}(\gamma) \cdot f(\tau) 
\quad \text{ for all $\gamma = \left(\begin{array}{cc} a & b \\ c & d\end{array} \right) \in \operatorname{SL}_2({\bf{Z}})$}. \end{align*}
Let $\Delta_l$ denote the hyperbolic Laplacian of weight $l$, defined for $\tau = u + iv \in \mathfrak{H}$ by 
\begin{align*} \Delta_l  
&:= - v^2 \left( \frac{\partial^2}{\partial u^2} + \frac{\partial^2}{\partial v^2} \right) + i l \left( \frac{\partial}{\partial u} + i \frac{\partial}{\partial v}\right). \end{align*}
Note that this Laplacian can be expressed in terms of the respective weight $l$ 
Maass weight raising and lowering operators $R_l$ and $L_l$ 
as $- \Delta_l = L_{l+2} R_l + l = R_{l-2} L_l,$ where 
\begin{align}\label{Rl} R_l &= 2 i \cdot \frac{\partial}{ \partial \tau} + l \cdot v^{-1} \end{align}
denotes the Maass weight raising operator of weight $l$ (which raises the weight by $2$), 
and 
\begin{align}\label{Ll}  L_l &= - 2 i v^2 \cdot \frac{\partial}{ \partial \overline{\tau}} \end{align}
denotes the Maass lowering operator (which lowers the weight $l$ by $2$).
 
\begin{definition} Fix a half-integer weight $l \in \frac{1}{2}{\bf{Z}}$ with $l \leq 1$, and an integral lattice $L_A \subset V_A$.
Let $\mathcal{S}_{L_A} \subset \mathcal{S}(V_A ({\bf{A}}))$ denote the subspace of Schwartz-Bruhat functions supported on $L_A^{\vee} \otimes \widehat{\bf{Z}}$
but trivial on $L_A \otimes \widehat{\bf{Z}}$. A twice differentiable function $f: \mathfrak{H} \longrightarrow \mathcal{S}_{L_A}$ is a said to be a
{\it{harmonic weak Maass form of weight $l$ with respect to $\Gamma = \operatorname{SL}_2({\bf{Z}})$ and representation $\omega_{L_A}$}} if it satisfies the following conditions. \\ 

\begin{itemize}

\item[(i)] The function is invariant under the Petersson weight-$l$ operator: $f \vert_{l, \omega_{L_A}} \gamma = f$ for all $\gamma \in \Gamma$. \\

\item[(ii)] There exists an $\mathcal{S}_{L_A}$-valued Fourier polynomial 
\begin{align*} P_f(\tau) &= \sum\limits_{\mu \in L_A^{\vee}/L_A} 
\sum_{m \leq 0} c_f^+(\mu, m) e(m \tau) {\bf{1}}_{\mu} \end{align*}
such that $ f(\tau) = P_f(\tau) + O(e^{- \varepsilon v})$ as $v = \Im(\tau) \rightarrow \infty$ for some $\varepsilon >0$. \\

\item[(iii)] The function is harmonic of weight $l$, i.e.~$\Delta_l f = 0$. \\

\end{itemize} We write $H_l(\omega_{L_A})$ for the space of such forms, and call $P_f(\tau)$ the holomorphic or principal part of $f$.

\end{definition}

\begin{definition} We let $\overline{\omega}_{L_A}$ denote the conjugate Weil representation on $\mathcal{S}_{L_A}$, 
hence $\overline{\omega}_{L_A, \psi}(k_{\gamma}) = \omega_{-L_A, \psi}(g_{\gamma})$ 
for each $\gamma \in \Gamma = \operatorname{SL}_2({\bf{Z}})$ and its corresponding diagonal image 
$k_{\gamma} \in \mathcal{K} = \operatorname{SL}_2(\widehat{\bf{Z}})$, cf.~\cite[(2.7)]{BY}.  \end{definition}

\subsubsection{Classical description and characterization as automorphic Green's functions} 

We now return to the setup above, with $L_A \subset V_A$ an integral lattice of signature $(2,2)$. Hence, the Siegel theta series 
\begin{align*} \theta_{L_A}(\tau, z, h_f): \mathfrak{H} \times D^{\pm}(V_A) \longrightarrow \mathcal{S}_{L_A}\end{align*}
defined for each $h = h_f \in \operatorname{GSpin}(V_A)({\bf{A}}_f)/U_A$ by
\begin{align*} \theta_{L_A}(\tau, z, h) &= \sum\limits_{ \mu \in L_A^{\vee}/L_A }\theta_{L_A}^{\star}(z, h, g_{\tau}; {\bf{1}}_{\mu})\end{align*}
is a nonholomorphic $\Gamma_h$-invariant function in $z \in D(V_A)$.
As a function in $\tau \in \mathfrak{H}$, it is a nonholomorphic Maass form of weight $0$
and representation $\overline{\omega}_{L_A}$, so $\theta_{L_A}(\tau, \cdot) \in H_0(\overline{\omega}_{L_A})$.
Given $f_0 \in H_{0}(\omega_{L_A})$ a harmonic weak Maass form of weight $-l=0$ and representation $\omega_{L_A}$, we consider the regularized theta lift 
\begin{align*} \Phi(f_0, z, h) 
&= \int_{\mathcal{F}}^{\star} \langle \langle f_0(\tau), \theta_{L_A} (\tau, z, h) \rangle \rangle d \mu(\tau)
= \operatorname{CT}_{s=0} \left(\lim_{T \rightarrow \infty} \int_{\mathcal{F}_T} \langle \langle f_0(\tau), \theta_{L_A} (\tau, z, h) \rangle \rangle v^{-s} d \mu(\tau) \right). \end{align*}

Let \begin{align*} f_{0, A}(\tau) &= \sum\limits_{\mu \in L_A^{\vee}/L_A} 
\sum\limits_{m \in {\bf{Q}} \atop m \gg -\infty} c_{f_{0, A}}(\mu, m) e(m \tau) {\bf{1}}_{\mu}
\in M_{0}^{!}(\omega_{L_A}) \cong \ker(\xi_0)\end{align*} be a weakly holomorphic form with 
integer Fourier coefficients $c_{f_{0, A}}(\mu, m) \in {\bf{Z}}$. The theorem of Borcherds \cite[Theorem 13.3]{Bo} (cf.~\cite[Theorem 1.2]{KuBL})
shows there exists a meromorphic modular form $\Psi( f_{0, A}, z, h)$ on $X_A = X_{U_A}$ of weight $l = \frac{c^+_{f_{0, A}}(0, 0)}{2}$ and divisor 
\begin{align*} \operatorname{Div}(\Psi_{f_{0, A}}) 
&= Z(f_{0, A}) = \sum\limits_{\mu \in L_A^{\vee}/L_A} \sum\limits_{m \in {\bf{Q}}_{>0}} c_{f_{0, A}}(\mu, -m) \cdot Z_A(m, \mu) \end{align*}
related to the regularized theta lift $\Phi(f_{0, A}, z, h)$ by 
\begin{align*} \Phi( f_{0, A}, z, h) 
&= - 2 \log \vert  \Psi( f_{0, A}, z, h) \vert_A^2 - c_{f_{0, A}}(0, 0) \cdot \left( 2 \log \vert y \vert_A + \Gamma'(1) \right) \end{align*}
Moreover, Howard-Madapusi Pera \cite[Theorem 9.1.1]{HMP} shows that this meromorphic modular forms  
$\Psi( f_{0, A}, z, h)$ takes algebraic values, so that the regularized theta lift $\Phi( f_{0, A}, z, h)$ attached to any 
$f_{0, A} \in M^!_{0}(\omega_{L_A})$ takes values in logarithms of algebraic numbers, and hence in the ring of periods (see \cite{KZ}). 
We have the following generalization for $f_{0, A} \in H_{0}(\omega_{L_A})$ a harmonic weak Maass form which is not necessarily weakly holomorphic.

\begin{theorem}[Borcherds, Bruinier]\label{Green}

Let $f_{0, A} \in H_{0}(\omega_{L_A})$ be a harmonic weak Maass form of weight $0$ and representation $\omega_{L_A}$ whose holomorphic part 
\begin{align*} f_{0, A}^+(\tau) 
&= \sum\limits_{\mu \in L_A^{\vee}/L_A} \sum\limits_{m \in {\bf{Q}} \atop m \gg - \infty} c_{f_{0,A}}^+(\mu, m) e(m \tau) {\bf{1}}_{\mu}\end{align*}
has integer Fourier coefficients $c_{f_{0,A}}^+(\mu, m) \in {\bf{Z}}$. Consider the special divisor defined by
\begin{align*} Z(f_{0, A}) &= \sum\limits_{\mu \in L_A^{\vee}/L_A} \sum\limits_{m \in {\bf{Q}} \atop m >0} 
c^+_{f_{0,A}}(\mu, -m) Z_A(\mu, m) \subset X_{U_A}. \end{align*}
The regularized theta lift $\Phi(f_{0,A}, z, h)$ is a smooth function on $ X_{U_A} \backslash Z(f_{0,A})$, 
with a logarithmic singularity along the divisor $- 2 Z(f_{0, A})$. Moreover: \\

\begin{itemize}

\item The $(1, 1)$ form $d d^c \Phi( f_{0, A}, z, h)$ has an analytic continuation to a smooth form on $X_{U_A}$, and satisfies the Green 
current equation $d d^c [\Phi( f_{0, A}, z, h)] + \delta_{Z(f_{0, A})}= [d d^c \Phi( f_{0, A}, z, h)].$
Here, $\delta_{Z(f_{0, A})}$ denotes the Dirac current of the divisor $Z(f_{0, A})$. \\

\item The regularized theta lift $\Phi( f_{0, A}, z, h)$ is an eigenfunction for the generalized Laplacian operator $\Delta_z$
defined on $z \in D(V_A)$, with eigenvalue $c_{f_{0, A}}^+(0, 0)/2$. \\

\end{itemize} In particular, the regularized theta lift $\Phi( f_{0, A}, \cdot)$ gives the automorphic Green's function $G_{Z(f_{0, A})}$ for the divisor $Z(f_{0, A})$, 
making it an arithmetic divisor $\widehat{Z}(f_{0, A}) = (Z(f_{0, A}), \Phi( f_{0, A}, \cdot))$ on the spin Shimura surface $X_{U_A}$. \\

\end{theorem}

\begin{proof} See \cite[Theorems 4.2 and 4.3]{BY} and \cite{Br}, as well as \cite[Proposition 5.6, Theorem 6.1, Theorem 6.2]{BF}. 
The special case of $f_{0,A}$ weakly holomorphic is due to Borcherds \cite{Bo}. \end{proof}

\subsection{Choice of harmonic weak Maass form} 

We choose the Maass form $f_{0, A} \in H_{0}(\omega_{L_A})$ so that the holomorphic cuspidal form $g_{A} = \xi_0(f_{0, A}) \in S_{2}(\overline{\omega}_{L_A})$
is the canonical lift in the sense of Theorem \ref{YZhang} below of the eigenform $f \in S_2(\Gamma_0(N))$. 
Here again, $f \in S_2(\Gamma_0(N))$ denotes the cuspidal newform parametrizing $E/{\bf{Q}}$.
We assume $(N, d_K)=1$. We then have the following relation to scalar-valued forms (cf.~\cite[$\S$3]{BY}). 

\begin{theorem}\label{YZhang} 

Let $(V_A, Q_A)$ be the quadratic space of signature $(2,2)$ defined in $(\ref{V_A})$.
Let $L_A \subset V_A$ be the lattice associated to the compact open subgroup $U_A \cong K_0(N)^2$ of 
$\operatorname{GSpin}(V_A)({\bf{A}}_f) \cong  \operatorname{GL}_2({\bf{A}}_f) \times_{{\bf{G}}_m} \operatorname{GL}_2({\bf{A}}_f)$ described Proposition \ref{spinID} and Corollary \ref{lattices}.
Write the Fourier series expansion of $f \in S_2^{\operatorname{new}}(\Gamma_0(N))$ as 
\begin{align*} f (\tau) &= \sum\limits_{m \geq 1} c_f(m) e(m \tau). \end{align*}
There exists an $\mathcal{S}_{L_A}$-valued modular form $g = g_{f, A}$ of weight $2$, determined 
canonically as the lifting of $f$ defined in \cite{YZ}, whose Fourier series expansion is given by 
\begin{align*} g(\tau) = g_{f, A}(\tau) &= \sum\limits_{ \mu \in L_A^{\vee}/L_A} g_{\mu}(\tau) {\bf{1}}_{\mu}, \end{align*}
where 
\begin{align*} g_{\mu}(\tau) = \sum\limits_{ m \in {\bf{Q}} \atop m \equiv N Q_A(\mu) \bmod (N) } 
c_f(m) s(m) e \left( \frac{m \tau}{ N} \right). \end{align*}
Here, $s(m)$ denotes the function defined on each class $m \bmod N$ by $s(m) = 2^{\varOmega(m, N)}$, 
where $\varOmega(m, N)$ denotes the number of prime divisors of the greatest common divisor $(m, N)$. \end{theorem} 

\begin{proof} This is a special case of \cite[Theorem 4.15]{YZ}, adapted to match the setup of \cite[p.~639, Lemma 3.1]{BY}. 
See also the more general theorem of Str\"omberg \cite[Theorem 5.2]{Str}.\end{proof}

Observe from the Fourier series expansion described in Theorem \ref{YZhang} above that $c^+_{f_{0, A}}(0,0)=0$
and hence that the corresponding regularized theta lift $\Phi( f_{0, A}, \cdot )$ is annihilated by the Laplacian $\Delta_z$. 

\subsection{Langlands Eisenstein series and the Siegel-Weil formula}

We now record the special case of the Siegel-Weil formula we use for later calculations; see \cite[Theorem 4.1]{KuBL}. 
Recall we introduced the anisotropic subspaces $(V_{A, j}, Q_{A, j})$ of signature $(1,1)$.
Let us temporarily write $(V_0, Q_0)$ to denote the ambient space $(V_A, Q_A)$ of signature $(2, 2)$, so that $(V_j, Q_j)$ for $j=0,1,2$ denotes any of these three spaces. 
In each case, we write $\omega_{j} = \omega_{L_j}$ to denote the Weil representation 
\begin{align*} \omega_{L_A}: \operatorname{SL}_2({\bf{A}}) \times \operatorname{GSpin}(V_A)({\bf{A}}) \longrightarrow \operatorname{Aut} \left(  \mathcal{S}(V_A({\bf{A}})) \right), \end{align*}
with $\theta_{L_j}$ the corresponding theta kernel defined on $g \in \operatorname{SL}_2({\bf{A}})$, 
$h \in \operatorname{GSpin}(V_j)({\bf{A}})$, and $\varphi \in \mathcal{S}(V_j ({\bf{A}}) )$ by 
\begin{align*} \theta_{L_j}(g, h; \varphi) &= \sum\limits_{x \in V_j( {\bf{Q}} ) } \omega_{j}(g, h) \varphi(x). \end{align*}
Let $\mathcal{K} = \operatorname{SL}_2(\widehat{\bf{Z}})$ denote the maximal compact subgroup of $\operatorname{SL}_2({\bf{A}}_f)$
and $\mathcal{K}_{\infty} = \operatorname{SO}_2({\bf{R}})$ the maximal compact subgroup of $\operatorname{SL}_2({\bf{R}})$. 
We then write the Iwasawa decomposition as
\begin{align}\label{Iwasawa} \operatorname{SL}_2({\bf{A}}) = N({\bf{A}}) M({\bf{A}}) \mathcal{K} \mathcal{K}_{\infty}, \end{align} for 
\begin{align*} N = \left\lbrace n(b): b \in {\bf{G}}_a \right\rbrace, \quad n(b) = \left( \begin{array}{cc} 1 & b \\ ~& 1 \end{array} \right) \end{align*}
and \begin{align*} M = \left\lbrace m(a): a \in {\bf{G}}_m \right\rbrace, 
\quad m(a) = \left( \begin{array}{cc} a & 0 \\ 0 & a^{-1}\end{array} \right). \end{align*} 
Let $\chi_{V_j}$ denote the idele class character of ${\bf{Q}}$ defined on $x \in {\bf{A}}^{\times} / {\bf{Q}}^{\times}$ by  
\begin{align*} \chi_{V_j} (x) = (x, (-1)^{ \frac{ d(j)(d(j)-1) }{2}  } \det(V_j))_{\bf{A}}), \end{align*}
where $(\cdot, \cdot)_{\bf{A}}$ denotes the Hilbert symbol on ${\bf{A}}$, $d(j) = \dim(V_j)$, and $\det(V_j)$ the Gram determinant. 
Let $(p(V_j), q(V_j))$ denote the signature of the space $V_j$.
Given $s \in {\bf{C}}$, let $I(s, \chi_{V_j})$ denote the corresponding principal series representation 
of $\operatorname{SL}_2({\bf{A}})$ induced by the quasi-character $\chi_{V_j} \vert \cdot \vert^s$. 
This consists of all smooth decomposable functions 
$\phi(g, s)$ on $g \in \operatorname{SL}_2({\bf{A}})$ and $s \in {\bf{C}}$ satisfying 
\begin{align*} \phi( n(b)m(a) g, s ) &= \chi_{V_j}(a) \vert a \vert^{s+1} \phi(g, s)   \end{align*} 
for all $b \in {\bf{A}}$, $a \in {\bf{A}}^{\times}$, and $g \in \operatorname{SL}_2({\bf{A}})$. 
Let $s_0(V_j) := \dim(V_j)/2 - 1$; there is an $\operatorname{SL}_2({\bf{A}})$-intertwining map 
\begin{align*} \lambda: \mathcal{S}(V_j({\bf{A}})) &\longrightarrow I(s_0(V_j), \chi_{V_j}), \quad 
\varphi \mapsto \lambda(\varphi)(g) := (\omega_j(g) \varphi)(0). \end{align*}

Recall that a section $\phi = \phi(g, s) \in I(s, \chi_{V_j})$ is called {\it{standard}} if its restriction to the maximal compact subgroup 
$\mathcal{K} \mathcal{K}_{\infty}$ does not depend on $s \in {\bf{C}}$. 
Given any standard section $\phi \in I(s, \chi_{V_j})$, we consider the corresponding 
Eisenstein series defined for $s\in {\bf{C}}$ with $\Re(s) \gg 1$ by the summation
\begin{align*} E(g, s; \phi ) = E_{L_j}(g, s; \phi) 
&= \sum\limits_{ \gamma \in P( {\bf{Q}} ) \backslash \operatorname{SL}_2({\bf{Q}})} \phi(\gamma g, s). \end{align*}
This Eisenstein series converges absolutely for $\Re(s) >1$. It has a meromorphic continuation $E^{\star}(g,s; \phi)$
to all $s \in {\bf{C}}$, and satisfies the Langlands functional equation $E^{\star}(g, s; \phi) = \pm E^{\star}(g, -s; M \phi)$ 
for $M$ the unipotent intertwining operator (see e.g.~\cite[$\S 3$]{Bu}). 
Observe that via the Iwasawa decomposition $(\ref{Iwasawa})$, the image $\lambda(\varphi) \in I(s_0(V_j), \chi_{V_j})$ 
has a unique extension to a standard section $\lambda(\varphi, s) \in I(s, \chi_{V_j})$ for which $\lambda(\varphi, s_0(V_j)) = \lambda(\varphi)$.
Let $dh$ denote the Tamagawa Haar measure on $\operatorname{SO}(V_j)({\bf{A}})$. 

\begin{theorem}[Siegel-Weil]\label{SW-abstract} 

Let $(V_j, Q_j)$ for $j=0, 1,2$ denote any of the quadratic spaces introduced above. 
We have for any $g \in \operatorname{SL}_2({\bf{A}})$ and decomposable Schwartz function
$\varphi \in \mathcal{S}(V_j({\bf{A}}))$ the average formula
\begin{align*} \frac{\kappa}{2} \cdot \int_{ \operatorname{SO}(V_j)({\bf{Q}}) \backslash \operatorname{SO}(V_j)({\bf{A}}) } 
\theta_{L_j}(h, g; \varphi) dh 
&= E_{L_j}(g, s_0, \lambda(\varphi)), \end{align*} 
where $dh$ denotes the Tamagawa Haar measure on $\operatorname{SO}(V_j)({\bf{A}})$, and 
\begin{align*} \kappa &= \begin{cases} 1 &\text{ if $\dim(V_j) > 2$} \\ 2 &\text{ if $\dim(V_j) \leq 2$}  \end{cases}
\quad \text{and} \quad s_0 = s_0(V_j) = \frac{\dim(V_j)}{2} -1. \end{align*}
Moreover, the Eisenstein series $E_{L_j}(g, s, \lambda(\Phi))$ in each case $j=0,1,2$ is holomorphic at $s = s_0$. \end{theorem}

\begin{proof} See \cite[Theorem 4.1]{KuBL}, and more generally \cite[$\S$ I.4]{Ku2}. \end{proof}

Let us now consider the following more explicit version of Theorem \ref{SW-abstract}. 
We first describe the theta kernel $\theta_{L_j}$ and Eisenstein series $E_{L_j}$ in terms of vector-valued modular forms. 
Following \cite[$\S$ 2.1]{BY}, we can for any integer weight $l \in {\bf{Z}}$ consider the
unique standard section $\Phi_{\infty}^l(s) \in I_{\infty}(s, \chi_{V_0})$ for which 
\begin{align}\label{IDC} \Phi_{\infty}^l( k(\theta), s) = \exp(i l \theta), \quad 
k(\theta) = \left( \begin{array}{cc} \cos \theta & -\sin \theta \\ \sin \theta & \cos \theta \end{array} \right) 
\in \mathcal{K}_{\infty} = \operatorname{SO}_2( {\bf{R}}). \end{align}
In terms of the Iwasawa decomposition $(\ref{Iwasawa})$, this section also satisfies the transformation property 
\begin{align}\label{Iw} \Phi_{\infty}^l( n(b) m(a) k(\theta), s) &= \chi_{V_0}(a) \vert a \vert^{s+1} \exp(i l \theta)  \end{align} 
for all $n(b) \in N_2({\bf{A}})$, $m(a) \in M_2({\bf{A}})$, and $k(\theta) \in \operatorname{SO}_2({\bf{R}})$ when $j=1,2$. 
We shall use the same notation to denote the restriction to each of the subspaces $\Phi_{\infty}^l = \Phi_{\infty}^l(s) \in I(s, \chi_{V_j})$.
As explained in \cite[(2.15)]{BY}, we deduce from our definition of the Gaussian  
$\varphi_{\infty} \in \mathcal{S}(V_0({\bf{R}})) \otimes C^{\infty}( D(V_0))$ that we have the relation(s)
\begin{align}\label{Gaussian} \lambda_{\infty}(\varphi_{\infty}) = \lambda_{\infty}(\varphi_{\infty}(\cdot, z)) 
= \Phi_{\infty}^{ \frac{ p(V_j) - q(V_j) }{2} }(s_0(V_j)) = \Phi_{\infty}^0(1) \in I_{\infty}(1, \chi_{V_j}). \end{align}
Here again, $(p(V_j), q(V_j))$ denotes the signature of $V_j$. 
We remark that each of the quadratic spaces $V_j$ we consider  
leads to looking at an Eisenstein series of weight $l = l(V_j) = (p(V_j) - q(V_j))/2 = 0$.
We know that $(\ref{Gaussian})$ has a unique extension to a standard section 
$\Phi_{\infty}^0(s) \in I_{\infty}(s, \chi_{V_j})$ so that $\Phi_{\infty}^0(s_0(V_j)) = \lambda_{\infty}(\varphi_{\infty})$. 

Given any integral lattice $L_j \subset V_j$, and writing $\lambda_f$ to denote the finite component 
of the standard section $\lambda(\Phi) = \lambda(\Phi, s) \in I(s, \chi_{V_j})$ described above, 
we consider the corresponding $\mathcal{S}_{L_j}$-valued Eisenstein series of 
weight $k = 0$ defined on $\tau = u + iv \in \mathfrak{H}$ and $s \in {\bf{C}}$ by 
\begin{align*} E_{L_j}(\tau, s; 0) &:= \sum\limits_{\mu \in L_j^{\vee}/L_j} 
E_{L_j}( g_{\tau}  , s; \Phi_{\infty}^0 \otimes \lambda_f( {\bf{1}}_{\mu } )) \cdot {\bf{1}}_{\mu}. \end{align*} 
We again consider the $\mathcal{S}_{L_j}$-valued theta function defined on 
$\tau \in \mathfrak{H}$, $z \in D(V_j)$, and $h \in \operatorname{GSpin}(V_j)({\bf{A}}_f)$ by  
\begin{align*}\theta_{L_j}(\tau, z, h) 
&:= \sum\limits_{\mu \in L_j^{\vee}/ L_j} \theta_{L_j}^{\star}(g_{\tau}, z, h_f; {\bf{1}}_{\mu}) \cdot {\bf{1}}_{\mu}. \end{align*}

\begin{theorem}[Siegel-Weil for $\mathcal{S}_{L_j}$-valued forms]\label{SW-vector} 

We have the identification of functions of $\tau \in \mathfrak{H}$:
\begin{align*} \frac{\kappa}{2} \cdot \int_{ \operatorname{SO}(V_j)({\bf{Q}}) \backslash \operatorname{SO}(V_j)({\bf{A}})}
\theta_{L_j}(\tau, z, h_f) dh &= E_{L_j}(\tau, s_0, k) = E_{L_j}(\tau, s_0(V_j); k(V_j)). \end{align*}
Here again, $s_0 = s_0(V_j) := \dim(V_j)/2 -1$, and $k = k(V_j) := (p(V_j) - q(V_j))/2 = 0$. \end{theorem}

\begin{proof} Cf.~\cite[Proposition 2.2]{BY}, and note that we deduce this from Theorem \ref{SW-abstract} with $(\ref{IDC})$ and $(\ref{Gaussian})$. \end{proof}

\subsection{Eisenstein series and Maass operators}

We now give the following more classical descriptions of the Eisenstein series appearing in Theorem \ref{SW-vector}, 
with relations to the Maass raising and lowering operators $R_l, L_l$. We use the same conventions with the three spaces $(V_j, q_j)$, $j=0,1,2$ as above.

\subsubsection{Classical description}

Consider the matrix $g_{\tau}$ for $\tau = u + iv \in \mathfrak{H}$ from the Iwasawa decomposition $(\ref{UDSL2})$ and $(\ref{Iwasawa})$. 
Following the discussion in \cite[$\S$ 2.2]{BY} and \cite{KY}, we consider elements of $\operatorname{SL}_2({\bf{A}})$ of the form 
\begin{align*} \gamma \cdot g_{\tau} &= n(\beta) \cdot m(\alpha) \cdot k(\theta) 
\quad  \gamma = \left( \begin{array}{cc} a & b \\ c & d \end{array} \right) \in \Gamma = \operatorname{SL}_2({\bf{Z}}), \end{align*}
with $\beta \in {\bf{R}}$, $\alpha \in {\bf{R}}_{>0}$, and $k(\theta) \in \operatorname{SO}_2({\bf{R}})$.
A direct calculation shows that  
\begin{align*} \alpha = v^{\frac{1}{2}} \cdot \vert c \tau  + d \vert^{-1}, \quad \exp(i \theta) = \frac{c \overline{\tau} + d }{\vert c \tau + d \vert}, \end{align*}
so that substituting into $(\ref{Iw})$ for any weight $l \in \frac{1}{2}{\bf{Z}}$ gives us
\begin{align*} \Phi_{\infty}^l(\gamma g_{\tau}, s) &= v^{\frac{s}{2} + \frac{1}{2}} (c \tau + d)^{-l} \vert c \tau + d \vert^{l-s-1}.\end{align*} 
Hence, writing $\Gamma_{\infty} = P({\bf{Q}}) \cap \Gamma$ for $\Gamma = \operatorname{SL}_2({\bf{Z}})$ as above, we find that
\begin{align*} E_{L_2}(g_{\tau}, s; \Phi_{\infty}^l \otimes \lambda_f({\bf{1}}_{\mu})) 
&= \sum\limits_{ \gamma \in \Gamma_{\infty} \backslash \Gamma}
(c \tau + d)^{-l} \frac{ v^{\frac{s}{2} + \frac{1}{2}}}{ \vert c \tau + d \vert^{s + 1 - l} } \cdot \lambda_f( {\bf{1}}_{\mu} )(\gamma) \\ 
&= \sum\limits_{ \gamma \in \Gamma_{\infty} \backslash \Gamma}
(c \tau + d)^{-l} \frac{ v^{\frac{s}{2} + \frac{1}{2}}}{ \vert c \tau + d \vert^{s + 1 - l} } 
\cdot \langle {\bf{1}}_{\mu }, (\omega_{L_j}^{-1}(\gamma) {\bf{1}}_{0} ) \rangle, \end{align*}
where $\langle \cdot, \cdot \rangle$ here denotes the $L^2$ inner product on $\mathcal{S}_{L_j}$. 
In this way, the vector-valued Eisenstein series we considered above can be written explicitly as 
\begin{align}\label{2.17} E_{L_j}(\tau, s; l) 
&= \sum\limits_{\gamma \in \Gamma_{\infty} \backslash \Gamma} \left[ \Im(\tau)^{\frac{(s +1 - l)}{2}} {\bf{1}}_0 \right]  \bigg\rvert_{l, \omega_{L_j}} \gamma. \end{align}

\subsubsection{Eisenstein series associated to the anisotropic subspaces $(L_{A,2}, Q_{A,2})$} 

Let us now study the Eisenstein series associated to the lattices $L_{A,2} = L_A \cap V_{A, 2}$ 
in the anisotropic subspaces $(V_{A, 2}, Q_{A, 2})$ of signature (1,1). 
Recall these have weight zero; we write $E_{L_{A,2}}(\tau, s) = E_{L_{A,2}}(\tau, s; 0)$ for simplicity.
Writing $\mathfrak{d}_{K}$ to denote the different of $\mathcal{O}_K$ with inverse different 
$\mathfrak{d}_{K}^{-1} = \left\lbrace \lambda \in \mathcal{O}_k: \operatorname{Tr}(\lambda \mathcal{O}_K) \subset {\bf{Z}} \right\rbrace$,
we have $L_{A,2}^{\vee} \cong \mathfrak{d}_{K}^{-1} \cap L_{A,2}$ and $L_{A,2}^{\vee}/L_{A,2} \cong (\mathfrak{d}_{K}^{-1} \cap L_{A,2} )/L_{A,2}$. 
We can also identify the quadratic character $\chi_{V_{A,2}}=\left( \frac{d(V_{A,2})}{\cdot} \right) = (\frac{d_K}{\cdot})$ with the quadratic 
Dirichlet character $\eta = \eta_K(\cdot) = (\frac{d_K}{}\cdot)$. Writing 
\begin{align*} \Lambda(s, \eta) =  d_K^{\frac{s}{2}} \Gamma_{\bf{R}}(s) L(s, \eta), 
\quad \Gamma_{\bf{R}}(s) := \pi^{-\frac{s}{2}} \Gamma \left( \frac{s}{2} \right) \end{align*} 
to denote its corresponding completed $L$-function, we consider the completed Eisenstein series defined by
\begin{align*} E^{\star}_{L_{A,2}}(\tau, s) &:= \Lambda(s+1, \eta) E_{L_{A,2}}(\tau, s). \end{align*}

\begin{proposition}\label{MSEis} 

The Eisenstein series $E_{L_{A,2}}^{\star}(\tau, s)$ has a meromorphic 
continuation to all $s \in {\bf{C}}$, and satisfies the symmetric functional equation 
$E^{\star}_{L_{A,2}}(\tau, s) = E_{L_{A,2}}^{\star}(\tau, -s).$ \end{proposition}

\begin{proof} See the proof of \cite[Proposition 2.5]{BY} or more generally \cite[Theorem 3.7.2]{Bu}. 
We deduce this in a more straightforward way from the Langlands functional equation for the (coherent) Eisenstein series
\begin{align*} E_{L_{A,2}}(\tau, s; 0) 
= \sum\limits_{\mu \in L_{A,2}^{\vee} / L_{A,2}} E(g_{\tau}, s, \Phi_{\infty}^0 \otimes \lambda_f({\bf{1}}_{\mu}))
= \sum\limits_{\gamma \in \Gamma_{\infty} \backslash \Gamma} \left[ \Im(\tau)^{\frac{(s +1)}{2}} {\bf{1}}_0 \right] 
\bigg\rvert_{0, \omega_{L_{A,2}}} \gamma. \end{align*}

To be more precise, it will suffice to prove the functional equation for each of the Langlands Eisenstein series
$E(g_{\tau}, s, \Phi_{\infty}^0 \otimes \lambda_f({\bf{1}}_{\mu})) = E_{\omega_{ L_{A,2} } }(g_{\tau}, s, \Phi_{\infty}^0 \otimes \lambda_f({\bf{1}}_{\mu}))$. 
Let us write the Euler product decomposition of $\Lambda(s, \eta) = \Lambda(s, \eta_D)$ as $\Lambda(s, \eta) = \prod_{v \leq \infty} L(s, \eta_v)$.
Let us also for simplicity write $\Phi_{\mu} = \lambda_f({\bf{1}}_{\mu})$ for the nonarchimedean part of our chosen global section
$\varphi = \Phi_{\infty}^0 \otimes \Phi_{\mu} \in I(s, \chi_{V_{A, 2}}) = I(s, \eta)$.
Given any standard section $\varphi = \varphi(s) \in  I(s, \eta)$ and $g \in \operatorname{SL}_2({\bf{A}})$, the Langlands 
functional equation implies that \begin{align*} E(g, s; \varphi) = E(g, -s; M(s) \varphi) \end{align*}
for $M(s) = \prod_{v \leq \infty} M_v(s): I(s, \eta) \rightarrow I(s, \eta)$ the global intertwining operator. 
Recall that for $\Re(s) \gg 0$ sufficiently large, each of the local intertwining operators 
$M_v(s): I_v(s, \eta) \rightarrow I_v(s, \eta)$ is given by the formula 
\begin{align*} M_v(s) \varphi_v(g, s) &= \int\limits_{ {\bf{Q}}_v } \varphi_v(w n(b) g, s) db, 
\quad w := \left( \begin{array}{cc} ~ & -1 \\ 1 &  \end{array} \right) \end{align*}
for $\varphi_v$ in the local principal series representation $I_v(s, \eta)$. 
At the real place $v = \infty$, it is well-known that 
\begin{align*} M_{\infty}(s) \Phi_{\infty}^0(g, s) &= C_{\infty}(s) \Phi_{\infty}^0(g, -s), 
\quad C_{\infty}(s) = \gamma_{\infty}(V_{A, 2}) \cdot \frac{\Gamma_{\bf{R}}(s)}{\Gamma_{\bf{R}}(s+1)  }.\end{align*} 
Here, $\gamma_{\infty}(V_{A, 2}) = 1$ denotes the local Weil index for the representation $\omega_{L_{A,2}}$
of $\operatorname{SL}_2({\bf{R}}) \times \operatorname{GSpin}(1,1)$ acting on $\mathcal{S}(V_A({\bf{R}}))$
associated to the signature $(1,1)$ lattice $L_{A ,2}$. At finite places $v \nmid d_K \infty$, is also well-known that 
\begin{align*} M_{v}(s) \Phi_{\mu}(g, s) &= C_{v}(s) \Phi_{\mu}^0(g, -s), \quad C_{v}(s) = \frac{L(s, \eta_v)}{ L(s + 1, \eta_v) }.\end{align*} 
For the remaining finite places $v \mid d_K$, we can use the same computation of the local 
intertwining operators $\Phi_{\mu}$ given in \cite[Proposition 2.5]{BY} to show that 
\begin{align*} M_v(s) \Phi_{\mu}(g, s) &= \gamma_v(V_{A, 2}) \operatorname{vol}(L_{A, 2, v}) \Phi_{\mu}(g, -s), \end{align*} 
where $\gamma_v(V_{A, 2})$ is the local Weil index, 
and $\operatorname{vol}(\Lambda_{A, 2, v}) = [ L_{A, 2, v}^{\vee}: L_{A, 2, v} ]^{-\frac{1}{2}}$
is the measure of $L_{A, 2, v}$ with respect to the self-dual Haar measure on $L_{A, 2, v}$ for the local additive character$\psi_v$. 
Combining the previous local functional equations with the product formulae
\begin{align*} \prod_{v \mid d_K} \operatorname{vol}(L_{A, 2, v}) = d_K^{-\frac{1}{2}}, 
\quad \prod_{v \leq \infty} \gamma_{v}(V_{A, 2}) = 1,\end{align*}
we then obtain the global functional equation
\begin{align*} E(g, s, \Phi_{\infty}^0 \otimes \Phi_{\mu}) 
&= \frac{\Lambda(s, \eta)}{\Lambda(s+1, \eta)} \cdot E(g, -s, \Phi_{\infty}^0 \otimes \Phi_{\mu}). \end{align*}
Using the classical (Dirichlet) functional equation $\Lambda(s, \eta) = \Lambda(1-s, \eta)$, we then deduce the claim. \end{proof} 

\subsubsection{Maass weight raising and lowering operators}

Recall that we defined the Maass weight raising and lowering operators $R_l$ and $L_l$ in $(\ref{Rl})$ and $(\ref{Ll})$ above.
These operators raise and lower the weights of these Eisenstein series by $2$, respectively. 
To be more precise, it is easy to check from the definitions that 
\begin{align*} L_l E_{L_j}(\tau, s; l) &= \frac{1}{2} \cdot (s+ 1 - l) \cdot E_{L_j}(\tau, s; l-2), \\
R_l E_{L_j}(\tau, s; l) &= \frac{1}{2} \cdot (s + 1 + l) \cdot E_{L_j}(\tau, s; l+2). \end{align*} 
Here, we have for the Eisenstein $E_{L_2}(\tau, s; 0)$ series corresponding to our signature $(1,1)$ subspace $V_2$ that 
\begin{equation}\begin{aligned}\label{2.20} 
L_2 E_{L_2}(\tau, s; 2) &= \frac{1}{2} \cdot (s-1) \cdot E_{L_2}(\tau, s; 0). \end{aligned}\end{equation}
Note that the Eisenstein series $E_{L_2}(\tau, s; 2)$ is incoherent. In particular, it has the following analytic properties. 

\begin{proposition}\label{incoherent} 

The nonholomorphic weight-two Eisenstein series $E_{L_2}(\tau, s; 2)$ has a meromorphic continuation 
$E^{\star}_{L_2}(\tau, s; 2) := \Lambda(s+1, \eta) E_{L_2}(\tau, s; 2)$ to all $s \in {\bf{C}}$,
and satisfies the odd functional equation \begin{align*} E_{L_2}^{\star}(\tau, s; 2) = - E_{L_2}^{\star}(\tau, -s; 2).\end{align*} \end{proposition}

\begin{proof} Cf.~Proposition \ref{MSEis} and \cite[Proposition 2.5]{BY} with \cite[$\S$2]{KuEis}, \cite[$\S 5.2$]{KY}, and \cite[$\S$2]{KRY}.
In switching to the standard sections $\Phi_{\infty}^2 \otimes \Phi_{\mu} \in I(s, \eta)$, we switch invariants at infinity. 
This is because $\Phi_{\infty}^2 \otimes \Phi_{\mu}$ corresponds to an incoherent quadratic space 
$\mathcal{C} = \lbrace \mathcal{C}_v \rbrace_v$ in the sense of \cite[Definition 2.1]{KuEis}, 
with $\mathcal{C}_v = (V_2({\bf{Q}}_v), Q_2)$ for all finite places $v<\infty$, but $\mathcal{C}_{\infty}$ of signature $(0, 2)$ and 
Hasse invariant $\varepsilon_{\infty}(\mathcal{C}_{\infty}) = -1$ at the real place $v= \infty$. This alters the sign or root number by $-1$. 
That is, we reduce to showing the functional equation of each Langlands Eisenstein series $E(g_{\tau}, s, \Phi_{\infty}^2 \otimes \Phi_{\mu})$.
Each of the local intertwining operators $M_v(s) \Phi_{\mu}(g, s)$ at $v<\infty$ finite is calculated in the same way as indicated in 
Proposition \ref{MSEis}. At the real place $v = \infty$, it is then well-known that we have a switch of signs 
$M_{\infty}(s) \Phi_{\infty}^{2}(g, s) = - C_{\infty}(s) \cdot \Phi_{\infty}^2(g, -s) 
= -\gamma_{\infty}(V_2) \frac{\Gamma_{\bf{R}}(s)}{\Gamma_{\bf{R}}(s+1)} \cdot \Phi_{\infty}^2(g, -s)$. \end{proof}

Observe that the Eisenstein series $E_{L_2}(\tau, s; 0)$ is holomorphic at $s= s_0 = s_0(V_2) := \dim(V_2)/2 -1= 0$ 
thanks to Siegel-Weil, Theorem \ref{SW-abstract} (cf.~Theorem \ref{SW-vector}). It follows that  at $s=0$, we have the identity
\begin{align}\label{EisID} L_2 E_{L_2}(\tau, 0; 2) &= - \frac{1}{2} \cdot E_{L_2}(\tau, 0; 0). \end{align}
Now, taking the first derivative with respect to $s$ on each side of $(\ref{2.20})$ we get 
\begin{align*} L_2 E'_{L_2}(\tau, s; 2) 
&= \frac{1}{2} \cdot (s-1) \cdot E'_{L_2}(\tau, s; 0) + \frac{1}{2} \cdot E_{L_2}(\tau, s; 0). \end{align*}
Evaluating this identity at $s=0$ gives us 
\begin{equation*}\begin{aligned} L_2 E_{L_2}'(\tau, 0; 2) 
&= \frac{1}{2} \cdot E_{L_2}(\tau, 0; 0) - \frac{1}{2} \cdot E_{L_2}'(\tau, 0;0) \end{aligned}\end{equation*}
and hence 
\begin{equation}\begin{aligned}\label{2.21}  2 L_2 E_{L_2}'(\tau, 0; 2) &=
E_{L_2}(\tau, 0; 0) - E_{L_2}'(\tau, 0; 0). \end{aligned}\end{equation}
Let $\partial$ and $\overline{\partial}$ denote the Dolbeault operators, so that the exterior 
derivative on differential forms on $\mathfrak{H}$ is given by $d = \partial + \overline{\partial}$.
We again write $d \mu(\tau) = \frac{du dv}{v^2}$ for $\tau = u + iv \in \mathfrak{H}$. We have the following useful relation.

\begin{lemma}\label{diff} The weight-lowering operator $L_l$ can be described in terms of differential forms as 
\begin{align*} \overline{\partial}(f d \tau) &= -v^{2-l} \overline{\xi_l(f)} d \mu(\tau) = - L_l f d \mu (\tau). \end{align*} \end{lemma}
\begin{proof} See \cite[Lemma 2.5]{Eh} (cf.~\cite[Lemma 2.3]{BY}). \end{proof}

We also derive the following result for later use.

\begin{proposition}\label{vanish}

We have that $E^{\star \prime}_{L_2}(\tau, 0;0) =0$, and hence via $(\ref{2.21})$ that 
$- 2 L_2 E_{L_2}'(\tau, 0; 2) = - E_{L_2}(\tau, 0; 0)$.
Expressed equivalently in terms of differential forms via Lemma \ref{diff}, we obtain the relation
\begin{align*} - 2 L_2 E'_{L_2}(\tau,0; 2) d \mu(\tau) &= 2 \overline{\partial} \left( E'_{L_2}(\tau, 0; 2) d \tau \right)
= - E_{L_2}(\tau, 0; 0) d\mu(\tau), \end{align*} equivalently 
\begin{equation}\begin{aligned}\label{2.23} E_{L_2}(\tau, 0; 0) d \mu(\tau) 
&= - 2 \overline{\partial} \left( E_{L_2}'(\tau, 0; 2) d \tau \right) . \end{aligned} \end{equation}
\end{proposition}

\begin{proof}

We know by the Siegel-Weil formula (Theorem \ref{SW-vector}) 
that the Eisenstein series $E_{L_2}(\tau, s; 0)$ is analytic at $s=0$. Hence, $E_{L_2}(\tau, s; 0)$ 
and its derivatives with respect to $s$ are analytic at $s=0$. 
This implies, for instance, that the values $E_{L_2}(\tau, 0; 0)$ and $E_{L_2}'(\tau, 0; 0)$ are defined and finite,
and that we can expand $E_{L_2}(\tau, s;0)$ into its Taylor series expansion around $s=0$. 
By Proposition \ref{MSEis}, the Eisenstein series $E_{L_2}(\tau, 0; 0)$ has a meromorphic continuation 
$E_{L_2}^{\star}(\tau, s) = E_{L_2}^{\star}(\tau, s; 0) = \Lambda(s+1, \eta)E_{L_2}(\tau, s;0)$
to all $s \in {\bf{C}}$ which satisfies an even functional equation $E_{L_2}^{\star}(\tau, s) = E_{L_2}^{\star}(\tau, -s)$. 
Comparing the corresponding Taylor series expansions around $s = 0$ as we may, we then see
that for any $s \in {\bf{C}}$ with $0 \leq \Re(s) <1$ we have the relation 
\begin{align*} E_{L_2}^{\star}(\tau, 0) + E_{L_2}^{\star \prime}(\tau, 0) s + O(s^2) 
&= E_{L_2}^{\star}(\tau, 0) - E_{L_2}^{\star \prime}(\tau, 0) s + O(s^2), \end{align*}
equivalently 
\begin{align*} E_{L_2}^{\star \prime}(\tau, 0) s + O(s^2) &=  - E_{L_2}^{\star \prime}(\tau, 0) s + O(s^2). \end{align*}
Taking the limit as $\Re(s) \rightarrow 0$, we see that 
$E^{\star \prime}_{L_2}(\tau, 0)  = \Lambda(1, \eta) E'_{L_2}(\tau, 0; 0) + \Lambda'(1, \eta) E_{L_2}(\tau, 0;0)$ must vanish.
Now, if we multiply $(\ref{2.20})$ by $\Lambda(s+1, \eta)$, viewing this as a scalar in $\tau$ which commutes with $L_2$,
we get the corresponding relation of completed Eisenstein series 
\begin{align}\label{2.20*} L_2 E_{L_2}^{\star}(\tau, s; 2) &= \frac{1}{2}(s-1) E_{L_2}^{\star}(\tau, s: 0)\end{align}
Differentiating both sides of $(\ref{2.20*})$ in $s$ and evaluating at $s=0$, using that $E_{L_2}^{\star \prime}(\tau, 0; 0)=0$, we obtain 
\begin{align}\label{CFI} 2 L_2 E^{\star \prime}_{L_2}(\tau, 0; 2) = E_{L_2}^{\star}(\tau, 0; 0). \end{align}
On the other hand, since the weight-two Eisenstein series $E_{L_2}(\tau, s; 2)$ is incoherent, 
with odd functional equation and vanishing central value $E_{L_2}(\tau, 0; 2) = 0$ (Proposition \ref{incoherent}), we see that 
\begin{align*} E^{\star \prime}_{L_2}(\tau, 0;2) =\Lambda'(1,\eta)E_{L_2}(\tau, 0; 2) + \Lambda(1, \eta) E^{\prime}_{L_2}(\tau, 0; 2)
= \Lambda(1, \eta) E_{L_2}^{\prime}(\tau, 0; 2). \end{align*} Hence, the completed functional identity $(\ref{CFI})$ is equivalent to 
\begin{align*} 2 L_2 \Lambda(1, \eta) E^{\prime}_{L_2}(\tau, 0; 2) &= \Lambda(1, \eta) E_{L_2}(\tau, 0; 0). \end{align*}
Dividing out by the scalar $\Lambda(1, \eta)$ on each side, we obtain the claimed identity $2L_2 E'_{L_2}(\tau, 0; 2) = E_{L_2}(\tau, 0; 0)$. \end{proof}

Let us now consider the Fourier series expansion of the Eisenstein series 
\begin{align*}  E_{L_2}(\tau, s; 2)
&= \sum\limits_{\mu \in L_2^{\vee}/L_2} \sum_{m \in {\bf{Q}}} A_{L_2}(s, \mu, m, v) e(m \tau) {\bf{1}}_{\mu}. \end{align*}
We can use\footnote{Note that no assumption is made on the signature of the quadratic space $(V, Q)$ underlying 
the Eisenstein series in \cite[$\S 4$]{KuBL}.} the discussion in Kudla \cite[$\S 2$]{KuBL} (cf.~\cite[$\S$ 2.2]{BY}) 
to show that the Laurent series expansions of each of the 
Fourier coefficients $A_{L_2}(s, \mu, m, v)$ around $s = 0$ takes the form 
\begin{align}\label{LSE} A_{L_2}(s, \mu, m, v) &= a_{L_2}(\mu, m) +
b_{L_2}(\mu, m, v) s + O\left( s^2 \right), \end{align}
and deduce that the corresponding derivative Eisenstein series at $s=0$ has the Fourier series expansion
\begin{equation}\begin{aligned}\label{EFSE}  E_{L_2}'(\tau, 0; 2) 
&=\sum\limits_{\mu \in L_2^{\vee}/L_2} \sum_{m \in {\bf{Q}}} b_{L_2}(\mu, m, v) e(m \tau) {\bf{1}}_{\mu}. \end{aligned}\end{equation}
Following the argument of Kudla \cite[Theorem 2.12]{KuBL}, we then consider the limiting values 
\begin{align}\label{kappa} \kappa_{L_2}(\mu, m) &= \begin{cases}
\lim_{v \rightarrow \infty} b_{L_2}(\mu, m, v) &\text{ if $\mu \neq 0$ or $m \neq 0$} \\
\lim_{v \rightarrow \infty} b_{L_2}(\mu, m, v) - \log(v) &\text{ if $\mu = 0$ and $m = 0$}.  \end{cases} \end{align}
We define from these coefficients the $\mathcal{S}_{L_2}$-valued 
periodic function $\mathcal{E}_{L_2}(\tau)$ on $\tau = u + iv \in \mathfrak{H}$ via 
\begin{align}\label{EEFSE} \mathcal{E}_{L_2}(\tau) 
&:= \sum\limits_{\mu \in L_2^{\vee}/ L_2} \sum\limits_{m \in {\bf{Q}}} 
\kappa_{L_2}(\mu, m) e(m \tau) {\bf{1}}_{\mu}. \end{align}
Observe (cf.~\cite[Remark 2.4, (3.5)]{BY}) that we can view this form $\mathcal{E}_{L_2}(\tau)$ defined by $(\ref{EEFSE})$ as the holomorphic 
part of derivative Eisenstein series $E_{L_2}'(\tau, 0; 2)$, i.e.~$\mathcal{E}_{L_2}(\tau) =  E_{L_2}'^+(\tau, 0; 2)$. We shall return to this point later.

\subsection{Summation along anisotropic subspaces of signature $(1,1)$} 

We now calculate the regularized theta lifts $\Phi( f_0, z, h)$ along real geodesic cycles corresponding to anisotropic subspaces 
of signature $(1,1)$ given by ideal representatives $\mathfrak{a} \subset \mathcal{O}_K$ of classes $A = [\mathfrak{a}] \in \operatorname{Pic}(\mathcal{O}_c)$ with their norm forms.
Let us simplify notations in writing $(V, q) = (V_A, Q_A)$ to denote the ambient quadratic space of signature $(2,2)$.
We then write $(V_j, Q_j)$ for $j=1,2$ to denote the respective subspaces $(V_{A, 1}, Q_{A, 1})$, and $(V_{A, 2}, Q_{A, 2})$ of signature $(1,1)$. 
We also write $L = L_A$, $L_1 = L_A \cap V_{A, 1}$, and $L_2 = L_A \cap V_{A, 2}$ for the 
corresponding lattices. Let $f_0 \in H_{0}(\omega_L)$ be any harmonic weak Maass form of weight $0$ and representation $\omega_L$. 

We again write $D(V) = D^{\pm}(V)$ for the Grassmannian of oriented $2$-planes $z \subset V({\bf{R}})$.
We also write $D(V_2) = D^{\pm}(V_2)$ for the domain of oriented lines. 
We consider $\operatorname{GSpin}(V_2)$ as a subgroup of $\operatorname{GSpin}(V)$ acting trivially on $V_1$.
Fixing a compact open subgroup $U \subset \operatorname{GSpin}(V)({\bf{A}}_f)$, 
we consider the intersection $U_2 := U \cap \operatorname{GSpin}(V_2)({\bf{A}}_f)$. We then consider the corresponding finite ``geodesic set'' 
\begin{align*} \mathfrak{G}(V_2) &= \operatorname{GSpin}(V_2)({\bf{Q}}) \backslash \operatorname{GSpin}(V_2)({\bf{A}}_f)/U_2. \end{align*}
We can associate to each class $[h] \in \mathfrak{G}(V_{2})$ (for $h \in \operatorname{GSpin}(V_2)({\bf{A}}_f)$ any representative) the symmetric space
\begin{align*} C_h &= \Gamma_{h} \backslash D^{\pm}(V_{2}), 
\quad \Gamma_{h} = \Gamma_{2, h} := \operatorname{GSpin}(V_2)({\bf{Q}}) \cap \operatorname{GSpin}(V_{2})({\bf{R}})^0 h U_{2} h^{-1}. \end{align*} 
Here, $\operatorname{GSpin}(V_{2})({\bf{R}})^0$ denotes the connected component of the identity of $\operatorname{GSpin}(V_{2})({\bf{R}})$.
Note that this connected component $\operatorname{GSpin}(V_2)({\bf{R}})^0$ acts transitively on each connected component $D^{\pm}(V_{2})$ of $D(V_2)$,
and so we can make an identification $D^{\pm}(V_2) \cong \operatorname{GSpin}(V_{2})({\bf{R}})^0/\operatorname{Stab}(z)$ for any fixed basepoint $z$.
Note as well that the connected component of the identity $\operatorname{SO}(V_2)({\bf{R}})^0$ of $\operatorname{SO}(V_2)({\bf{R}})$ acts simply
transitively on each connected component $D^{\pm}(V_2)$, and so we can identity $\operatorname{SO}(V_2)({\bf{R}})^0 \cong D^{\pm}(V_2)$.
We write $z$ to denote an element of $C_h$. Given $f_0 \in H_{0}(\omega_L)$ with regularized theta lift $\Phi(f_0, z, h)$ for the ambient space 
$(V, Q)$ of signature (2,2), we consider the integral $I(f_0, h)$ over the real cycle $C_h$ of dimension one,
\begin{align*} I(f_0, h) &:= \int\limits_{C_h} \Phi(f_0, z, h) d \nu(z).\end{align*}
Here, $d \nu$ denotes the pushforward of the Haar measure on $\operatorname{SO}(V_2)({\bf{R}})^0$, 
which coincides with the $\operatorname{O}(1,1)$-invariant length measure. 
We then consider the corresponding finite sum
\begin{align*} \Phi(f_0, \mathcal{G}) &:= \sum\limits_{ [h] \in \mathfrak{G}(V_{2}) \atop h \in \operatorname{GSpin}(V_2)({\bf{A}}_f)} 
\frac{I(f_0, h)}{\# \operatorname{Aut}(h)} =  \sum\limits_{ [h] \in \mathfrak{G}(V_{2}) \atop h \in \operatorname{GSpin}(V_2)({\bf{A}}_f)} 
\frac{1}{\# \operatorname{Aut}(h)} \int\limits_{C_h} \Phi(f_0, z, h) d \nu (z) \end{align*} over the real geodesic cycle defined by 
\begin{align*} \mathcal{G} &:= \coprod\limits_{[h] \in \mathfrak{G}(V_2) \atop h \in \operatorname{GSpin}(V_2)({\bf{A}}_f)} C_h. \end{align*}

\subsubsection{Measures}

We fix the Tamagawa measure on $\operatorname{SO}(V_2)({\bf{A}}) \cong {\bf{A}}_K^1$, so 
\begin{align*} \operatorname{vol}\left( \operatorname{SO}(V_2)({\bf{Q}}) \backslash \operatorname{SO}(V_2)({\bf{A}}) \right) = 2. \end{align*}
We take the multiplicative Lebesgue measure $\frac{dx}{x}$ on 
$\operatorname{SO}^{\pm}(V_2)({\bf{R}}) \cong K_{\infty}^1 = \lbrace (t, t^{-1}): t \in K_{\infty}, t>0 \rbrace \cong {\bf{R}}_{>0}$. 
We fix the Haar measure on ${\bf{A}}^{\times}$ with $\operatorname{vol}({\bf{Z}}_p^{\times}) = 1$ for each finite place $p$, and $\operatorname{vol}({\bf{A}}_f^{\times}/{\bf{Q}}^{\times}) = 1/2$. 
These choices induce a Haar measure on $\operatorname{GSpin}(V_2)({\bf{A}}) \cong {\bf{A}}_K^{\times}$ via the corresponding Hilbert exact sequence
\begin{align*}1 \longrightarrow {\bf{A}}^{\times} \longrightarrow \operatorname{GSpin}(V_2)({\bf{A}}) 
\longrightarrow \operatorname{SO}(V_2)({\bf{A}}) \rightarrow 1. \end{align*}

\begin{lemma}\label{4.4} 

Let $U \subset \operatorname{GSpin}(V)({\bf{A}}_f)$ be any compact open subgroup, 
and $U_2 = U \cap \operatorname{GSpin}(V_2)({\bf{A}}_f)$. Then,
\begin{align*} \Phi( f_0, \mathcal{G})
&=  \frac{1}{\operatorname{vol}(U_2)} \cdot \int_{ \operatorname{SO}(V_2)({\bf{Q}}) \backslash \operatorname{SO}(V_2)({\bf{A}}) } 
\Phi(f_0, z, h) dh. \end{align*} \end{lemma}

\begin{proof} Cf.~\cite[Lemma 4.5]{BY}, \cite[Lemma 2.13]{Sch}. Let $T(V_2) = \operatorname{GSpin}(V_2) \cong \operatorname{Res}_{K/{\bf{Q}}}{\bf{G}}_m$. 
Let $B(h)$ be the function on $h = h_{\infty}h_f \in T(V_2)({\bf{A}})$ defined by $\Phi(f_0, h_{\infty}z, h_f)$.
Note that $B(h)$ is left $\operatorname{SO}(V_2)({\bf{Q}})$-invariant, and right $U_2$-invariant.
Writing $(\Gamma_h)_{\operatorname{tors}} = \lbrace \pm 1 \rbrace$ to denote the torsion subgroup, we have  
\begin{align*} \int_{ \operatorname{SO}(V_2)({\bf{Q}}) \backslash \operatorname{SO}(V_2)({\bf{A}})  } B(h) dh 
&= \operatorname{vol}(U_2) \sum\limits_{h \in T(V_2)({\bf{Q}}) \backslash T(V_2)({\bf{A}}_f)/U_2 } \frac{1}{\# (\Gamma_h)_{\operatorname{tors}}}
\int\limits_{C_h = \Gamma_h \backslash D^{\pm}(V_2)} B(h) d \nu(z) \end{align*}
To see this, let $\pi: T(V_2) \rightarrow \operatorname{SO}(V_2)$ denote the natural projection. 
Fix a set of idele representatives $h$ of the finite set $T(V_2)({\bf{Q}}) \backslash T(V_2)({\bf{A}}_f)/U_2$, then partition 
$\operatorname{SO}(V_2)({\bf{Q}}) \backslash \operatorname{SO}(V_2)({\bf{A}}) \cong K^{\times} \backslash {\bf{A}}_K^{\times} / {\bf{A}}^{\times}$ 
into cosets $\operatorname{SO}(V_2)({\bf{Q}}) \backslash \operatorname{SO}(V_2)({\bf{Q}}) \pi(h) U_2 \pi(h^{-1}) $, 
then pull back to $T(V_2)({\bf{A}}_f)  \cong {\bf{A}}_{K,f}^{\times} $. 
Since each piece gets measure $\operatorname{vol}(U_2)/ \#(\Gamma_h)_{\operatorname{tors}}$, we deduce the identity. 
Identifying $\#\operatorname{Aut}(h) = \# (\Gamma_h)_{\operatorname{tors}} = \#\lbrace \pm1 \rbrace = 2$, we deduce the claimed formula. \end{proof}

\subsubsection{Decompositions of theta series}

Fix an $\mathcal{S}_{L}$-valued harmonic weak Maass form $f_0 = f_0^+ + f_0^- \in H_{0}(\omega_L)$.
We consider the integral lattice $L  \subset V$ with its corresponding $\mathcal{S}_{L}$-valued Siegel theta series 
$\theta_{L}(\tau, z, h)$ defined on $z \in D(V) = D^{\pm}(V)$, 
$h \in \operatorname{GSpin}(V)({\bf{A}}_f)$, and $\tau  = u + iv \in \mathfrak{H}$ by 
\begin{align*} \theta_{L}(\tau, z, h) &= \sum\limits_{ \mu \in L^{\vee}/L} \theta_{L}^{\star}(z, h, g_{\tau}; {\bf{1}}_{\mu }) \cdot {\bf{1}}_{\mu}. \end{align*}
Following \cite[(3.3), Lemma 3.1]{BY}, we argue that after replacing $f_0$ by its restriction $f_{0, L_1 \oplus L_2}$, 
we may also replace the theta series $\theta_{L}(\tau, z, h)$ of the lattice $L$ with the theta series 
$\theta_{L_1 \oplus L_2}(\tau, z, h)$ of the finite-index sublattice $L_1 \oplus L_2 \subset L$. That is, we use the relation 
$(\theta_{L})^{L_1 \oplus L_2}  = \theta_{L_1 \oplus L_2}$ to derive the identity
\begin{align*} \langle \langle f_0(\tau), \theta_{L}(\tau, z, h) \rangle \rangle 
= \langle \langle f_{0, L_1 \oplus L_2}(\tau), \theta_{L_1 \oplus L_2}(\tau, z, h) \rangle \rangle. \end{align*} 
Let us henceforth write $f_0(\tau)$ to denote the restriction $f_{0, L_1 \oplus L_2}$ of $f_0(\tau)$
to the finite-index sublattice $L_1 \oplus L_2$ of $L$ (see \cite[Lemma 3.1]{BY}). 
We shall then work with the corresponding theta series $\theta_{L_1 \oplus L_2}(\tau, z, h)$, 
which has the following convenient decomposition: 
For $z_{V_2} \in D(V_2)$, $h \in \operatorname{GSpin}(V_2)({\bf{A}}_f)$, and $\tau = u + iv \in \mathfrak{H}$, 
\begin{align}\label{splitting} \theta_{L}(z_{V_2}^{\pm}, \tau) 
&= \theta_{L_1}(\tau) \otimes \theta_{L_2}(\tau, z_{V_2}^{\pm}, h)
= \theta_{L_1}(\tau, 1, 1) \otimes \theta_{L_2}(\tau, z_{V_2}^{\pm}, h). \end{align} 

\subsubsection{Description of the regularized integral}

We have the following more convenient expression for 
\begin{align*}\Phi( f_0, z_{V_2}^{\pm}, h) &= \operatorname{CT}_{s=0} \left( 
\lim_{T \rightarrow \infty} \int\limits_{\mathcal{F}_T} \langle \langle f_0(\tau), \theta_{L_1}(\tau) 
\otimes \theta_{L_2}(\tau, z_{V_2}^{\pm}, h) \rangle \rangle v^{-s} d \mu(\tau) \right). \end{align*} 

\begin{lemma}\label{4.5} 

We have for any oriented line $z_{V_2}^{\pm} \in D(V_2)$ and $h \in \operatorname{GSpin}(V_2)({\bf{A}}_f)$ that
\begin{align*} \Phi(f_0, z_{V_2}^{\pm}, h) &= \left[  \lim_{T \rightarrow \infty} 
\int\limits_{\mathcal{F}_T} \langle \langle f_0(\tau), \theta_{L_1}(\tau) \otimes \theta_{L_2}(\tau, z_{V_2}^{\pm}, h) \rangle \rangle d\mu(\tau) - A_0 \log(T) \right], \end{align*}
where 
\begin{align*} A_0  &= \sum\limits_{\lambda_1 \in L_1^{\vee}/L_1 \atop \lambda_2 \in L_2^{\vee}/L_2} 
\sum\limits_{x \in V_1({\bf{Q}}) \atop x \in \lambda + L_1} c_{f_0}^+(\lambda_1 + \lambda_2, - Q(x)). \end{align*} \end{lemma}

\begin{proof}  

The proof of \cite[Proposition 2.5]{KuBL} applies to this setting. To be clear, we start with
\begin{equation*}\begin{aligned} \Phi(f_{0}, z_{V_2}^{\pm}, h) 
&=\operatorname{CT}_{s =0} \left( \lim_{T \rightarrow \infty} \int\limits_{\mathcal{F}_T} 
\langle \langle f_0(\tau), \theta_{L_{1}}(\tau) \otimes \theta_{L_{2}}(\tau, z_{V_2}^{\pm}, h) \rangle \rangle v^{-s} d \mu (\tau) \right). \end{aligned}\end{equation*}
As the first integral in the limit
\begin{align*} \int\limits_{ \mathcal{F}_1} \langle \langle f_0(\tau), \theta_{L_1}(\tau) \otimes \theta_{L_2}(\tau, z_{V_2}^{\pm}, h) \rangle \rangle 
v^{-s} d \mu (\tau) \end{align*}
is a holomorphic function, we have the preliminary expression
\begin{equation}\begin{aligned}\label{prelim} &\Phi(f_0, z_{V_2}^{\pm}, h) 
= \operatorname{CT}_{s =0} \left( \lim_{ T \rightarrow \infty} \int\limits_1^T C(v, h) v^{-s} \frac{dv}{v} \right) + 
\int\limits_{ \mathcal{F}_1} \langle \langle f_0(\tau), \theta_{L_1}(\tau) \otimes \theta_{L_2}(\tau, z_{V_2}^{\pm}, h) \rangle \rangle v^{-s} d \mu (\tau), \end{aligned}\end{equation}
with the first term given by a limit of partial Mellin transforms of the constant coefficients
\begin{align*} C(v, h) &= \int\limits_{-1/2}^{1/2} v^{-1}
\langle \langle f_{0}(u+iv), \theta_{L_1}(u + iv) \otimes \theta_{L_2}(u + iv, z_{V_2}^{\pm}, h) \rangle \rangle du. \end{align*}
Let 
\begin{align*} C^{\pm}(v, h) &= \int\limits_{-1/2}^{1/2} v^{-1}
\langle \langle f^{\pm}_{0}(u+iv), \theta_{L_1}(u + iv) \otimes \theta_{L_2}(u + iv, z_{V_2}^{\pm}, h)\rangle \rangle du. \end{align*}
Writing $M = L_{1} \oplus L_{2}$ and $z_2 = z_{V_2}^{\pm}$ to simplify notations, we have 
\begin{align*} \theta_{M}(\tau, z_{2}, h) &= v \sum\limits_{\mu \in M^{\vee}/M} \theta_{M, \mu}(\tau, z_{2}, h) {\bf{1}}_{\mu} =
 v \sum\limits_{\mu \in M^{\vee}/M} \sum\limits_{x \in V({\bf{Q}}) \atop x \in h \mu} 
 e \left( \tau Q(x_{z_2^{\perp}}) + \overline{\tau} Q(x_{z_2}) \right) {\bf{1}}_{\mu}.\end{align*}
Opening Fourier series expansions and using the orthogonality of additive characters, we find that 
\begin{equation*}\begin{aligned}
&C^+(v, h) = v^{-1} \int_{-1/2}^{1/2} \sum\limits_{\mu \in M^{\vee}/M} f_{0, \mu}^+(u + iv) \theta_{M, \mu}(u + iv, z_2, h) du \\
&= \sum\limits_{\mu \in M^{\vee}/M}  \sum\limits_{m \in {\bf{Q}} \atop m \gg -\infty } c_{f_{0}}^+(\mu, m) e(m iv)
\sum\limits_{x \in V({\bf{Q}}) \atop x \in h \mu} 
e \left( iv Q(x_{z_2^{\perp}}) -iv Q(x_{z_2}) \right) \int_{-1/2}^{1/2} e \left( mu + Q(x_{z_2^{\perp}})u + Q(x_{z_2}) u \right) du \\
&=  \sum\limits_{\mu \in M^{\vee}/M} \sum\limits_{m \in {\bf{Q}} \atop m \gg -\infty } 
\sum\limits_{ {x \in V({\bf{Q}}) \atop x \in h \mu} \atop m =- Q(x_{z_2^{\perp}}) - Q(x_{z_2})} c_{f_{0}}^+(\mu, m) e(m iv) e \left( iv Q(x_{z_2^{\perp}}) -iv Q(x_{z_2}) \right) \\
&=  \sum\limits_{\mu \in M^{\vee}/M} 
\sum\limits_{ {x \in V({\bf{Q}}) \atop x \in h \mu} \atop m =- Q(x_{z_2^{\perp}}) - Q(x_{z_2})} c_{f_{0}}^+(\mu, m) 
e\left(  - Q(x_{z_2^{\perp}}) iv - Q(x_{z_2}) iv + Q(x_{z_2^{\perp}}) iv - Q(x_{z_2}) iv \right) \\
&=  \sum\limits_{\mu \in M^{\vee}/M} 
\sum\limits_{ x \in V({\bf{Q}}) \atop x \in h \mu}  c_{f_{0}}^+(\mu, - Q(x_{z_2^{\perp}}) - Q(x_{z_2})) e^{4 \pi v Q(x_{z_2})} \end{aligned}\end{equation*}
and that
\begin{equation*}\begin{aligned} C^-(v, h) &= \sum\limits_{\mu \in M^{\vee}/M} \sum\limits_{ x \in V({\bf{Q}}) \atop x \in h \mu}  
c_{f_{0}}^-(\mu, - Q(x_{z_2^{\perp}}) - Q(x_{z_2})) 
W_{0} \left( 2 \pi v \left( - Q(x_{z_2^{\perp}}) - Q(x_{z_2}) \right) \right) e^{4 \pi v Q(x_{z_2})}. \end{aligned}\end{equation*}
Since $Q\vert_{z_2} <0$ for $z_2 \in D(V_{2})$, we deduce from known bounds for the Fourier coefficients that 
\begin{align*} \lim_{T \rightarrow \infty} \int_1^T C(v, h) v^{-s} \frac{dv}{v} = \lim_{T \rightarrow \infty} \int_1^T \left( C^+(v, h) + C^-(v, h) \right) v^{-s} \frac{dv}{v} \end{align*}
converges absolutely. We first consider the contributions from $x$ orthogonal to $z_2$, so $(x, z_2)=0$, 
equivalently $x_{z_2}=0$ so that $Q(x_{z_2})=0$ and $x \in V_{1}$. These are given by 
\begin{align*} C_{ V_1 }^+(v, h) &= \sum\limits_{ \lambda_1 \in L_1^{\vee}/L_{1} \atop \lambda_2 \in L_2^{\vee}/L_2  }
\sum\limits_{x \in V_1 ({\bf{Q}}) \atop x \in \lambda_1 + L_1} c_{f_0}^+(\lambda_1 + \lambda_2, -Q(x) ) = C_{V_1}^+ \end{align*}
and 
\begin{align*} C_{ V_1 }^-(v, h) &= \sum\limits_{ \lambda_1 \in L_1^{\vee}/L_1 \atop \lambda_2 \in L_2^{\vee}/L_2  }
\sum\limits_{x \in V_1({\bf{Q}}) \atop x \in \lambda_1 + L_1} W_{0} \left( - 2 \pi v Q(x_{z_2^{\perp}}) \right)
c_{f_{0}}^-(\lambda_1 + \lambda_2, -Q(x_{z_2^{\perp}})) = C_{V_1}^-(v).\end{align*} 
Here, $C_{V_1}^+$ does not depend on $v$, and neither of the coefficients $C_{V_1}^{\pm}$ depends on $h$. We have for $\Re(s) >0$ that 
\begin{align*} \lim_{T \rightarrow \infty} \int_1^T C^+_{V_1}(v, h) v^{-s} \frac{dv}{v} 
&= C^+_{V_1} \cdot \lim_{T \rightarrow \infty} \int_1^T v^{-s} \frac{dv}{v} = C^+_{V_1} \cdot \frac{ \left( 1-T^{-s} \right)}{s}, \end{align*}
with 
\begin{align*} \lim_{s \rightarrow 0} \left( \lim_{T \rightarrow \infty} \int_1^T C^+_{V_1}(v, h) v^{-s} \frac{dv}{v} \right) 
&= C^+_{V_1} \cdot \lim_{T \rightarrow \infty} \log(T). \end{align*}
Hence, this term does not contribute to the Laurent series expansion around $s=0$. This gives the stated contribution $A_0=C_{V_1}^+$.
For the remaining contributions of the $x$ not orthogonal to $z_2$, we have 
\begin{align*} C_{V_2}^+(v, h) &= \sum\limits_{ \lambda_1 \in L_1^{\vee}/L_1 \atop  \lambda_2 \in L_2^{\vee}/L_2} \sum\limits_{ x \in V_2({\bf{Q}}) \atop x \in h \lambda  }
c_{f_0}^+( \lambda_1 + \lambda_2, -Q( x_{z_2} ) - Q( x_{ z_2^{\perp} } ) ) e^{4 \pi v Q(x_{z_2}) } \end{align*}
and 
\begin{align*} C_{V_2}^-(v, h) &= \sum\limits_{ \lambda_1 \in L_1^{\vee}/L_1 \atop  \lambda_2 \in L_2^{\vee}/L_{2}} \sum\limits_{ x \in V_2({\bf{Q}}) \atop x \in h \lambda  }
c_{f_{0}}^-(\lambda_1 + \lambda_2, -Q(x_{z_2} ) - Q(x_{z_2^{\perp}})) W_{0} \left( - 2 \pi v \left( Q(x_{z_2}) + Q(x_{z_2^{\perp}}) \right)  \right)e^{4 \pi v Q(x_{z_2})}.\end{align*}
As explained in \cite[Proposition 2.5]{KuBL}, the integrals defined for $t>0$ by 
\begin{align*} \beta_{s+1}(t) &= \int_1^{\infty} e^{-t v} v^{-s} \frac{dv}{v} \end{align*}
are convergent for all $s \in {\bf{C}}$, and determine holomorphic functions of $s$. In this way, we deduce that 
\begin{align*} \operatorname{CT}_{s=0} \left( \lim_{T \rightarrow \infty} \int_1^T C(v, h) v^{-s} \frac{dv}{v} \right)
&= \lim_{T \rightarrow \infty} \left( \int_1^T C(v, h) v^{-s} \frac{dv}{v} - C^+_{V_1} \cdot \log(T)  \right) \\
&= \lim_{T \rightarrow \infty} \left( \int_1^T C(v, h) v^{-s} \frac{dv}{v} - A_0 \cdot \log(T)  \right). \end{align*}
Substituting this back into the initial expression $(\ref{prelim})$, we find the desired formula. \end{proof}

\begin{corollary}\label{4.6} 

Using the Siegel-Weil formula of Theorem \ref{SW-abstract} and Theorem \ref{SW-vector}, we have that 
\begin{align*} \Phi(f_0, \mathcal{G}) &= \frac{1}{ \operatorname{vol}(U_2) } \cdot \lim_{T \rightarrow \infty} \left[
 \int_{ \mathcal{F}_T } \langle \langle f_0 (\tau), \theta_{L_1}(\tau) \otimes E_{L_2}(\tau, 0; 0) \rangle \rangle d \mu(\tau) 
- A_0 \log(T) \right]. \end{align*}  \end{corollary}

\begin{proof} 

We expand the definition using Lemma \ref{4.4}, Lemma \ref{4.5} and the decomposition $(\ref{splitting})$;
we then switch the order of summation, and apply Theorem \ref{SW-vector} (with $\kappa =2$) 
to evaluate the inner integral over $\theta_{L_2}(z_{V_2}^{\pm}, h)$. In this way, we compute 
\begin{equation*}\begin{aligned} &\Phi( f_0, \mathcal{G}) = \frac{1}{ \operatorname{vol}(U_2)} \cdot 
\int_{ \operatorname{SO}(V_2)({\bf{Q}}) \backslash \operatorname{SO}(V_2)({\bf{A}}) } \Phi(f_0, z_{V_2}^{\pm}, h) dh \\
&= \frac{1}{ \operatorname{vol}(U_2)} \cdot \int_{ \operatorname{SO}(V_2)({\bf{Q}}) \backslash \operatorname{SO}(V_2)({\bf{A}}) }
\lim_{T \rightarrow \infty} \left[ \int_{\mathcal{F}_T} 
\langle \langle f_0(\tau), \theta_{L_1}(\tau) \otimes \theta_{L_2}(z_{V_2}^{\pm}, h, \tau) \rangle \rangle d\mu(\tau) - A_0 \log(T) \right] dh \\
&= \frac{1}{ \operatorname{vol}(U_2)} \cdot \lim_{T \rightarrow \infty} \left[
\int_{\mathcal{F}_T}  \langle \langle f_0(\tau), \theta_{L_1}(\tau) \otimes 
\left( \int_{ \operatorname{SO}(V_2)({\bf{Q}}) \backslash \operatorname{SO}(V_2)({\bf{A}}_f) }
\theta_{L_2}(z_{V_2}^{\pm}, h, \tau) dh \right) \rangle \rangle d\mu(\tau) - A_0 \log(T) \right] \\
&= \frac{1}{ \operatorname{vol}(U_2)} \cdot \lim_{T \rightarrow \infty} \left[
 \int_{\mathcal{F}_T}   \langle \langle f_0(\tau), \theta_{L_1}(\tau) \otimes E_{L_2}(\tau, 0; 0) \rangle \rangle d\mu(\tau) - A_0 \log(T) \right]. 
\end{aligned}\end{equation*} \end{proof}

\subsubsection{Main calculation}

Given $g \in S_{2}(\overline{\omega}_L)$ a cuspidal holomorphic modular form of weight $2$ and representation $\overline{\omega}_L$, 
let us now consider the Rankin-Selberg $L$-function given by the integral presentation 
\begin{align*} L(s, g \times \theta_{L_1}) &:= \langle g (\tau), \theta_{L_1} (\cdot) \otimes E_{L_2}(\tau, s; 2) \rangle
= \int\limits_{\mathcal{F}} \langle \langle \overline{g(\tau)}, \theta_{L_1}(\tau) \otimes E_{L_2}(\tau, s, 2) \rangle \rangle v^2 d\mu(\tau). \end{align*}
Note that via the appearance of the incoherent Eisenstein series $E_{L_2}(\tau, s; 2)$ with its odd, symmetric functional equation
$E^{\star}_{L_2}(\tau, s; 2) := \Lambda(s+1, \eta) E_{L_2}(\tau, s; 2) = - E^{\star}_{L_2}(\tau, -s; 2)$ (Proposition \ref{incoherent}), 
the corresponding completed Rankin-Selberg $L$-function $L^{\star}(s, g \times \theta_{L_1})$ satisfies the odd, symmetric functional equation 
\begin{align}\label{LFE} L^{\star}(s, g \times \theta_{L_1}) := \Lambda(s+1, \eta)L(s, g \times \theta_{L_1}) = -L^{\star}(-s, g \times \theta_{L_1}). \end{align}
We shall take $g = \xi_0(f_0)$, and write $L'(s, g \times \theta_{L_1}) = \frac{d}{ds} L(s, g \times \theta_{L_1})$ to denote the derivative with respect to $s$. 
We shall also consider the regularized integral $(\ref{regint})$ introduced above, 
\begin{equation*}\begin{aligned}  I(s, f_{0} \times \xi_0(\theta_{L_{1}})) &:=
\int_{\mathcal{F}}^{\star} \langle \langle f_{0}(\tau), \overline{\xi_0\theta_{L_{1}}(\tau)} \otimes E_{L_{2}}(\tau, s; 2) \rangle \rangle v^2 d \mu (\tau) \\
&= \lim_{T \rightarrow \infty} \int\limits_{\mathcal{F}_T}  \langle \langle f_{0}(\tau), \overline{\xi_0\theta_{L_{1}}(\tau)} \otimes E_{L_{2}}(\tau, s; 2) \rangle \rangle v^2 d \mu(\tau). \end{aligned}\end{equation*}
Note that this regularized theta integral can be thought of as a Rankin-Selberg $L$-function. 
That is, the ``shadow" $\xi_0\theta_{L_1}(\tau)$ determines a holomorphic theta series of weight $2$.
That is, we deduce from the appearance of the incoherent Eisenstein series $E_{L_2}(\tau, s; 2)$ that this regularized inner product $I(s, f_0, \xi_0 \theta_{L_1})$ has an analytic 
contination at all $s \in {\bf{C}}$, and satisfies and odd, symmetric functional equation 
\begin{align}\label{IFE} I^{\star}(s, f_0, \xi_0 \theta_{L_1}):= \Lambda(s+1, \eta) I(s, f_0 \times \xi_0 \theta_{L_1}) &= - I^{\star}(-s, f_0, \xi_0 \theta_{L_1}).\end{align} 
Recall that we write $\mathcal{E}_{L_2}(\tau)$ to denote the holomorphic part of the derivative Eisenstein series $E^{\prime}_{L_2}(\tau, 0; 2)$,
with Fourier series expansion given by the series $(\ref{EEFSE})$ and coefficients $(\ref{kappa})$. 

\begin{theorem}\label{4.7} We have that 
\begin{equation*}\begin{aligned} \Phi( f_0, \mathcal{G}) &=  - \frac{2}{\operatorname{vol}(U_2)} \left( 
 \operatorname{CT} \langle \langle f_0^+(\tau), {\bf{1}}_{L_1 + 0} \otimes \mathcal{E}_{L_2}(\tau) \rangle \rangle 
 + L'(0, \xi_0(f_0) \times \theta_{L_1}) + I'(0, f_0 \times \xi_0(\theta_{L_1})) \right). \end{aligned}\end{equation*} 
  \end{theorem}

\begin{proof} 

We derive a variation of \cite[Theorem 4.7]{BY} and \cite[Theorem 3.5]{Eh} via Proposition \ref{vanish} above. 
Here, Lemma \ref{4.4}, Lemma \ref{4.5}, and Corollary \ref{4.6} imply that
\begin{align}\label{initial} \Phi( f_0, \mathcal{G})
&= \frac{1}{\operatorname{vol}(U_2)} \cdot \lim_{T \rightarrow \infty} \left[  I_T(f_0) - A_0 \log(T) \right], \quad I_T(f_0) := \int_{\mathcal{F}_T} 
\langle \langle f_0(\tau), \theta_{L_1}(\tau) \otimes E_{L_2}(\tau, 0; 0) \rangle \rangle d \mu(\tau). \end{align}
Using the identity $(\ref{2.23})$ for $E_{L_2}(\tau, s, 0)$ at $s=0$, we find via the Leibniz rule for $d = \partial + \overline{\partial}$ that
\begin{equation*}\begin{aligned} &\overline{\partial} \langle \langle f_0(\tau), \theta_{L_1}(\tau) \otimes  E_{L_2}^{\prime}(\tau, 0; 2) d \tau \rangle \rangle
= d \langle \langle f_0(\tau), \theta_{L_1}(\tau) \otimes  E_{L_2}^{\prime}(\tau, 0; 2) d \tau \rangle \rangle \\
&=\langle \langle \overline{\partial} f_0(\tau), \theta_{L_1}(\tau) \otimes E_{L_2}^{\prime}(\tau, 0; 2) d \tau \rangle \rangle
+ \langle \langle f_0(\tau), \overline{\partial} \theta_{L_1}(\tau) \otimes \overline{\partial} E_{L_2}^{\prime}(\tau, 0; 2) d \tau \rangle \rangle
+ \langle \langle f_0(\tau),  \theta_{L_1}(\tau) \otimes \overline{\partial} E_{L_2}^{\prime}(\tau, 0; 2) d \tau \rangle \rangle\end{aligned}\end{equation*} 
and hence 
\begin{equation}\begin{aligned}\label{I_T} I_T(f_0) 
&= \int_{\mathcal{F}_T} \langle \langle f_0(\tau), \theta_{L_1}(\tau) \otimes E_{L_2}(\tau, 0; 0) \rangle \rangle d \mu(\tau) 
= -2 \int\limits_{\mathcal{F}_T} \langle \langle f_0(\tau), \theta_{L_1}(\tau) \otimes \overline{\partial} E_{L_2}^{\prime}(\tau, 0; 2) d \tau \rangle \rangle \\
&=  -2 \int\limits_{\mathcal{F}_T} d \langle \langle f_0(\tau), \theta_{L_1}(\tau) \otimes E_{L_2}^{\prime}(\tau, 0; 2) d \tau \rangle \rangle
+ 2 \int\limits_{\mathcal{F}_T} \langle \langle \overline{\partial} f_0(\tau), \theta_{L_1}(\tau) \otimes E_{L_2}^{\prime}(\tau, 0; 2) d \tau \rangle \rangle \\
&+ 2 \int\limits_{\mathcal{F}_T} \langle \langle  f_0(\tau),  \overline{\partial} \theta_{L_1}(\tau) \otimes E_{L_2}^{\prime}(\tau, 0; 2) d \tau \rangle \rangle. \end{aligned}\end{equation} 

To compute the first integral on the right-hand side of $(\ref{I_T})$, 
we apply Stokes' theorem\footnote{Note that this does not require a change
of sign after identifying the boundary $\partial \mathcal{F}_T$ with the interval $[iT, iT+1]$, 
and that there is a sign error in the first integral on the right-hand side of the second identity stated in \cite[p.~655, proof of Theorem 4.7]{BY}. 
There is also a sign error in the second integral, c.f.~\cite[Theorem 5.7.1]{AGHMP}.
This latter error appears to come from the differential forms identity 
$\overline{\partial}(f d \tau) = -v^{l-2} \xi_k(f) d \mu(\tau) = -L_l f d \mu(\tau)$, cf.~\cite[Lemma 2.5]{Eh},
which is used implicitly without the sign change in the first identification of \cite[p.~655]{BY}.} to find that 
\begin{equation}\begin{aligned}\label{I_T1} 
&-2 \int\limits_{\mathcal{F}_T} d \langle \langle f_0(\tau), \theta_{L_1}(\tau) \otimes E^{\prime}_{L_2}(\tau, 0; 2) d \tau \rangle \rangle 
= -2 \int_{\partial \mathcal{F}_T} \langle \langle f_0(\tau), \theta_{L_1}(\tau) \otimes E_{L_2}^{\prime}(\tau, 0; 2) d \tau \rangle \rangle \\
&= - 2 \int_{\tau = i T}^{iT +1} \langle \langle f_0(\tau), \theta_{L_1}(\tau) \otimes E^{\prime}_{L_2}(\tau, 0; 2) \rangle \rangle d \tau
= - 2 \int_{0}^1  \langle \langle f_0(u + iT), \theta_{L_1}(u + iT) \otimes E^{\prime}_{L_2}(u + iT, 0; 2) \rangle \rangle d u. \end{aligned}\end{equation}

To compute the second integral on the right-hand side of $(\ref{I_T})$, we use the relation of differential forms 
\begin{align*} \overline{\partial}(f_0(\tau) d \tau) 
&= - v^2 \overline{\xi_0(f_0)(\tau)} d \mu(\tau) = -L_0 f_0(\tau) d \mu (\tau) \end{align*}
implied by Lemma \ref{diff} to deduce that 
\begin{equation}\begin{aligned}\label{I_T2} 
 2 \int\limits_{\mathcal{F}_T} \langle \langle \overline{\partial} f_0(\tau), \theta_{L_1}(\tau) \otimes E_{L_2}^{\prime}(\tau, 0; 2) d \tau \rangle \rangle
 &= -2 \int\limits_{\mathcal{F}_T} \langle \langle \overline{\xi_0(f_0)(\tau)}, 
 \theta_{L_1}(\tau) \otimes E_{L_2}^{\prime}(\tau, 0; 2) \rangle \rangle v^2 d \mu(\tau). \end{aligned}\end{equation}

To compute the third integral on the right-hand side of $(\ref{I_T})$, we use the same relation of differential forms 
\begin{align*} \overline{\partial}(\theta_{L_1}(\tau) d \tau) &= - v^2 \overline{\xi_0(\theta_{L_1})(\tau)} d \mu(\tau) = -L_0 \theta_{L_1}(\tau) d \mu (\tau) \end{align*}
implied by Lemma \ref{diff} to deduce that 
\begin{align*} 2 \int\limits_{\mathcal{F}_T} \langle \langle  f_0(\tau),  \overline{\partial} \theta_{L_1}(\tau) \otimes E_{L_2}^{\prime}(\tau, 0; 2) d \tau \rangle \rangle &=
-2 \int\limits_{\mathcal{F}_T} \langle \langle f_0(\tau), L_0 \theta_{L_1}(\tau) \otimes E_{L_2}'(\tau, 0; 2) \rangle \rangle d \mu(\tau) \\
&= -2 \int\limits_{\mathcal{F}_T} \langle \langle f_0(\tau), \overline{\xi_0(\theta_{L_1})(\tau)} \otimes E_{L_2}'(\tau, 0; 2) \rangle \rangle v^2 d \mu(\tau). \end{align*}

Hence, we obtain the identity 
\begin{equation}\begin{aligned}\label{I_Tev} I_T(f_0) 
&= - 2 \int\limits_{t = iT}^{iT + 1} \langle \langle f_0(\tau), \theta_{L_1}(\tau) \otimes E_{L_2}^{\prime}(\tau, 0; 2)  \rangle \rangle d \tau
-2 \int_{\mathcal{F}_T} \langle \langle \overline{\xi_0(f_0)(\tau)}, \theta_{L_1}(\tau) \otimes E^{\prime}_{L_2}(\tau, 0; 2) \rangle \rangle v^2 d \mu(\tau) \\
&-2 \int\limits_{\mathcal{F}_T} \langle \langle f_0(\tau), \overline{\xi_0(\theta_{L_1})(\tau)} \otimes E_{L_2}'(\tau, 0; 2) \rangle \rangle v^2 d \mu(\tau).
\end{aligned}\end{equation}
Inserting this identity $(\ref{I_Tev})$ back into the initial formula $(\ref{initial})$ gives us the preliminary formula 
\begin{equation}\begin{aligned}\label{initial2} \Phi( f_0, \mathcal{G})
&= - \frac{1}{\operatorname{vol}(U_2)} \cdot \lim_{T \rightarrow \infty} 
\left[ 2 \int_{\tau = i T}^{iT +1} \langle \langle f_0(\tau), \theta_{L_1}(\tau) \otimes E^{\prime}_{L_2}(\tau, 0; 2)\rangle \rangle d \tau -  A_0 \log(T) \right] \\ 
&- \frac{1}{\operatorname{vol}(U_2)} \cdot \lim_{T \rightarrow \infty} 2 \int_{\mathcal{F}_T} \langle \langle \overline{\xi_0(f_0)(\tau)}, \theta_{L_1}(\tau) 
\otimes E'_{L_2}(\tau, 0; 2) \rangle \rangle v^2 d \mu(\tau) \\
&-  \frac{1}{\operatorname{vol}(U_2)} \cdot \lim_{T \rightarrow \infty} 2 \int_{\mathcal{F}_T} \langle \langle f_0(\tau), \overline{ \xi_0(\theta_{L_1})(\tau)} 
\otimes E'_{L_2}(\tau, 0; 2) \rangle \rangle v^2 d \mu(\tau). \end{aligned}\end{equation} 

We now argue as in \cite[Theorem 3.5, (3.12), (3.11)]{Eh}
that we may replace the $f_0(\tau)$ in the first integral on the right of $(\ref{initial2})$ with its holomorphic part $f_0^+(\tau)$,
as the remaining non-holomorphic part $f_0^-(\tau)$ is rapidly decreasing as $v \rightarrow \infty$. 
That is, we first split the constant coefficient term in $(\ref{initial2})$ into parts as 
\begin{equation}\begin{aligned}\label{CCcontribution} &\lim_{T \rightarrow \infty} \int_{\tau = i T}^{iT +1} 
\langle \langle f_0(\tau), \theta_{L_1}(\tau) \otimes E^{\prime}_{L_2}(\tau, 0; 2)\rangle \rangle d \tau \\
&= \lim_{T \rightarrow \infty} \int_{\tau = i T}^{iT +1} \langle \langle f^+_0(\tau), \theta_{L_1}(\tau) \otimes E^{\prime }_{L_2}(\tau, 0; 2)\rangle \rangle d \tau 
+ \lim_{T \rightarrow \infty} \int_{\tau = i T}^{iT +1} \langle \langle f^-_0(\tau), \theta_{L_1}(\tau) \otimes E^{\prime}_{L_2}(\tau, 0; 2)\rangle \rangle d \tau. 
\end{aligned}\end{equation}
We now consider the second integral on the right-hand side of $(\ref{CCcontribution})$, writing the Fourier series expansion as 
\begin{align*} \langle \langle f_0^-(\tau), \theta_{L_1}(\tau) \otimes E_{L_2}'(\tau, 0; 2) \rangle \rangle 
&= \sum\limits_{n \in {\bf{Z}}} a(n, i v) e(n \tau). \end{align*}
Opening up this expansion in the corresponding integral, then using the orthogonality of 
additive characters on the torus ${\bf{R}}/{\bf{Z}} \cong [0, 1]$ to evaluate, we find that
\begin{equation*}\begin{aligned}
&\int\limits_{\tau = i T}^{iT + 1} \langle \langle f_0^-(\tau), \theta_{L_1}(\tau, 1, 1) \otimes E_{L_2}'(\tau, 0; 2) \rangle \rangle d\tau
= \int_0^1 \langle \langle f_0^-(u + iT), \theta_{L_1}(u + iT, 1, 1) \otimes E_{L_2}'(u + iT, 0; 2) \rangle \rangle du \\
&= \sum\limits_{n \in {\bf{Z}}} a(n, iT) e(i n T) \int_0^1 e(nu)du = a(0, iT) 
= \sum\limits_{\mu = \lambda_1 + \lambda_2 \in (L_1 \oplus L_2)^{\vee}/ L_1 \oplus L_2} 
\sum\limits_{m \in {\bf{Q}}_{>0}} c_{f_0}^-(-\mu, m) W_0(- 2 \pi m v) c_g(\mu, m, v). \end{aligned}\end{equation*}
Here, we write $c_g(m, \mu, v)$ to denote the Fourier series coefficients of 
$g(\tau) = \theta_{L_1}(\tau, 1) \otimes E^{\prime}_{L_2}(\tau, 0; 2)$, i.e. 
\begin{align*} g(\tau) = \theta_{L_1}(\tau, 1) \otimes E^{\prime}_{L_2}(\tau, 0; 2) &=
\sum\limits_{ \mu \in (L_1 \oplus L_2)^{\vee}/(L_1 \oplus L_2 )  } \sum\limits_{ m \in {\bf{Q}} }
c_g(\mu, m, v) e(m \tau) {\bf{1}}_{\mu} . \end{align*}
We can now use the rapid decay for the Whittaker coefficients 
$W_0(y) = \int_{-2y}^{\infty} e^{-t} dt = \Gamma(1, 2 \vert y \vert)$ for $y \rightarrow - \infty$
in the Fourier series expansions of $f_0^-(\tau)$ with standard bounds for the Fourier coefficients 
of $f_0^-(\tau)$ and $g(\tau)$ to deduce that for some integer $M >0$ and some constant $C>0$, we have for each $m \geq M$ that 
\begin{align*} c_{f_0}^-(\mu, -m) W_0(-2 \pi m v) c_g(\mu, m, v) = O\left( e^{- m C v}  \right). \end{align*}
We deduce from this that for some constants $c, C >0$, we have the upper bound
\begin{align*} \vert a(0, iT) \vert \leq c \cdot \frac{e^{-CT}}{(1 - e^{-CT})}, \end{align*}
from which it follows that $\lim_{T \rightarrow \infty} \vert a(0, iT) \vert = 0$.
Hence, the second integral on the right-hand side of $(\ref{CCcontribution})$ vanishes in the limit with $T \rightarrow \infty$.
Hence, the first term on the right-hand side of $(\ref{initial2})$ can be simplified to  
\begin{align}\label{RHS} \frac{1}{\operatorname{vol}(U_2)} \cdot \lim_{T \rightarrow \infty} 
\left[ 2 \int_{\tau = i T}^{iT +1} \langle \langle f_0^+(\tau), \theta_{L_1}(\tau) 
\otimes E^{\prime}_{L_2}(\tau, 0; 2)\rangle \rangle d \tau -  A_0 \log(T) \right]. \end{align} 
To evaluate this simplified term, we first write out the Fourier expansion of $\theta_{L_1}(\tau) \otimes E_{L_2}'(\tau, 0; 2)$ as 
\begin{align*} \theta_{L_1}(\tau) \otimes E_{L_2}'(\tau, 0; 2) 
&= \sum\limits_{ \lambda_1 \in L_1^{\vee}/L_1 } \sum\limits_{\lambda_2 \in L_2^{\vee}/L_2}
\theta_{L_1, \lambda_1}(\tau) E_{L_2, \lambda_2}'(\tau, 0; 2) {\bf{1}}_{\lambda_1 + \lambda_2} \\
&=  \sum\limits_{ \lambda_1 \in L_1^{\vee}/L_1 } \sum\limits_{\lambda_2 \in L_2^{\vee}/L_2}
\sum\limits_{ x \in V_1({\bf{Q}}) \atop x \in \lambda_1 + L_1 } e \left( Q_1(x_{z_0^{\perp}}) \tau + Q_1(x_{z_0}) \overline{\tau} \right) 
\sum\limits_{m \in {\bf{Q}}} b_{L_2}(\lambda_2, m, v) e(m \tau) {\bf{1}}_{\lambda_1 + \lambda_2}. \end{align*}
Here, we fix a basepoint $z_0 \in D(V_1)$ corresponding to the identity splitting $V_1({\bf{R}}) = z_0 \oplus z_0^{\perp}$.
Hence, we write $x = x_{z_0} + x_{z_0^{\perp}}$ with $Q_1(x) = Q_1(x_{z_0}) + Q_1(x_{z_0^{\perp}})$
to describe the expansion of the theta series $\theta_{L_1}(\tau) = \theta_{L_1}(\tau, 1, 1) = \theta_{L_1}(\tau, z_0, 1)$ explicitly. 
We then open expansions and use the orthogonality of additive characters on the torus ${\bf{R}}/{\bf{Z}}$ to compute 
\begin{equation}\begin{aligned}\label{CTcalculation}
&\int\limits_{\tau = i T}^{iT + 1} \langle \langle f_0^+(\tau), \theta_{L_1}(\tau) \otimes E_{L_2}'(\tau, 0; 2) \rangle \rangle d \tau
= \int\limits_0^1 \langle \langle f_0^+(u + iT), \theta_{L_1}(u + iT) \otimes E_{L_2}'(u + iT, 0; 2)  \rangle \rangle du \\
&= \sum\limits_{ \lambda_1 \in L_1^{\vee} / L_1 \atop \lambda_2 \in L_2^{\vee}/L_2  }
\sum\limits_{n \gg - \infty} c_{f_0}^+(\lambda_1 + \lambda_2, n) e(niT)
\sum\limits_{x \in V_1({\bf{Q}}) \atop x \in \lambda_1 + L_1} e \left( Q_1(x_{z_0^{\perp}})iT - Q_1(x_{z_0})iT \right)
\sum\limits_{m \in {\bf{Q}}} b_{L_2}(\lambda_2, m, v) e(miT) \\
&\times \int\limits_0^1 e\left( nu + Q_1(x_{z_0^{\perp}})u + Q_1(x_{z_0})u + mu \right) du \\
&= \sum\limits_{ \lambda_1 \in L_1^{\vee} / L_1 \atop \lambda_2 \in L_2^{\vee}/L_2  }
\sum\limits_{n \gg - \infty} c_{f_0}^+(\lambda_1 + \lambda_2, n) e(niT)
\sum\limits_{x \in V_1({\bf{Q}}) \atop x \in \lambda_1 + L_1} e \left( Q_1(x_{z_0^{\perp}})iT - Q_1(x_{z_0})iT \right)
\sum\limits_{m \in {\bf{Q}} \atop n + Q_1(x_{z_0^{\perp}}) + Q_1(x_{z_0}) + m=0 } b_{L_2}(\lambda_2, m, T) e(miT)\\
&=  \sum\limits_{ \lambda_1 \in L_1^{\vee} / L_1 \atop \lambda_2 \in L_2^{\vee}/L_2  } \sum\limits_{x \in V_1({\bf{Q}}) \atop x \in \lambda_1 + L_1} 
\sum\limits_{n \gg - \infty, m \in {\bf{Q}} \atop n + Q_1(x_{z_0^{\perp}}) + m= - Q_1(x_{z_0}) } 
c_{f_0}^+(\lambda_1 + \lambda_2, n) e \left( n iT + Q_1(x_{z_0^{\perp}})iT + miT - Q_1(x_{z_0})iT \right) b_{L_2}(\lambda_2, m, T)  \\
&= \sum\limits_{ \lambda_1 \in L_1^{\vee} / L_1 \atop \lambda_2 \in L_2^{\vee}/L_2  } \sum\limits_{x \in V_1({\bf{Q}}) \atop x \in \lambda_1 + L_1} 
\sum\limits_{n \gg - \infty, m \in {\bf{Q}} \atop n + Q_1(x_{z_0^{\perp}}) + m= - Q_1(x_{z_0}) } 
c_{f_0}^+(\lambda_1 + \lambda_2, -Q_1(x_{z_0^{\perp}}) - Q_1(x_{z_0}) - m) \\ &\times e \left( -2Q_1(x_{z_0}) iT \right)
b_{L_2} \left( \lambda_2, -n-Q_1(x_{z_0^{\perp}}) - Q_1(x_{z_0}), T \right) \\
&=  \sum\limits_{ \lambda_1 \in L_1^{\vee} / L_1 \atop \lambda_2 \in L_2^{\vee}/L_2  } \sum\limits_{x \in V_1({\bf{Q}}) \atop x \in \lambda_1 + L_1} 
\sum\limits_{n \gg - \infty, m \in {\bf{Q}} \atop n + m = -Q_1(x)} c_{f_0}^+(\lambda_1+ \lambda_2, -Q_1(x) - m) 
b_{L_2}(\lambda_2, -Q_1(x) - n, T) e \left( - 2Q_1(x_{z_0}) iT \right) \\
&= \sum\limits_{ \lambda_1 \in L_1^{\vee} / L_1 \atop \lambda_2 \in L_2^{\vee}/L_2  } \sum\limits_{x \in V_1({\bf{Q}}) \atop x \in \lambda_1 + L_1} 
\sum\limits_{m \in {\bf{Q}} \atop -m -Q_1(x) \gg - \infty} c_{f_0}^+(\lambda_1+ \lambda_2, -Q_1(x) - m) 
b_{L_2}(\lambda_2, m, T) e \left( - 2Q_1(x_{z_0}) iT \right). \end{aligned}\end{equation} 
Putting this together with the description of $A_0$ from Lemma \ref{4.5}, we find that 
\begin{equation}\begin{aligned}\label{limiting}& \lim_{T \rightarrow \infty} 
\left[ 2 \int_{\tau = i T}^{iT +1} \langle \langle f_0^+(\tau), \theta_{L_1}(\tau) \otimes E^{\prime}_{L_2}(\tau, 0; 2)\rangle \rangle d \tau -  A_0 \log(T) \right] \\ 
&=  \lim_{T \rightarrow \infty} 2 \sum\limits_{ \lambda_1 \in L_1^{\vee} / L_1 \atop \lambda_2 \in L_2^{\vee}/L_2  } \sum\limits_{x \in V_1({\bf{Q}}) \atop x \in \lambda_1 + L_1} 
\left( e \left( -2Q_1(x_{z_0}) iT \right) \sum\limits_{m \in {\bf{Q}} } c_{f_0}^+(\lambda_1+ \lambda_2, -Q_1(x) - m) b_{L_2}(\lambda_2, m, T)  -   c_{f_0}^+(\lambda_1 + \lambda_2, -Q_1(x))  \log(T)    \right)  \\
&=  \lim_{T \rightarrow \infty} 
2 \sum\limits_{ \lambda_1 \in L_1^{\vee} / L_1 \atop \lambda_2 \in L_2^{\vee}/L_2  } \sum\limits_{x \in V_1({\bf{Q}}) \atop x \in \lambda_1 + L_1} 
\left( e^{4 \pi Q_1(x_{z_0}) T} \sum\limits_{m \in {\bf{Q}} } c_{f_0}^+(\lambda_1+ \lambda_2, -Q_1(x) - m) b_{L_2}(\lambda_2, m, T)  -   c_{f_0}^+(\lambda_1 + \lambda_2, -Q_1(x))  \log(T)    \right).\end{aligned}\end{equation}
Observe that the $x=0$ contribution to $(\ref{limiting})$ can be calculated via the coefficient formula $(\ref{kappa})$ for the holomorphic 
part $\mathcal{E}_{L_2}(\tau) = E_{L_2}'(\tau,0; 2)$ of the derivative Eisenstein series described in $(\ref{EEFSE})$. More precisely,  
\begin{align*} &\lim_{T \rightarrow \infty} 
2 \sum\limits_{ \lambda_1 \in L_1^{\vee} / L_1 \atop \lambda_2 \in L_2^{\vee}/L_2  } 
\left( \sum\limits_{m \in {\bf{Q}} } c_{f_0}^+(\lambda_1+ \lambda_2, - m) b_{L_2}(\lambda_2, m, T)  -   c_{f_0}^+(\lambda_1 + \lambda_2, 0)  \log(T)    \right)  \\
&= 2 \sum\limits_{ \lambda_1 \in L_1^{\vee} / L_1 \atop \lambda_2 \in L_2^{\vee}/L_2  } 
\sum\limits_{m \in {\bf{Q}} } c_{f_0}^+(\lambda_1+\lambda_2, -m) \kappa_{L_2}(\lambda_2, m)
= 2 \operatorname{CT} \langle \langle f_0^+(\tau), {\bf{1}}_{L_1 + 0} \otimes \mathcal{E}_{L_2}(\tau) \rangle \rangle. \end{align*}
On the other hand, since $z_0 \in D(V_1)=\lbrace z \subset V_1: \dim(z) = 1, Q_1\vert_z < 0 \rbrace$ is negative by our conventions, 
we have $Q(x_{z_0}) <0$. Since the space $(V_1, Q_1)$ is anisotropic, we then see that the remaining contributions 
\begin{equation*}\begin{aligned} &\lim_{T \rightarrow \infty} 
2 \sum\limits_{ \lambda_1 \in L_1^{\vee} / L_1 \atop \lambda_2 \in L_2^{\vee}/L_2  } \sum\limits_{x  \neq 0 \in V_1({\bf{Q}}) \atop x \in \lambda_1 + L_1} 
\left( e^{4 \pi Q_1(x_{z_0}) T} \sum\limits_{m \in {\bf{Q}} } c_{f_0}^+(\lambda_1+ \lambda_2, -Q_1(x) - m) b_{L_2}(\lambda_2, m, T)  -   c_{f_0}^+(\lambda_1 + \lambda_2, -Q_1(x))  \log(T)    \right)  \\
&= \lim_{T \rightarrow \infty } 2 \sum\limits_{ \lambda_1 \in L_1^{\vee} / L_1 \atop \lambda_2 \in L_2^{\vee}/L_2  } 
\sum\limits_{x  \neq 0 \in V_1({\bf{Q}}) \atop x \in \lambda_1 + L_1} 
e^{4 \pi Q_1(x_{z_0}) T}  c_{f_0}^+(\lambda_1 + \lambda_2, -Q_1(x)) \left( b_{L_2}(\lambda_2, 0, T) - \log(T) \right)  \\ 
&+  \lim_{T \rightarrow \infty} 2 \sum\limits_{ \lambda_1 \in L_1^{\vee} / L_1 \atop \lambda_2 \in L_2^{\vee}/L_2  } 
\sum\limits_{x  \neq 0 \in V_1({\bf{Q}}) \atop x \in \lambda_1 + L_1}  e^{4 \pi Q_1(x_{z_0}) T} 
\sum\limits_{m \in {\bf{Q}} \atop m \neq 0 } c_{f_0}^+(\lambda_1+ \lambda_2, -Q_1(x) - m) b_{L_2}(\lambda_2, m, T) \end{aligned}\end{equation*}
to (\ref{limiting}) decay exponentially with $T$, and hence vanish in the limit with $T \rightarrow \infty$ (cf.~\cite[Proposition 2.11]{KuBL}). 
Taking the limit with $T \rightarrow \infty$ in $(\ref{prelim})$, we obtain the formula
\begin{equation*}\begin{aligned} \Phi( f_0, \mathcal{G})
&= -  \frac{2}{\operatorname{vol}(U_2)} \cdot \operatorname{CT} \langle \langle f_0^+(\tau), {\bf{1}}_{L_1 + 0} \otimes \mathcal{E}_{L_2}(\tau) \rangle \rangle \\
&- \frac{2}{\operatorname{vol}(U_2)} \left(   \langle \overline{\xi_0(f_0)}, \theta_{L_1} \otimes E'_{L_2}(\cdot, 0; 2) \rangle 
+  \langle f_0, \overline{ \xi_0(\theta_{L_1})} \otimes E'_{L_2}(\cdot, 0; 2) \rangle \right) \\
&=  - \frac{2}{\operatorname{vol}(U_2)} \left( \operatorname{CT} \langle \langle f_0^+(\tau), {\bf{1}}_{L_1 + 0} \otimes \mathcal{E}_{L_2}(\tau) \rangle \rangle 
+  L'(0,\xi_0(f_0) \times \theta_{L_1}) + I'(0, f_0 \times \xi_0(\theta_{L_1})) \right). \end{aligned}\end{equation*} \end{proof}

\subsection{Application to the central derivative value $\Lambda'(1/2, \Pi \otimes \chi)$} 

Recall that $\Pi = \operatorname{BC}_{K/{\bf{Q}}}(\pi)$ denotes the quadratic basechange lifting of the cuspidal automorphic representation $\pi = \otimes_v \pi_v$ of 
$\operatorname{GL}_2({\bf{A}})$ corresponding to an elliptic curve $E/{\bf{Q}}$ parametrized by a newform 
$f \in S_2^{\operatorname{new}}(\Gamma_0(N))$ to $\operatorname{GL}_2({\bf{A}}_K)$. 
By the theory of cyclic basechange, we have an equivalence of the $\operatorname{GL}_2({\bf{A}}_K) \times \operatorname{GL}_1({\bf{A}}_K)$-automorphic
$L$-function $\Lambda(s, \Pi \otimes \chi)$ with the $\operatorname{GL}_2({\bf{A}}) \times \operatorname{GL}_2({\bf{A}})$ Rankin-Selberg $L$-function
$\Lambda(s, \pi \times \pi(\chi)) = \Lambda(s, f \times \theta(\chi))$. 
We now consider the relation to Rankin-Selberg $L$-functions of vector-valued forms appearing in Theorem \ref{4.7}. Write
\begin{align*} f(\tau)  = f_E(\tau) = \sum\limits_{m \geq 1} c_f(m) e(m \tau) = \sum\limits_{m \geq 1} a_f(m) m^{\frac{1}{2}} e(m \tau)
\in S_2^{\operatorname{new}}(\Gamma_0(N)), \quad \tau = u + iv \in \mathfrak{H}. \end{align*}
Hence, the finite part $L(s, f)$ of the standard $L$-function $\Lambda(s, f) = \Lambda(s, \pi) = L(s, \pi_{\infty}) L(s, \pi)$
has the Dirichlet series expansion $L(s, f) = \sum_{m \geq 1} a_f(n) n^{-s} = \sum_{m \geq 1} c_f(n) n^{-(s + 1/2)}$ (first for $\Re(s)>1$). 
Recall that we fix a ring class character $\chi$ of some conductor $c \in {\bf{Z}}_{\geq 1}$ of $K$. 
Hence, $\chi = \otimes_x \chi_w$ is a character of the class group 
\begin{align*} \operatorname{Pic}(\mathcal{O}_c) 
&= {\bf{A}}_K^{\times}/  K_{\infty}^{\times} K^{\times} \widehat{\mathcal{O}}_c^{\times}, 
\quad \widehat{\mathcal{O}}_c^{\times} = \prod_{w < \infty} \mathcal{O}_{c, w}^{\times} \end{align*} 
of the ${\bf{Z}}$-order $\mathcal{O}_c = {\bf{Z}} + c \mathcal{O}_K$ of conductor $c$ in $K$. 
We consider the corresponding Hecke theta series defined by the twisted linear combination (see e.g.~\cite[(5.4)]{GZ})
\begin{align}\label{Hecketheta} \theta(\chi)(\tau) &= \sum\limits_{A \in \operatorname{Pic}(\mathcal{O}_c)} \chi(A) \theta_{A}(\tau),\end{align}
where each of the partial theta series $\theta_A(\tau)$ can be defined classically as follows. 
Let $w_K = \# \mu(K)$ denote the number of roots of unity $\mu(K)$ in $K$. 
Since the unit group $\mathcal{O}_K^{\times} \cong {\bf{Z}} \times \mu(K) = \langle \varepsilon_K \rangle \times \mu(K)$ is not torsion
by Dirichlet's unit theorem, we fix a fundamental domain $\mathfrak{a}^{\star} = [\alpha_{\mathfrak{a}}, z_{\mathfrak{a}}]^{\star}$ for the
action of $\mathcal{O}_K^{\times} / \mu(K) = \langle \varepsilon_K \rangle$ on $\mathfrak{a}$. The corresponding theta
series can then be described more explicitly via the expansion 
\begin{align*} \theta_A(\tau) &= \frac{1}{w_K} \sum\limits_{ \lambda \in \mathfrak{a}^{\star} } 
\sqrt{v} K_0 \left( 2 \pi \left\vert  \frac{ {\bf{N}}_{K/{\bf{Q}}}(\lambda) }{  {\bf{N}} \mathfrak{a} }  \right\vert v \right)
e \left(  \frac{ {\bf{N}}_{K/{\bf{Q}}}(\lambda) }{  {\bf{N}} \mathfrak{a} } \cdot u  \right) 
=  \sum\limits_{m \in {\bf{Z}}} r_A(m) \sqrt{v} K_0(2 \pi \vert m \vert v) e(m u),\end{align*}  
where $K_0$ denotes the $K$-Bessel function, and $r_A(m)$ denotes the corresponding counting function 
\begin{align*} r_A(m) &= \frac{1}{w_K} \cdot \# \left\lbrace \lambda \in \mathfrak{a}^{\star} = [\alpha_{\mathfrak{a}} ,z_{\mathfrak{a}}]^{\star}: 
\frac{ {\bf{N}}_{K/{\bf{Q}}}(\lambda) }{  {\bf{N}} \mathfrak{a} } =m \right\rbrace. \end{align*}
A classical theorem of Hecke shows that each $\theta(\chi)(\tau)$ is a nonholomorphic Maass form of weight zero, 
level $\Gamma_0(d_K)$ and character $\eta = \eta_K$. We consider the corresponding Rankin-Selberg presentation 
\begin{align*} \Lambda(s, \pi \times \pi(\chi)) 
&= \Lambda(s, f \times \theta(\chi)) = \sum\limits_{A \in \operatorname{Pic}(\mathcal{O}_c)} \chi(A) \Lambda(s, f \times \theta_A), \end{align*}
given as a twisted linear combination of the partial Rankin-Selberg $L$-functions 
(cf.~e.g.~\cite[$\S$ IV (0.1)]{GZ})\footnote{Observe that since $\theta_A(\tau)$ has weight zero, the arithmetic normalization
of the Rankin-Selberg $L$-function $L(2s, \eta) \sum_{m \geq 1} c_f(m) c_{\theta_A}(m) m^{- \left( s + \frac{2+0}{2} - 1 \right)}
= L(2s, \eta) \sum_{m \geq 1} c_f(m) c_{\theta_A}(m) m^{-s} = L(2s, \eta) \sum_{m \geq 1} c_f(m) r_A(m) m^{-s}$ 
coincides with the unitary normalization $L(2s, \eta) \sum_{m \geq 1} a_f(m) a_{\theta_A}(m) m^{-s} 
= L(2s, \eta) \sum_{m \geq 1}  c_f(m)m^{-\frac{1}{2}} c_{\theta_A}(m) m^{\frac{1}{2}} m^{-s}$.} 
\begin{equation}\begin{aligned}\label{classRS} \Lambda(s, f \times \theta_A) := \langle f, \theta_A E^{\star}(\cdot, s; 2) \rangle 
&= \frac{\Gamma(s)}{(4 \pi)^{s}} \cdot \Lambda(2s, \eta) \cdot \sum\limits_{m \geq 1} \frac{ c_f(m) r_A(m)}{ m^s}  \\
&=  \frac{\Gamma(s)}{(4 \pi)^{s}} \cdot \Lambda(2s, \eta) 
\cdot \frac{1}{w_K}\sum\limits_{ \lambda \in \mathfrak{a}^{\star} \atop [\mathfrak{a}] = A \in \operatorname{Pic}(\mathcal{O}_c) }
\frac{ c_f( {\bf{N}}(\lambda))}{ {\bf{N}}(\lambda)^{s}} \quad \quad (\Re(s) >1) \end{aligned}\end{equation}
associated to each class $A \in \operatorname{Pic}(\mathcal{O}_c)$.

Recall that Theorem \ref{YZhang} gives us a vector-valued lift $g = g_{f, A}$ of the eigenform $f$. 
We again consider for each class $A \in \operatorname{Pic}(\mathcal{O}_c)$ the corresponding quadratic space $(V_A, Q_A)$ described in
Definition \ref{V_A}, with vector space $V_A = \mathfrak{a}_{\bf{Q}} \oplus \mathfrak{a}_{\bf{Q}}$, and quadratic form 
$Q_A(z) = Q_A((z_1, z_2)) = Q_{\mathfrak{a}}(z_1) - Q_{\mathfrak{a}}(z_2)$.
As well, we consider the anisotropic subspaces $(V_{A, j}, Q_{A, j})$ of signature (1,1) defined by 
$V_{A, 1} = \mathfrak{a}_{\bf{Q}}$ with $Q_{A, 1} = - Q_{\mathfrak{a}}$ and $V_{A, 2} = \mathfrak{a}_{\bf{Q}}$ 
with $Q_{A, 2} = Q_{\mathfrak{a}}$. 
Recall we write $L_A \subset V_A$ for the lattice determined by the compact open
subgroup $U_A \subset \operatorname{GSpin}(V_A)({\bf{A}}_f)$ described in Corollary \ref{lattices}.
We write $L_{A,j}:= L_{A} \cap V_{A, j}$ for each of $j=1, 2$ to denote 
the signature (1,1) sublattice determined by restriction to $V_{A,j}$. 
By Theorem \ref{YZhang}, we can associate to $f \in S_2^{\operatorname{new}}(\Gamma_0(N))$
an $\mathcal{S}_{L_A}$-valued modular form $g = g_{f, A}$ of weight $2$. 
Recall as well that we consider the (incomplete, partial) Rankin-Selberg $L$-functions given by the Petersson inner products
\begin{align*} L(s, g \times \theta_{L_{A,1}}) &:= \langle g(\cdot), \theta_{ L_{A,1} }(\cdot) \otimes E_{L_{A,2}}(\cdot, s; 2)  \rangle
=  \langle g(\tau), \theta_{ L_1 }(\tau) \otimes E_{L_{A,2}}(\tau, s; 2)  \rangle. \end{align*}
We also consider the completed version, given with respect to the completed Eisenstein series $E_{L_{A,2}}^{\star}(\tau, s; 2)$: 
\begin{align*} L^{\star}(s, g \times \theta_{L_{A, 1}}) &:= \langle g(\cdot), \theta_{ L_{A,1} }(\cdot) \otimes E^{\star}_{L_{A,2}}(\cdot, s; 2)  \rangle
=  \langle g(\tau), \theta_{ L_{A,1} }(\tau) \otimes E^{\star}_{L_{A,2}}(\tau, s; 2)  \rangle. \end{align*}

\begin{corollary}\label{vectorvalue} 

We have the equivalences of $L$-functions 
\begin{align*} \Lambda(s -1/2, \Pi \otimes \chi) 
&= \sum\limits_{A \in \operatorname{Pic}(\mathcal{O}_c)} \chi(A) \Lambda(s-1/2, f \otimes \eta \times \theta_A)
= \frac{1}{2} \cdot \sum\limits_{A \in \operatorname{Pic}(\mathcal{O}_c)} \chi(A) L^{\star}(2s-2, g_{f, A} \times \theta_{L_{A, 1}}). \end{align*}
Here, we note that the shifts are aligned; the basechange $L$-function $\Lambda(s, \Pi \otimes \chi)$ has central value at $s=1/2$,
while the completed vector-valued Rankin-Selberg $L$-function $L^{\star}(s, g_{f, A} \times \theta_{L_{A,1}})$ has central value at $s=0$
(with the central value occurring at $s=1$ here). Hence, we have the relation of central derivative values 
\begin{align*} \Lambda'(1/2, \Pi \otimes \chi) 
&= \sum\limits_{A \in \operatorname{Pic}(\mathcal{O}_c)} \chi(A) \Lambda'(1/2, f \otimes \eta \times \theta_A)
= \frac{1}{2} \cdot \sum\limits_{A \in \operatorname{Pic}(\mathcal{O}_c)} \chi(A) L^{\star \prime}(0, g_{f, A} \times \theta_{L_{A, 1}}). \end{align*} \end{corollary}

\begin{proof} 

In the same way as for \cite[$\S 4$, (4.24)]{BY} (with Fourier coefficient notations as described above), 
each partial Rankin-Selberg product $L(s, g_{f, A} \times \theta_{A, 1})$ has the Dirichlet series expansion 
\begin{equation*}\begin{aligned} L(s, g \times \theta_{L_{A, 1}}) &= \frac{ \Gamma \left( \frac{s +2}{2} \right)  }{ (4 \pi)^{\frac{s+2}{2}}} 
\sum\limits_{ \mu \in L_{A,1}^{\vee}/ L_{A,1} } \sum\limits_{ m \in {\bf{Q}}_{>0} } \frac{ c_{g_{f, A}}(\mu, m) 
r_{ L_{A,1}}(\mu, m)}{ m^{  \frac{s + 2}{2}  }}, \end{aligned}\end{equation*} 
where $r_{ L_{A,1}}(\mu, m)$ denotes the counting function 
\begin{align*} r_{ L_{A,1}}(\mu, m) 
&= \frac{1}{w_K} \cdot \# \left\lbrace \lambda \in \mu + L_{A,1}: Q_{A, 1}(\lambda) = m \right\rbrace / \langle \varepsilon_K \rangle. \end{align*}
Here again, we fix a fundamental domain for the action of the fundamental unit $\langle \varepsilon_K \rangle \cong \mathcal{O}_K^{\times}/\mu(K)$. 
Now, since $[N^{-1} \mathfrak{a}] = [(N^{-1}) \mathfrak{a}] = [\mathfrak{a}] \in C(\mathcal{O}_K) = I(K)/P(K)$,
we see that the lattice $L_{A,1} = N^{-1} \mathfrak{a}$ also forms an ideal representative for the class of $A = [\mathfrak{a}]$, 
and $Q_{A,1}(x, y)$ is a binary quadratic form representative. Hence,
$r_{L_{A,1}}(\mu, m)$ counts the number of ideals in $\mu + \mathfrak{a}^{\star}$ of norm $m$. 
It then follows that we can identify the partial Rankin-Selberg $L$-function
$L(s, g_{f, A} \times \theta_{L_{A, 1}})$ with the classical partial Rankin-Selberg $L$-function $L(s, f \times \theta_A)$, as we can expand 
\begin{equation*}\begin{aligned} L(s, g_{f, A} \times \theta_{L_{A, 1}}) 
&= \frac{ \Gamma \left( \frac{s +2}{2} \right)  }{ (4 \pi)^{\frac{s+2}{2}}} \cdot \frac{1}{w_K} \sum\limits_{ \mu \in L_{A,1}^{\vee}/ L_{A,1} } 
\sum\limits_{ \lambda \in \mu + \mathfrak{a}^{\star} } \frac{ c_{g_{f, A}}(\mu, Q_{A,1}(\lambda)) }{ Q_{A,1}(\lambda)^{  \frac{s + 2}{2}  }} \\ 
&=  \frac{ \Gamma \left( \frac{s +2}{2} \right)  }{ (4 \pi)^{\frac{s+2}{2}}} \cdot \frac{2}{w_K}
\sum\limits_{\lambda \in \mathfrak{a}^{\star}} \frac{c_{f}( {\bf{N}}(\lambda) ) }{ {\bf{N}}(\lambda)^{\frac{s+2}{2}}  }
=  \frac{ \Gamma \left( \frac{s +2}{2} \right)  }{ (4 \pi)^{ \frac{s+2}{2}  }   } \cdot \frac{2}{w_K} \sum\limits_{\lambda \in \mathfrak{a}^{\star}} 
\frac{ c_f ( {\bf{N}}(\lambda) ) }{ {\bf{N}}(\lambda)^{ \frac{s+2}{2} }  }. \end{aligned}\end{equation*}
Here, we use the relation of coefficients described in Theorem \ref{YZhang} and the Dirichlet 
series expansion is taken over rational integers $m \geq 1$ coprime to $N$.
We then deduce that we have for each class $A \in \operatorname{Pic}(\mathcal{O}_c)$ the relation 
$L^{\star}(2s-2, g_{f, A} \times \theta_{L_{A, 1}}) = 2 \Lambda(s, f \times \theta_A)$ (cf.~\cite[$\S$ IV (0.1), p.~271]{GZ}). 
The stated relations follow, with the analytic continuation and functional equations determined by the underlying Eisenstein series. \end{proof}

\begin{theorem}[Twisted linear combinations of regularized theta integrals]\label{MAIN}

Let $f = f_E \in \operatorname{S}_2^{\operatorname{new}}(\Gamma_0(N))$ be the cuspidal eigenform 
parametrizing an elliptic curve $E/{\bf{Q}}$, with $\pi = \pi(f)$ the corresponding cuspidal automorphic representation of $\operatorname{GL}_2({\bf{A}})$,
and $\Pi = \operatorname{BC}_{K/{\bf{Q}}}(\pi)$ its quadratic basechange lifting to $\operatorname{GL}_2({\bf{A}}_K)$. 
Let $\chi$ be any ring class character of the real quadratic field $K$ of conductor $c$ coprime to $d_K N$. 
For each class $A \in \operatorname{Pic}(\mathcal{O}_c)$, let $f_{0, A} \in H_{0}(\omega_{L_A})$ 
denote the harmonic weak Maass form of weight zero with image $\xi_0(f_{0, A}) = g_{f, A} \in S_{2}(\overline{\omega}_{L_A})$ 
where $g_{f, A}$ denotes the vector-valued lifting of $f \in S_2^{\operatorname{new}}(\Gamma_0(N))$
the space of vector-valued forms $S_{2}(\overline{\omega}_{L_A})$ as described in Theorem \ref{YZhang}. Then, 
\begin{align*} &\frac{\Lambda'(1/2, \Pi \otimes \chi)}{\Lambda(1, \eta)} \\ 
&=  - \sum\limits_{A \in \operatorname{Pic}(\mathcal{O}_c)} \chi(A) 
\left(   \operatorname{CT} \langle \langle f_{0, A}^+(\tau), {\bf{1}}_{L_{A,1} \oplus 0} \otimes \mathcal{E}_{L_{A,2}}(\tau) \rangle \rangle
+I'(0, f_{0, A} \times \xi_0(\theta_{L_{A,1}}))
+  \frac{\operatorname{vol}(U_{A, 2})}{2} \cdot \Phi( f_{0,A}, \mathcal{G}_A) \right). \end{align*}
Here, for each class $A \in \operatorname{Pic}(\mathcal{O}_c)$, we write 
$U_{A, 2} := U \cap \operatorname{GSpin}_{V_{A, 2}}({\bf{A}}_f)$ as in Lemma \ref{4.4} above.

\end{theorem} 

\begin{proof} 

Formally, this is a consequence of Corollary \ref{vectorvalue}, after applying Theorem \ref{4.7} to each of the partial Rankin-Selberg $L$-series 
$L(s, g_{f, A} \times \theta_{A,1}) = L(s, \xi_0(f_{0, A}) \times \theta_{A, 1})$. In this way, we obtain the relation
\begin{equation*}\begin{aligned} 
&\sum\limits_{A \in \operatorname{Pic}(\mathcal{O}_c)} \chi(A) \cdot \frac{\operatorname{vol}(U_{A, 2})}{2} \cdot \Phi( f_{0, A}, \mathcal{G}_A) \\
&= -\sum\limits_{A \in \operatorname{Pic}(\mathcal{O}_c)} \chi(A)  \cdot \left( 
\operatorname{CT}  \langle \langle f_{0, A}^+(\tau),  {\bf{1}}_{L_{A,1} \oplus 0} \otimes \mathcal{E}_{L_{A,2}}(\tau) \rangle \rangle 
+ L'(0, \xi_0(f_{0, A}) \times \theta_{A, 1}) + I'(0, f_{0, A} \times \xi_0(\theta_{L_{A,1}})) \right). \end{aligned}\end{equation*}
We then identify the second term $\sum_A \chi(A) L^{\prime}(0, \xi_0(f_{0,A}) \times \theta_{A,1})$ in this expression with central derivative value 
$L'(1/2, \Pi \otimes \chi) = \Lambda'(1/2, \Pi \otimes \chi)/\Lambda(1, \eta)$ via Corollary \ref{vectorvalue} to obtain the stated relation. \end{proof}

Writing $d_K$ again to denote the fundamental discriminant associated to $K = {\bf{Q}}(\sqrt{d})$, let $h_K = \# \operatorname{Pic}(\mathcal{O}_K)$ denote the class number, 
and $\varepsilon_K = \frac{1}{2}(t + u \sqrt{d_K})$ for the smallest solution $t, u >0$ (with $u$ minimal) to Pell's equation $t^2 - d_K u^2 = 4$.
We can then express the formula derived above for the central derivative value $L'(1/2, \Pi \otimes \chi)$ in terms of Dirichlet's analytic class number formula
\begin{align}\label{Dirichlet} L(1, \eta) &= \frac{2 \log \varepsilon_K \cdot h_K}{\sqrt{d_K}} \quad \implies \quad 
\Lambda(1, \eta) = d_K^{\frac{1}{2}} \Gamma_{\bf{R}}(1) L(1, \eta) = 2 h_K \log \varepsilon_K. \end{align}

\begin{corollary}\label{DCNF} 

We have that 
\begin{align*} &\Lambda'(1/2, \Pi \otimes \chi) = \Lambda'(1/2, \pi \times \pi(\chi)) = \Lambda'(1/2, f \times \theta(\chi)) = \Lambda'(E/K, \chi, 1)
\\ &= - 2 h_K \log \varepsilon_K  \sum\limits_{A \in \operatorname{Pic}(\mathcal{O}_c)} \chi(A) 
\left(  \operatorname{CT} \langle \langle f_{0, A}^+, {\bf{1}}_{L_{A,1} \oplus 0} \otimes \mathcal{E}_{L_{A,2}} \rangle \rangle
+I'(0, f_{0, A} \times \xi_0(\theta_{L_{A,1}})) +  \frac{\operatorname{vol}(U_{A, 2})}{2} \cdot \Phi( f_{0,A}, \mathcal{G}_A) \right).\end{align*}
Moreover, if we assume Hypothesis \ref{EHH} that the inert level $N^-$ is the squarefree product of an odd number of primes, 
then this central derivative value is not forced by the functional equation $(\ref{symmFE})$ to vanish identically. \end{corollary}

\begin{proof} This simply restates Theorem \ref{MAIN} in terms of the Dirichlet analytic class number formula $(\ref{Dirichlet})$. \end{proof}

\section{Relation to the conjecture of Birch and Swinnerton-Dyer}

Let us now consider Theorem \ref{MAIN} from the point of view of the refined conjecture of Birch and Swinnerton-Dyer,
comparing with the Gross-Zagier formula \cite{GZ}.
To date, there is no known or conjectural construction of points on the corresponding elliptic curve $E(K[c])$ or
modular curve $X_0(N)(K[c])$ analogous to Heegner points\footnote{There is however a $p$-adic construction due to Darmon \cite{Da}.}, 
where $K[c]$ denotes the ring class extension of conductor $c$ of the real quadratic field $K$.
We can consider the implications for arithmetic terms in the refined Birch and 
Swinnerton-Dyer formula for $L^{\star \prime}(E/K, \chi, 1)$ here, in the style of the comparison given in Popa \cite[$\S 6.4$]{Po}.
Taking for granted the refined conjecture of Birch and Swinnerton-Dyer for $E(K[c]))$ in this setting, we can then find ``automorphic" 
interpretations of the corresponding Tate-Shafarevich group $\Sha(E/K[c])$ and regulator $\operatorname{Reg}(E/K[c])$. 
We also derive an unconditional result in special cases to illustrate surprising connections here.

Again, we fix $\chi$ a primitive ring class character of some conductor $c \geq 1$ prime to $d_K N$, 
and view this as a character of the class group $\operatorname{Pic}(\mathcal{O}_c)$. 
Recall that the reciprocity map of class field theory gives us an isomorphism 
$\operatorname{Pic}(\mathcal{O}_c) := {\bf{A}}_K^{\times} /{\bf{A}}^{\times}  K_{\infty}^{\times} K^{\times} \widehat{\mathcal{O}}_c^{\times}
\longrightarrow \operatorname{Gal}(K[c]/K)$,
where $K[c]$ is (by definition) the ring class extension of conductor $c$ of $K$. Recall as well
that by the theory of cyclic basechange of \cite{La} and more generally \cite{AC} with Artin formalism,
we can write the completed Hasse-Weil $L$-function $\Lambda(E/K[c], s)$ of $E$ basechanged to $K[c]/K$ as the product 
\begin{equation}\begin{aligned}\label{AF}
\Lambda(E/K[c] , s) &= \prod\limits_{\chi \in \operatorname{Pic}(\mathcal{O}_c)^{\vee} \cong \operatorname{Gal}(K[c]/K)^{\vee}} \Lambda(E/K, \chi, s) 
= \prod\limits_{\chi \in \operatorname{Pic}(\mathcal{O}_c)^{\vee} \cong \operatorname{Gal}(K[c]/K)^{\vee}} \Lambda(s-1/2, f \times \theta(\chi)). \end{aligned}\end{equation}
It then follows as a formal consequence of $(\ref{AF})$ that we have the relation(s) 
\begin{equation}\begin{aligned}\label{AF2}
\operatorname{ord}_{s=1} \Lambda(E/ K[c], s) 
&= \sum\limits_{\chi \in  \operatorname{Pic}(\mathcal{O}_c)^{\vee} \cong \operatorname{Gal}(K[c]/K)^{\vee}} 
\operatorname{ord}_{s=1/2} \Lambda (s, \Pi \otimes \chi), \end{aligned}\end{equation}
so that the conjecture of Birch and Swinnerton-Dyer predicts the rank equivalence 
\begin{equation}\begin{aligned}\label{AF3} \operatorname{rk}_{\bf{Z}} E(K[c]) = \operatorname{ord}_{s=1} \Lambda(E/ K[c], s) 
&= \sum\limits_{\chi \in  \operatorname{Pic}(\mathcal{O}_c)^{\vee} \cong \operatorname{Gal}(K[c]/K)^{\vee}} 
\operatorname{ord}_{s=1/2} \Lambda(s, \Pi \otimes \chi). \end{aligned}\end{equation}

Let us now assume Hypothesis \ref{EHH}, so that for each ring class character $\chi$ on the right hand side of $(\ref{AF3})$,
we know by the symmetric functional equation $(\ref{symmFE})$ that $\operatorname{ord}_{s=1/2}\Lambda(s, \Pi \otimes \chi) \geq 1$. 
Let us also assume for the moment that the rank equality predicted by the conjecture of Birch and Swinnerton-Dyer holds, so that 
\begin{align}\label{B+S-D} \operatorname{rk}_{\bf{Z}} E(K[c]) = h(\mathcal{O}_c) 
:= \# \operatorname{Pic}(\mathcal{O}_c) = \# \operatorname{Gal}(K[c]/K). \end{align} 
Let $r_E(K[c])$ denote the Mordell-Weil rank of $E$ over the ring class extension $K[c]$ of conductor $c$ over $K$.
The refined conjecture of Birch and Swinnerton-Dyer predicts that the leading term in the Taylor series expansion 
around $\Lambda^{(r_E(K[c]))}(E/K[c], s)/(r_E(K[c]))!$ around $s=1$ is given by the following formula. 
Let $\Sha_E(K[c])$ denote the Tate-Shafarevich group of $E$ over $K[c]$, 
\begin{align*} \Sha_E(K[c]) &= \ker \left( H^1(K, E ) \longrightarrow \prod_w H^1(K_w, E) \right), \end{align*}
which we shall assume is known to be finite. Let $R_E(K[c])$ denote the regulator of $E$ over $K[c]$.
Hence, fixing a basis $(e_j)_{j=1}^{r_E(K[c])}$ of $E(K[c])/E(K[c])_{\operatorname{tors}}$, and writing $\left[ \cdot, \cdot \right]$
to denote the N\'eron-Tate height pairing, 
\begin{align*} R_E(K[c]) &= \det \left( \left[ e_i, e_j \right] \right)_{i, j}. \end{align*}
Let us also write $T_E(K[c])$ to denote the product over local Tamagawa factors, so 
\begin{align*} T_E(K[c]) &= \prod_{\nu < \infty \atop \operatorname{primes ~ of }  \mathcal{O}_{K[c]} } [E(K[c]_{\nu}) : E_0(K[c]_{\nu}) ] 
\cdot \left\vert \frac{\omega}{ \omega_{\nu}^*} \right\vert_{\nu}, \end{align*}
where $\omega = \omega_E$ is a fixed invariant differential for $E/K[c]$, and each $\omega_{\nu}^*$ the N\'eron differential at $\nu$. 
The refined conjecture of Birch and Swinnerton-Dyer then predicts that the leading term in the Taylor series expansion of 
$\Lambda^{(r_E(K[c])))}(E/K[c], s)/(r_E(K[c]))!$ around $s=1$ is given by the formula 
\begin{equation}\begin{aligned}\label{refinedB+S-D} 
\frac{ \# \Sha_E(K[c]) \cdot R_E(K[c]) \cdot T_E(K[c])}{ \sqrt{\vert d_{K[c]}\vert } \cdot \# E(K[c])_{\operatorname{tors}}^2} \cdot 
\prod\limits_{ {\mu \mid \infty \atop \mu: K[c] \rightarrow {\bf{R}} } \atop \operatorname{ real~places} } \int_{ E(K[c]_{\mu})} \vert \omega \vert \cdot 
\prod\limits_{ { \sigma \mid \infty \atop \sigma, \overline{\sigma}: K[c] \rightarrow {\bf{C}}} \atop \operatorname{ pairs~of~complex~places}} 
2 \int\limits_{E(K[c]_{\sigma})} \omega \wedge \overline{\omega}. \end{aligned}\end{equation}

Let us first assume for simplicity that the class number is one: $h(\mathcal{O}_c) = h_K = 1$, and hence that $K[1]=K$. 
Then, assuming the conjecture of Birch and Swinnerton-Dyer $(\ref{B+S-D})$ and $(\ref{refinedB+S-D})$, we 
derive via Theorem \ref{MAIN} and Corollary \ref{Green} the (conditional) identifications 
\begin{equation*}\begin{aligned} &\Lambda'(E/K, 1) = \Lambda'(1/2, \Pi) = \Lambda'(1/2, \Pi) 
= \frac{ \# \Sha_E(K) \cdot R_E(K) \cdot T_E(K)}{ \sqrt{d_K} \cdot \# E(K)_{\operatorname{tors}}^2} \cdot 
\prod\limits_{ \mu \mid \infty  \atop \mu: K \rightarrow {\bf{R}}} \int_{ E(K_{\mu})} \vert \omega \vert \\
&= - 2 \log \varepsilon_K 
\left(   \operatorname{CT} \langle \langle f^+_{0, \mathcal{O}_K}, {\bf{1}}_{ L_{ \mathcal{O}_K, 1} \oplus 0} \otimes \mathcal{E}_{L_{ \mathcal{O}_K , 2}} \rangle \rangle  
+ I'(0, f^+_{0, \mathcal{O}_K} \times \xi_0(\theta_{L_{\mathcal{O}_K, 1}})) 
+ \frac{\operatorname{vol}(U_{\mathcal{O}_K, 2})}{2} \cdot \Phi(f_{0, \mathcal{O}_K}, \mathcal{G}_{\mathcal{O}_K}) \right). \end{aligned}\end{equation*}
This suggests that the regulator $R_E(K) = [e, e]$ should be given by the formula 
\begin{equation}\begin{aligned}\label{REG} &R_E(K) = [e, e] \\ 
&= - \frac{ 2 \log \varepsilon_K  \sqrt{d_K}  \#E(K)_{\operatorname{tors}}^2 \left(   
\operatorname{CT} \langle \langle f^+_{0, \mathcal{O}_K}, {\bf{1}}_{L_{\mathcal{O}_K, 1} \oplus 0} 
\otimes \mathcal{E}_{L_{\mathcal{O}_K, 2}} \rangle \rangle + I'(0, f^+_{0, \mathcal{O}_K} \times \xi_0(\theta_{L_{\mathcal{O}_K, 1}}))
+ \frac{\operatorname{vol}(U_{\mathcal{O}_K, 2})}{2} \cdot \Phi(f_{0, \mathcal{O}_K}, \mathcal{G}_{\mathcal{O}_K})
 \right) }{  \# \Sha_E(K) \cdot T_E(K) \cdot 
\prod\limits_{ \mu \mid \infty  \atop \mu: K \rightarrow {\bf{R}}} \int_{ E(K_{\mu})} \vert \omega \vert   }. \end{aligned}\end{equation}
Similarly, the cardinality $\# \Sha_E(K)$ of Tate-Shafarevich group $\Sha_E(K)$ should be given by the formula 
\begin{equation}\begin{aligned}\label{SHA} &\# \Sha_E(K) \\ 
&= - \frac{ 2 \log \varepsilon_K \sqrt{d_K}  \#E(K)_{\operatorname{tors}}^2 \left(  
\operatorname{CT} \langle \langle f^+_{0, \mathcal{O}_K }, {\bf{1}}_{ L_{\mathcal{O}_K, 1} \oplus 0} 
\otimes \mathcal{E}_{ L_{\mathcal{O}_K, 2} } \rangle \rangle 
+ I'(0, f^+_{0, \mathcal{O}_K} \times \xi_0(\theta_{L_{\mathcal{O}_K, 1}}))
+ \frac{\operatorname{vol}(U_{\mathcal{O}_K, 2})}{2} \cdot \Phi(f_{0, \mathcal{O}_K}, \mathcal{G}_{\mathcal{O}_K}) \right) }{ R_E(K) \cdot T_E(K) \cdot 
\prod\limits_{ \mu \mid \infty  \atop \mu: K \rightarrow {\bf{R}}} \int_{ E(K_{\mu})} \vert \omega \vert   }. \end{aligned}\end{equation}
Note that we can also derive similar, conditional arithmetic expressions for $\# \Sha_E(K[c])$
and $R_E(K[c])$ in the more general setting where $h_K \geq 1$, e.g.~after specializing our main result to the principal
character $\chi = \chi_0$ of the class group of $K$, and summing over classes. 
Finally, we establish the following unconditional result. 

\begin{theorem}\label{URBSD} 

Assume that $\operatorname{ord}_{s=1} \Lambda(E/K, 1) = 1$, so that either $\Lambda(E, 1) = \Lambda(1/2, \pi)$ 
or the quadratic twist $\Lambda(E^{(d_K)}, 1) = \Lambda(1/2, \pi \otimes \eta)$ vanishes. 
Let us also assume that $E$ has semistable reduction so that its conductor $N$ is squarefree, 
with $N$ coprime to the discriminant $d_K$ of $K$, and for each prime $p \geq 5$: \\
\begin{itemize}
\item The residual Galois representations $E[p]$ and $E^{(d_K)}[p]$ attached to $E$ and $E^{(d_K)}$ are irreducible. \\
\item There exists a prime divisor $l \mid \mid N$ distinct from $p$ where the residual representation $E[p]$ is ramified,
and a prime divisor $q \mid \mid N d_K$ distinct from $p$ where the residual representation $E^{(d_K)}[p]$ is ramified. \\ \end{itemize}
Writing $[e, e]$ to denote the regulator of  either $E$ or $E^{(d_K)}$ according to which factor vanishes, 
we have the following unconditional identity, up to powers of $2$ and $3$:  \\
\begin{align*} &\frac{ \# \Sha_E({\bf{Q}}) \cdot \# \Sha_{ E^{(d_K)}}( {\bf{Q}}) \cdot [e, e] \cdot T_E({\bf{Q}}) \cdot T_{ E^{(d_K)} }({\bf{Q}})  }
{ \# E({\bf{Q}})_{\operatorname{tors}}^2 \cdot  \# E^{(d_K)}( {\bf{Q}} )_{\operatorname{tors}}^2  } \cdot \int_{ E({\bf{R}}) } \vert \omega_E \vert
\cdot \int_{ E^{(d_K)}({\bf{R}}) } \vert \omega_{E^{(d_K)}} \vert \\
&= - 2 h_K \log \varepsilon_K \cdot  \sum\limits_{A \in \operatorname{Pic}(\mathcal{O}_K)} 
\left( \operatorname{CT} \left( \langle \langle f^+_{0, A}, {\bf{1}}_{L_{A,1} \oplus 0} \otimes \mathcal{E}_{L_{A,2}} \rangle \rangle \right)
+I'(0, f_{0,A} \times \xi_0(\theta_{L_{A,1}}))
+ \frac{\operatorname{vol}(U_{A, 2})}{2} \cdot \Phi(f_{0,A}, \mathcal{G}_A) \right). \end{align*} \end{theorem}

\begin{proof} 

Assuming as we do that $\operatorname{ord}_{s=1} \Lambda(E/K, 1) = 1$, we deduce from the Artin formalism that 
\begin{align*} \Lambda'(E/K, 1) &= \Lambda'(E, 1) \Lambda(E^{(d_K)}, 1) +  \Lambda' (E^{(d_K)}, 1) \Lambda(E, 1), \end{align*}
or equivalently that 
\begin{align*} \Lambda'(1/2, \Pi) = \Lambda'(1/2, \pi) \Lambda(1/2, \pi \otimes \eta) 
+ \Lambda'(1/2, \pi \otimes \eta) \Lambda(1/2, \pi), \end{align*}
where precisely one of the summands on the right-hand side in each version does not vanish. 
Note that we can take for granted the refined conjecture of Birch and Swinnerton-Dyer $(\ref{refinedB+S-D})$ 
for the nonvanishing summand up to powers of $2$ and $3$ by our hypotheses, using the combined works of Kato \cite{Ka}, Kolyvagin \cite{Ko}, 
Rohrlich \cite{Ro}, and Skinner-Urban \cite{SU} with the corresponding Euler characteristic calculations of 
Burungale-Skinner-Tian \cite{BST} (cf.~\cite{BST}, \cite{Ca}) for the 
analytic rank zero part, together with Jetchev-Skinner-Wan \cite{JSW}, Skinner-Zhang \cite{SZ}, and Zhang \cite{WZ} for the analytic rank one part.
We refer to the summary given in \cite[Theorem 3.10]{BST} for the current status of these deductions, confirming
the $p$-part of the conjectural Birch-Swinnerton-Dyer formula via Iwasawa-Greenberg main conjectures.
Applying $(\ref{refinedB+S-D})$ to each factor, we can then deduce (up to powers of $2$ and $3$) that we have the refined product formula  
\begin{align*} &\Lambda'(E/K, 1) = \Lambda'(1/2, \Pi) \\
&= \frac{ \# \Sha_E({\bf{Q}}) \cdot \# \Sha_{ E^{(d_K)}}( {\bf{Q}}) \cdot [e, e] \cdot T_E({\bf{Q}}) \cdot T_{ E^{(d_K)} }({\bf{Q}})  }
{ \# E({\bf{Q}})_{\operatorname{tors}}^2 \cdot  \# E^{(d_K)}( {\bf{Q}} )_{\operatorname{tors}}^2  } \cdot \int_{ E({\bf{R}}) } \vert \omega_E \vert
\cdot \int_{ E^{(d_K)}({\bf{R}}) } \vert \omega_{E^{(d_K)}} \vert.  \end{align*}
The stated identity then follows from Theorem \ref{MAIN} and Corollary \ref{Green}. \end{proof}

\end{document}